\documentclass[12pt, reqno]{amsart}

\usepackage{booktabs,array}
\usepackage{url}
\usepackage{hyperref}
\usepackage{pdfsync}
\usepackage{amssymb}
\usepackage{amsmath}
\usepackage{amsthm}
\usepackage{stmaryrd}
\usepackage{fullpage}	% smaller margins %incompatible with ucthesis
\usepackage{amscd}   	% for commutative diagrams
\usepackage[all]{xy} 		% for complicated commutative diagrams
\usepackage{comment} 	% for \begin{comment} ... \end{comment}
\usepackage{mathrsfs}      % for \mathscr
\usepackage{placeins}
\usepackage{enumerate}
\usepackage{mathtools}
\usepackage{colonequals}
\usepackage{xspace}
\usepackage{nicefrac}
\usepackage{booktabs}
\usepackage{caption}

\usepackage{tikz}
\usetikzlibrary{decorations.pathmorphing}

\usepackage[alphabetic,backrefs,lite,msc-links]{amsrefs} % for bibliography

\usepackage{color}

\DeclareMathSymbol{\minus}{\mathbin}{AMSa}{"39}

\newcommand{\marg}[1]{}

\newcommand{\note}[1]{} 				

\newcommand{\defi}[1]{\textsf{#1}} 	% for defined terms

\DeclareMathOperator{\ns}{ns}

\DeclareMathOperator{\spc}{sp}

\newcommand{\C}{{\mathbb C}}
\newcommand{\F}{{\mathbb F}}

\newcommand{\Q}{{\mathbb Q}}

\newcommand{\Z}{{\mathbb Z}}

\newcommand{\Qbar}{{\overline{\Q}}}
\newcommand{\Zhat}{{\widehat{\Z}}}
\newcommand{\Qhat}{{\widehat{\Q}}}

\newcommand{\kbar}{{\overline{k}}}
\newcommand{\Kbar}{{\overline{K}}}

\newcommand{\pp}{{\mathfrak p}}

\newcommand{\calC}{{\mathcal C}}

\newcommand{\calN}{{\mathcal N}}
\newcommand{\calO}{{\mathcal O}}

\newcommand{\calS}{{\mathcal S}}

\newcommand{\scrE}{{\mathscr E}}

\def\Q{\mathbb{Q}}
\def\C{\mathbb{C}}

\def\P{\mathbb{P}}
\def\A{\mathbb{A}}
\def\Z{\mathbb{Z}}

\DeclareMathOperator{\sfr}{sf} 
\DeclareMathOperator{\tr}{tr}

\DeclareMathOperator{\rk}{rk}

\DeclareMathOperator{\im}{im}
 
\DeclareMathOperator{\End}{End}
\DeclareMathOperator{\Jac}{Jac}

\DeclareMathOperator{\Aut}{Aut}
\DeclareMathOperator{\Gal}{Gal} 
 \DeclareMathOperator{\Res}{Res}

 \DeclareMathOperator{\Pic}{Pic}
 
\DeclareMathOperator{\Spec}{Spec}

\DeclareMathOperator{\cond}{cond}
\DeclareMathOperator{\lcm}{lcm}

\newcommand{\an}{{\operatorname{an}}}

\newcommand{\tors}{{\operatorname{tors}}}
\newcommand{\red}{{\operatorname{red}}}

\newcommand{\GL}{\operatorname{GL}}
\newcommand{\PGL}{\operatorname{PGL}}

\newcommand{\SL}{\operatorname{SL}}

\newcommand{\PSL}{\operatorname{PSL}}

\newcommand{\magma}{\texttt{Magma}\xspace}
\newcommand{\magmacmd}[1]{\texttt{#1}}
\newcommand{\smallmat}[4]{\left[\begin{smallmatrix} #1 & #2 \\ #3 & #4 \end{smallmatrix}\right]}
\newcommand{\eclabel}[1]{\href{https://www.lmfdb.org/EllipticCurve/Q/#1}{\texttt{#1}}}
\newcommand{\ecnflabel}[2]{\href{https://www.lmfdb.org/EllipticCurve/#1-#2}{\texttt{#2}}}
\newcommand{\mflabel}[1]{\href{https://www.lmfdb.org/ModularForm/GL2/Q/holomorphic/#1}{\texttt{#1}}}

\newcommand{\repolink}[2]{\href{https://github.com/AndrewVSutherland/ell-adic-galois-images/blob/main/#1/#2}{\texttt{#2}}}
\newcommand{\glink}[1]{\hyperlink{#1}{\texttt{#1}}}
\newcommand{\gtarget}[1]{\hypertarget{#1}{\texttt{#1}}}
\newcommand{\glabel}[1]{\texttt{#1}} % use for group links with no target
\newcommand{\XNgdis}{X(N)}
\newcommand{\JNgdis}{J(N)}
\newcommand{\XNconn}{X(N)'}
\newcommand{\JNconn}{J(N)'}

\DeclareMathOperator{\alg}{alg}

\newcommand{\kron}[2]{\left(\frac{#1}{#2}\right)}
\DeclareMathOperator{\disc}{disc}
\DeclareMathOperator{\Tr}{Tr}
                               
\numberwithin{equation}{subsection}
\newtheorem{theorem}{Theorem}[section]
\numberwithin{theorem}{subsection}
\newtheorem{lemma}[theorem]{Lemma}
\newtheorem{corollary}[theorem]{Corollary}
\newtheorem{proposition}[theorem]{Proposition}

\theoremstyle{definition}
\newtheorem{definition}[theorem]{Definition}

\newtheorem{conjecture}[theorem]{Conjecture}
\newtheorem{example}[theorem]{Example}
\newtheorem{algorithm}[theorem]{Algorithm}
\newtheorem{remark}[theorem]{Remark}

\makeatletter\let\@wraptoccontribs\wraptoccontribs\makeatother

\SetSymbolFont{stmry}{bold}{U}{stmry}{m}{n}

\usepackage{longtable}

\definecolor{mylinkcolor}{rgb}{0.65,0.0,0.0}
\definecolor{myurlcolor}{rgb}{0.0,0.0,0.65}
\hypersetup{colorlinks=true,urlcolor=myurlcolor,citecolor=myurlcolor,linkcolor=mylinkcolor,linktoc=page,breaklinks=true}

\begin{document}

\title[$\ell$-adic images of Galois for elliptic curves over $\Q$]{$\ell$-adic images of Galois for elliptic curves over $\Q$}
%keywords{}
\author{Jeremy Rouse}
\address{Dept. of Mathematics and Statistics, Wake Forest University, Winston-Salem, NC 27109}
\email{rouseja@wfu.edu}
\urladdr{\url{https://users.wfu.edu/rouseja/}}
\author{Andrew V. Sutherland}
\address{Dept. of Mathematics, Massachusetts Institute of Technology, Cambridge, MA 02139}
\email{drew@math.mit.edu}
\urladdr{\url{https://math.mit.edu/~drew/}}
\author{David Zureick-Brown}
\address{Dept.~of Mathematics, Emory University, Atlanta, GA 30322}
\email{dzb@mathcs.emory.edu}
\urladdr{\url{https://www.math.emory.edu/~dzb/}}
\contrib[and an appendix with]{John Voight}

\date{\today}

\begin{abstract}
We discuss the $\ell$-adic case of Mazur's ``Program B'' over $\Q$, the problem of classifying the possible images of $\ell$-adic Galois representations attached to elliptic curves~$E$ over~$\Q$, equivalently, classifying the rational points on the corresponding modular curves.  The primes $\ell=2$ and $\ell\ge 13$ are addressed by prior work, so we focus on the remaining primes $\ell = 3, 5, 7, 11$. For each of these~$\ell$, we compute the directed graph of arithmetically maximal $\ell$-power level modular curves $X_H$, compute explicit equations for all but three of them, and classify the rational points on all of them except $X_{\rm ns}^{+}(N)$, for $N = 27, 25, 49, 121$, and two level $49$ curves of genus $9$ whose Jacobians have analytic rank $9$.

Aside from the $\ell$-adic images that are known to arise for infinitely many $\Qbar$-isomorphism classes of elliptic curves $E/\Q$, we find only 22 exceptional images that arise for any prime~$\ell$ and any $E/\Q$ without complex multiplication;  these exceptional images are realized by~20 non-CM rational $j$-invariants.  We conjecture that this list of 22 exceptional images is complete and show that any counterexamples must arise from unexpected rational points on $X_{\rm ns}^+(\ell)$ with $\ell\ge 19$, or one of the six modular curves noted above. This yields a very efficient algorithm to compute the $\ell$-adic images of Galois for any elliptic curve over $\Q$.

In an appendix with John Voight we generalize Ribet's observation that simple abelian varieties attached to newforms on $\Gamma_1(N)$ are of $\GL_2$-type; this extends Kolyvagin's theorem that analytic rank zero implies algebraic rank zero to isogeny factors of the Jacobian of $X_H$.
\end{abstract}

\maketitle
\vspace{-4pt} % pull Mazur quote onto front page

%----------------------------------------------------------------------------
\section{Introduction}
%----------------------------------------------------------------------------

To an elliptic curve $E$ over a number field $K$ one can associate representations
\[
\arraycolsep=2.5pt
 \begin{array}{rlcl}
 \rho_{E, N}\colon
 & G_{K} \rightarrow \Aut E[N] 
 &\simeq 
 &\GL_{2}(\mathbb{Z}/N \mathbb{Z}) \vspace{3pt}
 \\
 \rho_{E, \ell^{\infty}}\colon 
 & G_{K} \rightarrow \GL_{2}(\mathbb{Z}_{\ell}) 
 &= 
 &\displaystyle{\varprojlim_n} \GL_{2}\left(\Z/\ell^n\Z\right)
 \\
 \rho_E \colon
 & G_{K} \rightarrow \GL_{2}(\widehat{\Z}) 
 &= 
 &\displaystyle{\varprojlim_N} \GL_{2}\left(\Z/N\Z\right)
 \end{array}
\]
of the absolute Galois group $G_{K} \coloneqq \Gal(\overline{K}/K)$. If $E$ does not have complex multiplication (CM), then the images of these representations are open \cite{Serre:openImage}, and thus of finite index. It is conjectured (\cite{zywina:On-the-possible-images-of-the-mod-ell-representations-associated-to-elliptic-curves-over-Q-arxiv}*{Conjecture 1.12}, 
\cite{sutherland:Computing-images-of-Galois-representations-attached-to-elliptic-curves}*{Conjecture 1.1})  that $\rho_{E, \ell}$ is surjective for non-CM $E/\Q$ for all $\ell > 37$, and that there is an absolute bound on the index of $\rho_E(G_\Q)$.

Concordant with his watershed papers \cites{mazur:eisenstein, Mazur:isogenies} classifying rational points on the modular curves $X_1(N)$ and $X_0(N)$ of prime level, Mazur proposed the following program (labeled ``Program B'') in his article \cite{mazur:rationalPointsOnModular}:
\begin{center}
\begin{minipage}[c]{14cm}
``Given a number field $K$ and a subgroup $H$ of $\GL_2(\Zhat) \simeq \prod_p \GL_2 (\Z_p)$, classify all elliptic curves $E/K$ whose associated Galois representation on torsion points maps $\Gal(\Kbar/K)$ into $H \leq \GL_2 (\Zhat)$.''
\end{minipage}
\end{center}

Program B has seen substantial progress  over the past decade: for prime level see
\cites{
biluP:Serres-uniformity-problem-in-the-split-Cartan-case,
bilupr:rationalpoints,
zywina:On-the-possible-images-of-the-mod-ell-representations-associated-to-elliptic-curves-over-Q-arxiv,
sutherland:Computing-images-of-Galois-representations-attached-to-elliptic-curves,
balakrishnanDMTV:cartan,
balakrishnanDMTV:Quadratic-Chabauty-For-Modular-Curves:-Algorithms-And-Examples,
leFournL:Residual-Galois-representations-of-elliptic-curves-with-image-contained-in-the-normaliser-of-a-nonsplit-Cartan},
for prime power level see
\cites{
RZB,
SutherlandZ:Modular-curves-of-prime-power-level-with-infinitely-many-rational-points},
for multi-prime level see
\cites{
zywina:Possible-indices-for-the-Galois-image-of-elliptic-curves-over-Q-arxiv,
daniels2019serre,
gonzalesJLR:Elliptic-curves-with-abelian-division-fields,
brauJ:Elliptic-curves-with-2-torsion-contained-in-the-3-torsion-field,
morrow:Composite-images-of-Galois-for-elliptic-curves-over-Q-and-entanglement-fields,
jonesM:Elliptic-curves-with-non-abelian-entanglements,
danielsLR:Coincidences-of-division-fields,
danielsM:A-group-theoretic-perspective-on-entanglements-of-division-fields,
danielsLRM:Towards-a-classification-of-entanglements-of-Galois-representations-attached-to-elliptic-curves,
rakvi:A-Classification-of-Genus-0-Modular-Curves-with-Rational-Points,
BarbulescuShinde:classification-of-ECM-friendly-families},
and for CM curves see
\cites{bourdonC:Torsion-points-and-Galois-representations-on-CM-elliptic-curves,
  lozanoRobledo:Galois-representations-attached-to-elliptic-curves-with-complex-multiplication,
  campagnaP:Entanglement-in-the-family-of-division-fields-of-elliptic-curves-with-complex-multiplication,
  lombardo:Galois-representations-attached-to-abelian-varieties-of-CM-type
}.

%----------------------------------------------------------------------------
\subsection{Main Theorem}
%----------------------------------------------------------------------------

Our main theorem addresses almost all of the remaining groups of prime power level. For each open subgroup $H$ of $\GL_2(\Zhat)$, we use $X_H$ to denote the modular curve whose non-cuspidal $K$-rational points classify elliptic curves $E/K$ for which $\rho_E(G_K)$ is contained in $H$ (see Section~\ref{ssec:modular-curves} for a precise definition).

\begin{remark}\label{remark:conjugac}
We write $\rho_{E,\ell^\infty}(G_K)= H$ to indicate that $H$ is conjugate to the inverse image of $\rho_{E,\ell^\infty}(G_K)$ under the projection $\GL_2(\Zhat)\to \GL_2(\Z_\ell)$.  In general, we are only interested in the image of $\rho_E$ up to conjugacy (the exact image depends on a choice of basis).
\end{remark}

\begin{definition}
\label{defn:exceptional}
Let $\ell$ be a prime, let $H$ be an open subgroup of $\GL_2(\Zhat)$ of $\ell$-power level, and let $K$ be a number field. We say that a point $P \in X_H(K)$ is \defi{exceptional} if $X_H(K)$ is finite and $P$ corresponds to a non-CM elliptic curve $E/K$. We call an element of $K$ \defi{exceptional} if it is the $j$-invariant of an exceptional point. If there is a non-CM elliptic curve $E/K$ such that $j(E)$ is exceptional and $\rho_{E,\ell^\infty}(G_K)=H$ (as opposed to a subgroup of $H$) then we refer to $H$ as $K$-\defi{exceptional}; we may omit $K$ when $K = \Q$.
\end{definition}

For $K=\Q$ the known exceptional $j$-invariants of prime power level are listed in Table \ref{table:exceptional-points} along with the known exceptional points on $X_H$ for exceptional subgroups $H$ of prime power level that do not contain $-I$; the groups of prime level appear in both \cite{sutherland:Computing-images-of-Galois-representations-attached-to-elliptic-curves} and \cite{zywina:On-the-possible-images-of-the-mod-ell-representations-associated-to-elliptic-curves-over-Q-arxiv}.

\begin{remark}
For the $H$ and $j$ listed in Table \ref{table:exceptional-points}, if $-I \in H$ then all but finitely many elliptic curves with $j(E) = j$ have $\ell$-adic image equal to $H$ (rather than a proper subgroup). If $H' \subsetneq H$ is a proper subgroup such that $H = \langle H',-I\rangle$ and $E/\Q$ is an elliptic curve with $\ell$-adic image~$H$ and $j(E) \ne 1728$, then there are $[N_{\GL_{2}(\Z_{\ell})}(H) : N_{\GL_{2}(\Z_{\ell})}(H')]$ quadratic twists $E'$ of $E$ whose $\ell$-adic image is $H'$; see Section~\ref{sec:universal-curve}.
\end{remark}
\begin{remark} 
\label{remark:generic-image}
If $P \in X_H(K)$ corresponds to an elliptic curve $E$ then $\rho_E(G_K)$ may be a \emph{proper} subgroup of $H$. For example, any elliptic curve $E/\Q$ with $j$-invariant $-11^2$ admits a rational 11-isogeny, which forces $\rho_{E,11^\infty}(G_\Q)$ to lie in the Borel group $B(11)$, but $\rho_{E,11^\infty}(G_\Q)$ is generically the index-5 subgroup
$\langle\smallmat{10}{1}{0}{10}, \smallmat{3}{0}{0}{2} \rangle$ with label \glink{11.60.1.4}.
In this example, $-11^2$ is an exceptional $j$-invariant and \glink{11.60.1.4} is an exceptional group, but $B(11)$ is \emph{not} an exceptional group, because no non-CM elliptic curve over $\Q$ has $11$-adic image $B(11)$.
The group $B(11)=\glink{11.12.1.1}$ appears in Table~\ref{table:missing} along with 13 other groups $H$ of prime power level for which a similar phenomenon occurs.
\end{remark}

\begin{conjecture}
Table \ref{table:exceptional-points} lists all rational exceptional $j$-invariants of prime power level.
\end{conjecture}

\begin{theorem}
\label{T:mainTheorem}
Let $\ell$ be a prime, let $E/\Q$ be an elliptic curve without complex multiplication, and let $H = \rho_{E,{\ell}^{\infty}}(G_{\Q})$.
Exactly one of the following is true:
\begin{enumerate}[{\rm(i)}]
\item \label{item:mainTheorem-1}
the modular curve $X_H$ is isomorphic to $\P^1$ or an elliptic curve of rank one, in which case $H$ is identified in \cite{SutherlandZ:Modular-curves-of-prime-power-level-with-infinitely-many-rational-points}*{Corollary 1.6} and $\langle H,-I\rangle$ is listed in \cite{SutherlandZ:Modular-curves-of-prime-power-level-with-infinitely-many-rational-points}*{Tables 1--4};
\item \label{item:mainTheorem-2}
the modular curve $X_H$ has an exceptional rational point and $H$ is listed in Table \ref{table:exceptional-points}; 
\item \label{item:mainTheorem-3}
$H \leq N_{\ns}(3^3), N_{\ns}(5^2), N_{\ns}(7^2), N_{\ns}(11^2)$, or $N_{\ns}(\ell)$ for some $\ell \ge 19$;
\item \label{item:mainTheorem-4}
$H$ is a subgroup of one of the groups labeled\, \textup{\glink{49.147.9.1}} or \textup{\glink{49.196.9.1}}.
\end{enumerate}
\end{theorem}
\noindent
See Subsection \ref{ssec:subgroup-labels} for a discussion of the group labels, and see Tables~\ref{table:remaining-cases} and~\ref{table:amstats-maximal} for a summary of the groups $H$ that appear in Theorem~\ref{T:mainTheorem} for $\ell \leq 37$.
\medskip

We conjecture that in cases (\ref{item:mainTheorem-3}) and (\ref{item:mainTheorem-4}) the curve $X_H$ has no exceptional rational points; see Section~\ref{section:comments-on-the-remaining-cases} for a discussion of evidence for this conjecture. For the 6 groups in Table \ref{table:remaining-cases}, the genus of $X_H$ is 9, 12, 14, 69, or 511, and for each isogeny factor $A$ of $J_H$, the analytic rank of $A$ is equal to the dimension of $A$. We were not able to provably determine the rational points on these curves; see Section \ref{section:comments-on-the-remaining-cases} for a more detailed discussion of these cases.

\begin{remark} 
The recent groundbreaking work of \cites{balakrishnanDMTV:cartan, balakrishnanDMTV:Quadratic-Chabauty-For-Modular-Curves:-Algorithms-And-Examples} yields a complete classification of the possible $13$-adic and $17$-adic images for elliptic curves over $\Q$; these are the only primes other than $\ell = 2$ for which this is currently known.
\end{remark}

%----------------------------------------------------------------------------
\subsection{Outline of proof}
%----------------------------------------------------------------------------

The proof of Theorem~\ref{T:mainTheorem} breaks down into the following steps.
\begin{itemize}\setlength{\itemsep}{2pt}
\item We begin by computing the collection $\calS$ of \defi{arithmetically maximal} subgroups of prime power level; see Section \ref{sec:group-theory} and the tables in Appendix \ref{S:tables}.
\item For $H \in \calS$ we compute the isogeny decomposition of $J_H$ and the analytic ranks of its simple factors; see Sections \ref{sec:enumeration-of-moduli} and \ref{sec:analytic-rank}.
\item For many $H \in \calS$ we compute equations for $X_H$ and $j_H\colon X_H \to X(1)$; see Section \ref{sec:comp-equat-x_h}.
\item For most $H \in \calS$ with $-I \in H$, we determine $X_H(\Q)$; see Sections   \ref{sec:bookkeeping}, \ref{sec:enumeration-of-moduli}, \ref{sec:analysis-of-rational-points-small-genus}.

\item For $H \in \calS$ with $-I \not \in H$, we compute equations for the universal curve $E \to U$, where $U \subset X_H$ is the locus of points with $j \ne 0, 1728,\infty$; see Section \ref{sec:universal-curve}.
\end{itemize}
Analyses of many of the relevant $X_H$ already exist in the literature, which is summarized in Section \ref{sec:bookkeeping}; the remaining cases are listed in Table~\ref{table:cases}.  Some are routine, for example, some of the genus~2 curves in Section \ref{sec:analysis-of-rational-points-small-genus} have rank zero Jacobians and are handled with no additional work by \magma's \magmacmd{Chabauty0} command, but many require novel arguments.
\begin{itemize}\setlength{\itemsep}{3pt}  
\item We use the moduli enumeration of Section \ref{sec:enumeration-of-moduli} to show that $X_H(\F_p)$ and therefore $X_H(\Q)$ is empty in several cases without computing equations for $X_H$; see Table \ref{table:empty-via-moduli}.
\item In Section \ref{sec:analysis-of-rational-points-small-genus} a few curves admit an involution $\iota$ such that the map $P \mapsto P - \iota(P)$ projects to the torsion subgroup of the Jacobian, and we ``sieve'' with this map. In one case we are only able to compute an upper bound on $J_H(\Q)_{\tors}$, but are then able to circumvent this via sieving.
\item In Section \ref{sec:comp-equat-x_h} we compute canonical models of two high genus curves (genus 36 and~43), and in Section \ref{ssec:genus-36-43} use these models to set up a non-trivial local computation to show that $X_H(\Q_p)$ is empty, where $p$ divides the level of $H$ and is thus the only possible prime of bad reduction.
\item The curve $X_H$ from Section \ref{ssec:genus-41} has genus 41 and we were not able to compute a model; $X_H$ maps to $X_{{\rm ns}}^{+}(11)$, which is a rank 1 elliptic curve, and we perform an ``equationless'' Mordell--Weil sieve (via moduli) to show that $X_H(\Q)$ is empty.
\item Throughout, our approaches are heavily informed by the isogeny decompositions and analytic ranks computed in Section \ref{sec:analytic-rank}. In the end, for a rigorous proof we often do not need to know that the analytic and algebraic ranks agree, but a few of the cases in Section \ref{sec:analysis-of-rational-points-small-genus} rely on this and thus make crucial use of Appendix \ref{S:kolyvagin}. 
\item Appendix \ref{S:kolyvagin} contains an additional novelty: we make use of the Jacobian of a connected, but geometrically disconnected curve.
\end{itemize}

\FloatBarrier
\begin{table}[htp]
\def\arraystretch{1.3}
\begin{tabular}{lccc}
\hspace{1pt}label & level & notes & $j$-invariants/models of exceptional points \\\toprule
\glink{16.64.2.1}&$2^4$ & $N_{\ns}(16)$& $-2^{18} \cdot 3 \cdot 5^{3} \cdot 13^{3} \cdot 41^{3} \cdot 107^{3}  /17^{16} $ \\
& & & $-2^{21}\!\cdot 3^{3}\!\cdot 5^{3}\!\cdot 7\cdot 13^{3}\!\cdot 23^{3}\!\cdot 41^{3}\!\cdot 179^{3}\!\cdot 409^{3} /79^{16}$  \\
\glink{16.96.3.335}&$2^4$ & $H(4) \subsetneq N_{\spc}(4)$ & \href{https://www.lmfdb.org/EllipticCurve/Q/3362/a/2}{$257^{3} / 2^{8}$} \\
\glink{16.96.3.338}&$2^4$ & $H(4) \subsetneq N_{\spc}(4)$ & \href{https://www.lmfdb.org/EllipticCurve/Q/200/b/2}{$2^{11}$} \\
\glink{16.96.3.343}&$2^4$ & $H(4) \subsetneq N_{\spc}(4)$& \href{https://www.lmfdb.org/EllipticCurve/Q/3362/a/1}{$17^{3} \cdot 241^{3} /2^{4}$} \\
\glink{16.96.3.346}&$2^4$ & $H(4) \subsetneq N_{\spc}(4)$& \href{https://www.lmfdb.org/EllipticCurve/Q/200/b/1}{$2^{4} \cdot 17^{3}$} \\
\glink{32.96.3.82}&$2^5$ & $H(4) \subsetneq N_{\spc}(8)$& \href{https://www.lmfdb.org/EllipticCurve/Q/17918/b/1}{$3^{3} \cdot 5^{6} \cdot 13^{3} \cdot 23^{3} \cdot 41^{3}  /(2^{16} \cdot 31^{4})$} \\
\glink{32.96.3.230}&$2^5$ & $H(4) \subsetneq N_{\spc}(4)$& \href{https://www.lmfdb.org/EllipticCurve/Q/17918/c/2}{$-3^{3} \cdot 5^{3} \cdot 47^{3} \cdot 1217^{3}  /(2^{8} \cdot 31^{8})$}\\
\midrule
\glink{25.50.2.1} & $5^2$ & $H(5) = N_{\ns}(5)$ & \href{https://www.lmfdb.org/EllipticCurve/Q/396900/e/1}{$2^{4} \cdot 3^{2} \cdot 5^{7} \cdot 23^{3}$} \\
\glink{25.75.2.1} & $5^2$ &$H(5) = N_{\spc}(5)$  & \href{https://www.lmfdb.org/EllipticCurve/Q/21175/bm/1}{$2^{12} \cdot 3^{3} \cdot 5^{7} \cdot 29^{3}/7^{5}$} \\
\midrule
\glink{7.56.1.2} & 7 & $H(7) \subsetneq N_{\spc}(7)$  & \href{https://www.lmfdb.org/EllipticCurve/Q/2450/i/1}{$3^3 \cdot 5 \cdot 7^5 / 2^7$} \\
\gtarget{7.112.1.2} & $7$ & $-I \not \in \langle\smallmat{0}{5}{5}{0}, \smallmat{0}{5}{3}{0} \rangle$ & \href{https://www.lmfdb.org/EllipticCurve/Q/2450/y/1}{$y^2 + xy + y = x^3 - x^2 - 2680x - 50053$}\\
& & &  $y^2 + xy + y = x^3 - x^2 - 131305x + 17430697$\\
\midrule
\gtarget{11.60.1.3} & $11$ & $\langle\smallmat{10}{1}{0}{10}, \smallmat{2}{0}{0}{3}\rangle$ &  \href{https://www.lmfdb.org/EllipticCurve/Q/1089/i/2}{$-11\cdot 131^3$} \\
\gtarget{11.60.1.4} & $11$ & $\langle\smallmat{10}{1}{0}{10}, \smallmat{3}{0}{0}{2} \rangle$ &  \href{https://www.lmfdb.org/EllipticCurve/Q/1089/c/2}{$-11^2$} \\
\gtarget{11.120.1.3} & $11$ & $-I \not \in \langle\smallmat{1}{1}{0}{1}, \smallmat{4}{0}{0}{2} \rangle$ & \href{http://lmfdb.org/EllipticCurve/Q/121/c/1}{$y^2 + xy = x^3 + x^2 - 3632x + 82757$}\\
\gtarget{11.120.1.4} & $11$ & $-I \not \in \langle\smallmat{1}{1}{0}{1}, \smallmat{3}{0}{0}{2} \rangle$ & \href{https://www.lmfdb.org/EllipticCurve/Q/121/a/1}{$y^2 + xy + y = x^3 + x^2 - 305x + 7888$}\\
\gtarget{11.120.1.8} & $11$ & $-I \not \in \langle\smallmat{1}{1}{0}{1}, \smallmat{2}{0}{0}{3} \rangle$ & \href{https://www.lmfdb.org/EllipticCurve/Q/121/a/2}{$y^2 + xy + y = x^3 + x^2 - 30x - 76$}\\
\gtarget{11.120.1.9} & $11$ & $-I \not \in \langle\smallmat{1}{1}{0}{1}, \smallmat{2}{0}{0}{4} \rangle$ & \href{https://www.lmfdb.org/EllipticCurve/Q/121/c/2}{$y^2 + xy = x^3 + x^2 - 2x - 7$}\\
\midrule  
\glink{13.91.3.2} & 13 & $S_4(13)$ & \href{https://www.lmfdb.org/EllipticCurve/Q/50700/z/1}{$2^4\cdot 5\cdot 13^4\cdot 17^3/3^{13}$}\\
& & & \href{https://www.lmfdb.org/EllipticCurve/Q/61347/bb/1}{$-2^{12}\cdot 5^3\cdot 11\cdot 13^4/3^{13}$} \\
& & & $2^{18}\!\cdot 3^3\!\cdot\! 13^4\!\cdot\! 127^3\!\cdot  139^3\!\cdot\! 157^3\!\cdot  283^3\!\cdot  929 / (5^{13}\!\cdot 61^{13})$ \\
\midrule
\gtarget{17.72.1.2} & 17 & $\langle\smallmat{4}{1}{0}{4},\smallmat{2}{0}{0}{3}\rangle$ & \href{https://www.lmfdb.org/EllipticCurve/Q/14450/o/2}{$-17\cdot 373^3/2^{17}$}\\
\gtarget{17.72.1.4} & 17  & $\langle\smallmat{4}{1}{0}{4},\smallmat{3}{0}{0}{2}\rangle$ & \href{https://www.lmfdb.org/EllipticCurve/Q/14450/b/2}{$-17^2 \cdot 101^3/2$} \\
\midrule
\gtarget{37.114.4.1} & 37 & $\langle\smallmat{8}{1}{0}{8},\smallmat{6}{0}{0}{9}\rangle$ & \href{https://www.lmfdb.org/EllipticCurve/Q/1225/b/2}{$-7\cdot 11^3$} \\
\gtarget{37.114.4.2} & 37 & $\langle\smallmat{8}{1}{0}{8},\smallmat{9}{0}{0}{6}\rangle$ & \href{https://www.lmfdb.org/EllipticCurve/Q/1225/b/1}{$-7\cdot 137^3\cdot 2083^3$} \\
\bottomrule
\end{tabular}
\vspace*{1em}
\caption{Known (and conjecturally all) exceptional groups and $j$-invariants of prime power level. When $-I \not \in H$ we list models of elliptic curves $E/\Q$ for which $\rho_{E,\ell^\infty}(G_\Q)=H$. The \magma script \repolink{ratpoints}{checkimages.m} in the software repository associated with this paper \cite{RouseSZB:magma-scripts-ell-adic-images} verifies each line of this table. Generators for the arithmetically maximal groups can be found in Appendix~\ref{S:tables}.}
\label{table:exceptional-points}
\end{table}

\begin{table}[htp]
\def\arraystretch{1.25}
\begin{tabular}{lccc}
label & level & group & genus \\\toprule
\glink{27.243.12.1} & $3^3$ &$N_{\ns}(3^3)$ & 12\\
\glink{25.250.14.1} & $5^2$ &$N_{\ns}(5^2)$ & 14\\
\glink{49.1029.69.1} & $7^2$ &$N_{\ns}(7^2)$ & 69\\
\glink{49.147.9.1} & $7^2$ & $\left \langle\smallmat{16}{6}{20}{45}, \smallmat{20}{17}{40}{36}\right\rangle$ & 9\\
\glink{49.196.9.1} & $7^2$ &$\left \langle\smallmat{42}{3}{16}{31}, \smallmat{16}{23}{8}{47}\right\rangle$ & 9\\
\glink{121.6655.511.1} & $11^2$ &$N_{\ns}(11^2)$ & 511\\
\bottomrule
\end{tabular}
\vspace*{1em}
\caption{Arithmetically maximal groups of $\ell$-power level with $\ell \leq 17$ for which $X_H(\Q)$ is unknown; each has rank $=$ genus, rational CM points, no rational cusps, and no known exceptional points.}
\label{table:remaining-cases}
\end{table}
\vspace{-4pt}

\begin{table}[htp!]
\begin{tabular}{r|rrrrrrrrrrrr}\toprule
$\ell$ & $2$ & $3^*$ & $5^*$ & $7^*$ & $11^*$ & $13$ & $17$ & $37^*$ & $\text{other}^*$\\
subgroups & 1208 & 47 & 25 & 17 & 8 & 12 & 3 &3 &1\\
exceptional subgroups& 7 & 0 & 2 & 2 & 6 & 1 & 2 &2&0\\
non-exceptional subgroups & 1201 & 47 & 23 & 15 &2 & 11 & 1 & 1&1\\
max level & 32 & 27 & 25 & 7 & 11 & 13 & 17 & 37&1\\
max index & 96 & 72 & 120 & 112 & 120 & 91 & 72 & 114&1\\
max genus & 3 & 0 & 2 & 1 & 1 & 3 & 1 & 4&0\\
\bottomrule
\end{tabular}
\bigskip

\caption{Numerical summary of open $H\leq\GL_2(\Zhat)$ of prime power level which can conjecturally occur as $\rho_{E,{\ell}^{\infty}}(G_{\Q})$ for $E/\Q$ without CM.
The starred primes depend on the conjecture that $N_{\ns}(\ell)$ for $\ell > 17$ and the groups from Table \ref{table:remaining-cases}  do not occur. See the file \repolink{models}{sample.txt} in \cite{RouseSZB:magma-scripts-ell-adic-images} for a list of elliptic curves over $\Q$ that realize each of the subgroups in this table.}
\label{table:amstats-maximal}
\end{table}

%----------------------------------------------------------------------------
\subsection{Applications}
%----------------------------------------------------------------------------

In \cite{greenberg:The-image-of-Galois-representations-attached-to-elliptic-curves-with-an-isogeny} Greenberg shows that if $\ell$ is a prime number and $E/\Q$ is an elliptic with a cyclic $\ell$-isogeny, then $\rho_{E,\ell^{\infty}}(G_\Q)$ is as large as possible given the existence of the isogeny if $\ell > 7$, or the existence of $\ell$-power isogenies if $\ell = 5$. The $\ell = 7$ case is handled in \cite{greenbergRSS:On-elliptic-curves-with-an-isogeny-of-degree-7}. As a consequence of our work, we obtain a classification of the $3$-adic
image for a non-CM elliptic curve $E/\Q$ with a rational $3$-isogeny.
\begin{corollary}\label{cor:3adic}
Suppose that $E/\Q$ is a non-CM elliptic curve with a rational $3$-isogeny and let $H = \langle \rho_{E,3^{\infty}}(G_{\Q}), -I \rangle$. Then exactly one of the following is true:
\begin{enumerate}[{\rm(i)}]
\item $[\GL_{2}(\Z_{3}) : H] = 4$, and $H$ has label \textup{\texttt{3.4.0.1}};
\item  $[\GL_{2}(\Z_{3}) : H] = 12$ and $H$ has label \textup{\texttt{3.12.0.1}}, \textup{\texttt{9.12.0.1}}, or \textup{\texttt{9.12.0.2}};
\item  $[\GL_{2}(\Z_{3}) : H] = 36$ and $H$ has label \textup{\texttt{9.36.0.1}}, \textup{\texttt{9.36.0.2}}, \textup{\texttt{9.36.0.3}}, \textup{\texttt{9.36.0.4}}, \textup{\texttt{9.36.0.5}}, \textup{\texttt{9.36.0.6}}, \textup{\texttt{9.36.0.7}}, \textup{\texttt{9.36.0.8}}, \textup{\texttt{9.36.0.9}}, or \textup{\texttt{27.36.0.1}}.
\end{enumerate}
In each case, the relevant modular curve has genus zero and infinitely many rational points. Moreover, $\rho_{E,3^{\infty}}(G_{\Q})$ (which may have index two in $H$) contains all matrices congruent to the identity matrix modulo $27$ and all matrices $\smallmat{\lambda}{0}{0}{\lambda}$ with $\lambda \in \Z_{3}$ and $\lambda \equiv 1 \pmod{9}$.
\end{corollary}

This result plays a role in the classification of odd degree isolated points
on $X_{1}(N)$ given in \cite{bourdonGRS:Odd-degree-isolated-points-on-X1N-with-rational-j-invariant}.
In \cites{lombardo:Explicit-Kummer-theory-for-elliptic-curves-arxiv, lombardoT:Some-uniform-bounds-for-elliptic-curves-overQ} Lombardo and Tronto study the degree of the extension $\Q(E[N],N^{-1} \alpha)/\Q(E[N])$ where $\alpha \in E(\Q)$ is a point of infinite order, and the presence of scalars
in $\rho_{E,\ell^{\infty}}(G_\Q)$ is used to bound certain Galois cohomology
groups. Corollary~\ref{cor:3adic}
gives precise information about the scalars present in the $3$-adic image.
\medskip

An additional application of Theorem~\ref{T:mainTheorem} is a very fast ``algebraic'' algorithm to compute all of the $\ell$-adic Galois images of a given non-CM elliptic curve $E/\Q$; this algorithm is described in Section~\ref{sec:computing-images}.  By exploiting the classification of $\ell$-adic Galois images of CM elliptic curves obtained by Lozano-Robledo \cite{lozanoRobledo:Galois-representations-attached-to-elliptic-curves-with-complex-multiplication} we are able to extend this algorithm to also handle CM elliptic curves over $\Q$ (and more generally, any CM elliptic curve over its minimal field of definition) as described in Section~\ref{sec:computing-cm-images}.
An implementation of this algorithm is available in the GitHub repository associated to this paper \cite{RouseSZB:magma-scripts-ell-adic-images},
and we have applied it to the elliptic curves $E/\Q$ in the $L$-functions and Modular Forms Database (LMFDB) \cite{lmfdb}, as well as two other large databases of elliptic curves over $\Q$ \cites{SteinWatkins:A-database-of-elliptic-curves-first-report,balakrishnanHKSSW:Databases-of-elliptic-curves-ordered-by-height-and-distributions-of-Selmer-groups-and-ranks}.  This data has now been added to the LMFDB, where it extends existing $2$-adic and mod-$\ell$ Galois image data that was previously available. As a result, the $\ell$-adic images $\rho_{E,\ell^\infty}(G_\Q)$ are now known for all elliptic curves $E/\Q$ of conductor up to $500000$ and every prime~$\ell$.
\medskip

Other applications abound; examples include 
\cites{jonesr:galoistheoryof,
cerchiaR:Uniform-bounds-on-the-image-of-the-arboreal-Galois-representations-attached-to-non-CM-elliptic-curves,
lombardo:Explicit-Kummer-theory-for-elliptic-curves-arxiv,
guvzvic:Torsion-of-elliptic-curves-with-rational-j-invariant-defined-over-number-fields-of-prime-degree,
bourdonGRS:Odd-degree-isolated-points-on-X1N-with-rational-j-invariant,
gonzalez-jimenezN:Growth-of-torsion-groups-of-elliptic-curves-upon-base-change,
BELOV,
gonzalesJLR:On-the-torsion-of-rational-elliptic-curves-over-quartic-fields,
gonzalesJLR:On-the-minimal-degree-of-definition-of-p-primary-torsion-subgroups-of-elliptic-curves,
danielsLRNS:-Torsion-subgroups-of-rational-elliptic-curves-over-the-compositum-of-all-cubic-fields,
reiter:thesis,
chiloyanLR:-A-classification-of-isogeny-torsion-graphs-of-Q-isogeny-classes-of-elliptic-curves,
BarbulescuShinde:classification-of-ECM-friendly-families
}.

%----------------------------------------------------------------------------
\subsection{Code}
%----------------------------------------------------------------------------

We make extensive use of the computer algebra system \magma \cite{Magma}.
Code verifying the computational claims made in this paper is available at the GitHub repository
\begin{center}
\url{https://github.com/AndrewVSutherland/ell-adic-galois-images}
\end{center}
which also includes related algorithms and data.

%----------------------------------------------------------------------------
\section{\texorpdfstring{The Modular curves $X_H$}{The modular curves X\_H}}
\label{sec:modular-curves}
%----------------------------------------------------------------------------

\subsection{Notation}
In this article $\ell$ and $p$ denote primes and $N$ is a positive integer.  We define
\[
\Z(N)\coloneqq \Z/N\Z, \qquad \GL_2(N)\coloneqq \GL_2(\Z/N\Z),\qquad \SL_2(N)\coloneqq \SL_2(\Z/N\Z),
\]
with the understanding that $\Z(1)$ is the zero ring and $\GL_2(1)$ and $\SL_2(1)$ are trivial groups.
We then have $\Zhat=\varprojlim \Z(N)$ and $\GL_2(\Zhat)=\varprojlim \GL_2(N)$, and natural projection maps
\[
\pi_N\colon \GL_2(\Zhat)\to \GL_2(N).
\]
We view $\Z(N)^2$ as a $\Z(N)$-module of column vectors equipped with a left action of $\GL_2(N)$ via matrix-vector multiplication.

For $N>1$ we use $B(N)$ to denote the \defi{Borel subgroup} of $\GL_2(N)$ consisting of upper triangular matrices, $C_{\spc}(N)$ to denote the \defi{split Cartan subgroup} of diagonal matrices in $\GL_2(N)$, and $C_{\ns}(N)$ to denote the \defi{nonsplit Cartan subgroup} of $\GL_2(N)$, which can be constructed by picking an imaginary quadratic order $\mathcal O$ in which every prime $\ell|N$ is inert and embedding $(\mathcal O/N\mathcal O)^\times$ in $\GL_2(N)$ via its action on a $\Z(N)$-basis for $\mathcal O/N\mathcal O$ (see Section~\ref{sec:computing-cm-images}), which is unique up to conjugacy.  We use $N_{\spc}(N)$ to denote the subgroup of the normalizer of $C_{\spc}(N)$ in $\GL_2(N)$ whose projection to $\GL_2(M)$ contains $C_{\spc}(M)$ with index 2 for all $M>2$ dividing~$N$, which is unique up to conjugacy and equal to the normalizer when $N>2$ is a prime power; we similarly define $N_{\ns}(N)$.   If $G$ is a subgroup of the projectivization $\PGL_2(N)$ of $\GL_2(N)$, then we denote by $G(N)$ the preimage of $G$ in $\GL_2(N)$.

%----------------------------------------------------------------------------
\subsection{Group Theory and Galois representations}
\label{ssec:group-theory}
%----------------------------------------------------------------------------

Let $N \ge 1$ be an integer, let $k$ be a perfect field of characteristic coprime to $N$, and let $G_k \coloneqq \Gal(\kbar/k)$. A basis $(P_1,P_2)$ of $E[N](\kbar)$ determines an isomorphism $E[N](\kbar) \overset\sim\rightarrow \Z(N)^2$ via the map $P_i \mapsto e_i$. This gives rise to an isomorphism $\iota \colon \Aut \left( E[N](\kbar)\right) \simeq \GL_2(N)$ as follows: for $\phi\in\Aut \left(E[N](\kbar)\right)$ such that
\begin{align*}
 \phi(P_1) = &\, aP_1 + cP_2 \\
 \phi(P_2) = &\, bP_1 + dP_2,
\end{align*}
 we define
\[
\iota(\phi) \coloneqq \left[ \begin{matrix} a & b \\ c & d \end{matrix} \right],
\]
which yields a Galois representation $G_k \to \GL_2(N)$ given by $\sigma \mapsto \iota(\sigma_N)$, where $\sigma_N$ is the automorphism of $E[N](\kbar)$ induced by $\sigma$.
% and also define $\rho_{E,N}\colon G_k \to \GL_2(N)$ as $\rho_{E,N}(\sigma) \coloneqq \iota(\sigma)$, with $\sigma$ acting on $E[N](\kbar)$.

\begin{remark}
Our choice of $\iota$ corresponds to the left action of $\GL_2(N)$ on $\Z(N)^2$. Many sources are ambiguous about the choice of left vs right action; this ambiguity often makes no difference (see \cite{RZB}*{Remark 2.2}), but it does when $H\le \GL_2(N)$ is not conjugate to its transpose. Our choice here is consistent with \cites{sutherland:Computing-images-of-Galois-representations-attached-to-elliptic-curves,SutherlandZ:Modular-curves-of-prime-power-level-with-infinitely-many-rational-points}, but differs from \cite{RZB}, which uses right actions.
\end{remark}

For each elliptic curve $E/k$ we fix a system of compatible bases for $E[N](\kbar)$ for all $N\ge 1$ coprime to the characteristic of $k$, meaning that if $N=N_1N_2$ then the bases for $E[N_1](\kbar)$ and $E[N_2](\kbar)$ are the images of the basis for $E[N](\kbar)$ under multiplication by $N_2$ and $N_1$, respectively.  
Let $\iota_E\colon\! \Aut(E[N](\kbar)) \overset\sim\longrightarrow \GL_2(N)$ denote the isomorphism determined by the basis for $E[N](\kbar)$, and define the Galois representation $\rho_{E,N}\colon G_k\to \GL_2(N)$ via $\rho_{E,N}(\sigma) \coloneqq \iota_E(\sigma_N)$.  When $k$ has characteristic zero we shall also use $\iota_E$ to denote the isomorphism $\iota_E\colon \Aut(E_{\rm tor}(\kbar)) \overset\sim\longrightarrow \GL_2(\Zhat)$ determined by this system of bases, and use it to define the Galois representation $\rho_E\colon G_k\to \GL_2(\Zhat)$ via $\rho_E(\sigma) \coloneqq \iota_E(\sigma)$.  If $k$ is a number field and $\pp$ is a prime of good reduction for $E$, for each $N\ge 1$ coprime to $N(\pp)$ we shall use the reduction of our chosen basis for $E[N](\kbar)$ as a basis for $E_\pp[N](\overline\F_\pp)$.

For a subgroup $H \leq \GL_2(\Zhat)$ we use $H(N)$ to denote the image $\pi_N(H)$ of $H$ under $\pi_N$. An open subgroup $H \leq \GL_2(\Zhat)$ contains the kernel of $\pi_N\colon \GL_2(\Zhat)\to\GL_2(N)$ for some $N\ge 1$;
the least such $N$ is the \defi{level} of $H$, in which case $H = \pi_N^{-1}(H(N))$ is completely determined by $\pi_N(H)$.  We will typically work directly with subgroups of $\GL_2(N)$ instead of $\GL_2(\Zhat)$, with the understanding that every subgroup of $\GL_2(N)$ uniquely determines an open subgroup of $\GL_2(\Zhat)$ as its inverse image under $\pi_N$; we may use the symbol $H$ to denote both $H$ and $\pi_N(H)$, which we note have the same \defi{index} $[\GL_2(\Zhat):H]=[\GL_2(N):\pi_N(H)]$.  We define the \defi{level} of a subgroup $H \leq \GL_2(N)$ to be the level of $\pi_N^{-1}(H)$.

We may also identify open subgroups of $\GL_2(\Z_\ell)$ with their preimages in $\GL_2(\Zhat)$ under the projection map $\GL_2(\Zhat)\to \GL_2(\Z_\ell)$.  For any $e\ge 0$ we may identify $H\le \GL_2(\ell^e)$ with the corresponding open subgroups of $\GL_2(\Z_\ell)$ or $\GL_2(\Zhat)$ of the same index.  In what follows we shall only be interested in subgroups of $\GL_2(\Zhat)$ up to conjugacy and view each subgroup~$H$ as a representative of its conjugacy class.  All inclusions $H_1\le H_2$ of subgroups of $\GL_2(\Zhat)$ should be understood to indicate that $H_1$ is conjugate to a subgroup of $H_2$.

%----------------------------------------------------------------------------
\subsection{Modular curves}
\label{ssec:modular-curves}
%----------------------------------------------------------------------------

Let $H$ be a subgroup of $\GL_2(N)$ of level $N$. We define the modular curve $Y_H$ (resp.~$X_H$) to be the coarse moduli space of the stack $\mathcal{M}^0_H$ (resp.~$\mathcal{M}_H$), over $\Spec \Z[1/N]$, that parameterizes elliptic curves (resp.~generalized elliptic curves) with $H$-level structure. Equivalently, by \cite{DeligneRapoport}*{IV-3.1} (see also \cite{RZB}*{Lemma 2.1}), $X_H$ is isomorphic to the coarse space of the stack quotient $X(N)/H$, where $X(N)$ is the classical modular curve parameterizing full level structures.

More precisely, an $H$-\defi{level structure} on an elliptic curve $E/\kbar$ is an equivalence class $[\iota]_H$ of isomorphisms $\iota\colon E[N](\kbar)\to \Z(N)^2$, where $\iota\sim \iota'$ if $\iota = h\circ \iota'$ for some $h\in H$. The set $Y_H(\kbar)$ consists of equivalence classes of pairs $(E,[\iota]_H)$, where $(E,[\iota]_H)\sim (E',[\iota']_H)$ if there is an isomorphism $\phi\colon E\to E'$ for which
the induced isomorphism $\phi_N\colon E[N]\to E'[N]$ satisfies $\iota\sim\iota'\circ \phi_N$. Equivalently, $Y_H(\kbar)$ consists of pairs $(j(E),\alpha)$, where $\alpha=HgA_E$ is a double coset in $H\backslash\GL_2(N)/A_E$, where $A_E\coloneqq \{\varphi_N:\varphi\in \Aut(E_\kbar)\}$ with $\varphi_N\coloneqq \iota_E(\varphi_{|_{E[N]}})$. 

Each $\sigma\in G_k$ induces an isomorphism $E^\sigma[N]\to E[N]$ defined by $P \mapsto \sigma^{-1}(P)$, which we denote $\sigma^{-1}$.  We define a right $G_k$-action on $Y_H(\kbar)$ via $(E,[\iota]_H)\mapsto (E^\sigma,[\iota\circ \sigma^{-1}]_H)$.
The set of $k$-rational points $Y_H(k)$ consists of the elements in $Y_H(\kbar)$ that are fixed by this $G_k$-action.  Each $P\in Y_H(k)$ is represented by a pair $(E,[\iota]_H)\in Y_H(k)$ such that $E$ is defined over $k$ and for all $\sigma\in G_k$ there exist $\varphi\in \Aut(E_{\kbar})$ and $h\in H$ such that
\begin{equation}\label{eq:iota}
\iota\circ\sigma^{-1} = h\circ \iota\circ \varphi_N,
\end{equation}
as shown in \cite{zywina:Possible-indices-for-the-Galois-image-of-elliptic-curves-over-Q-arxiv}*{Lemma~3.1}.\footnote{The running assumption $-I\in H$ in \cite{zywina:Possible-indices-for-the-Galois-image-of-elliptic-curves-over-Q-arxiv} is not used in the proof of \cite{zywina:Possible-indices-for-the-Galois-image-of-elliptic-curves-over-Q-arxiv}*{Lemma~3.1.}}
Equivalently, a pair $(j(E),\alpha)$ with $\alpha=HgA_E$ lies in the set $Y_H(k)$ if and only if we have $j(E)\in k$ and  $Hg\rho_{E,N}(\sigma) A_E=HgA_E$ for all $\sigma\in G_k$.

\begin{remark}
For an elliptic curve $E/k$, if $\rho_{E,N}(G_k)\le H$ then there exists an isomorphism $\iota\colon E[N]\overset\sim\rightarrow \Z(N)^2$ for which $(E, [\iota]_H)\in Y_H(k)$.
Conversely, assuming $\textrm{char}(k)\ne 2,3$, if $(E, [\iota]_H)\in Y_H(k)$, then for every twist $E'$ of $E$ there is an isomorphism $\iota'\colon E'[N]\overset\sim\rightarrow \Z(N)^2$ with $(E', [\iota']_H)\in Y_H(k)$, and for at least one $E'$ we have $\rho_{E',N}(G_k)\le H$, as we now show.

If $E'$ is a twist of $E$ then there is an isomorphism $\phi\colon E'_{\bar k}\overset{\sim}{\to}E_{\bar k}$ that induces an isomorphism $\phi_N\colon E'[N](\kbar)\overset{\sim}{\to}E[N](\kbar)$, and for $\iota'\coloneqq \iota\circ\phi_N$ the pairs $(E,[\iota]_H)$ and $(E',[\iota']_H)$ represent the same point in $Y_H(\kbar)$, thus if $(E,[\iota]_H)$ lies in $Y_H(k)$, then so does $(E',[\iota']_H)$.

If the point $(E,[\iota]_H)=(j(E),\alpha)$ lies in $Y_H(k)$, say for $\alpha=HgA_E$, then $\rho_{E,N}(G_k)$ must lie in $g^{-1}HgA_E$.  In other words, up to conjugacy in $\GL_2(N)$ the subgroups $\rho_{E,N}(G_k)$ and~$H$ can differ only by elements of $A_E$, which for $\textrm{char}(k)\ne 2,3$ is cyclic of order 2, 4, 6, the last two occurring only for $j(E)=0,1728$.  When $\#A_E=2$ every twist of $E$ is a quadratic twist and \cite{sutherland:Computing-images-of-Galois-representations-attached-to-elliptic-curves}*{Lemma~5.24} implies that there is a quadratic twist $E'$ of $E$ for which $\rho_{E',N}(G_k)\le H$.  Otherwise $E$ has potential CM and we may apply Proposition~\ref{prop:CMtwist}, which also addresses quartic and sextic twists.  These results assume $k$ is a number field, but the arguments apply whenever $A_E$ is cyclic (including $\textrm{char}(k)= 2,3$ provided $j(E)\ne 0,1728$).
\end{remark}

\begin{remark}
The elliptic curves
\[
E\colon \href{https://www.lmfdb.org/EllipticCurve/Q/144/a/3}{y^2 = x^3 - 1}\qquad\text{and}\qquad E' \colon \href{https://www.lmfdb.org/EllipticCurve/Q/1728/n/4}{y^2 = x^3 + 2}
\]
with $j(E)=j(E') = 0$ are cubic twists that correspond to the same point in $X_0(2)(\Q)$, but~$E$ has a rational point of order 2 and $\rho_{E,2}(G_\Q) = B(2)$, while $E'$ has no rational points of order 2 and $\rho_{E',2}(G_\Q) = \GL_2(2)$. In particular, \cite{RZB}*{Lemma 2.1} is incorrectly stated, it holds only up to twist.
\end{remark}

The complement $X_H^{\infty} \coloneqq X_H - Y_H$ is the set of cusps of $X_H$, and corresponds to generalized elliptic curves with $H$-level structure. The set of $k$-rational cusps $X_H^{\infty}(k)$ can be alternatively described as follows. Let $U(N)\coloneqq\langle\smallmat{1}{1}{0}{1},-I\rangle\subseteq \GL_2(N)$. We define a right $G_k$-action on $H\backslash \GL_2(N)/U(N)$ via $HgU(N)\mapsto Hg\chi_N(\sigma)U(N)$, where $\chi_N(\sigma)\coloneqq\smallmat{e}{0}{0}{1}$ is defined by $\sigma(\zeta_N)=\zeta_N^e$. By \cite{DeligneRapoport}*{VI-5.3}, $X_H^\infty(k)$ is in bijection with the subset of $H\backslash \GL_2(N)/U(N)$ fixed by $\chi_N(G_k)$.

We define the \defi{genus} $g(H)$ of $H$ and $\pi_N^{-1}(H)$ to be the genus of each of the geometric connected components of the modular curve $X_H$, which can be directly computed from $H$ via the usual formula
\[
g(H) = g(\Gamma_H) = 1+\frac{i(\Gamma_H)}{12} - \frac{\nu_2(\Gamma_H)}{4} - \frac{\nu_3(\Gamma_H)}{3} - \frac{\nu_\infty(\Gamma_H)}{2},
\]
if we take $\Gamma_H\coloneqq \pm H\cap \SL_2(N)$, let $i(\Gamma_H)=[\SL_2(N):\Gamma_H]$, let $\nu_2$ (resp.~$\nu_3$) count the cosets in $\Gamma_H\backslash\SL_2(N)$ fixed by the right action of $\smallmat{0}{1}{-1}{0}$ (resp.~$\smallmat{0}{1}{-1}{-1}$), and let $\nu_{\infty}(\Gamma_H)$ count the orbits of $\Gamma_H\backslash \SL_2(N)$ under the right action of $\smallmat{1}{1}{0}{1}$. The function \magmacmd{GL2Genus} in the file \repolink{groups}{gl2.m} in \cite{RouseSZB:magma-scripts-ell-adic-images} implements this computation.

We define the \defi{index} $i(H)$ of $H$ to be the integer $[\GL_2(\Zhat)\!:\!H]=[\GL_2(N)\!:\!H(N)]$.

\begin{remark}
  For $H\le \GL_2(\Zhat)$ of level $N$ the inverse image $\Gamma$ of $\Gamma_H\leq \SL_2(N)$ in $\SL_2(\Z)$ is a congruence subgroup of some level $M|N$, but $M$ may be strictly smaller than $N$; see \cite{SutherlandZ:Modular-curves-of-prime-power-level-with-infinitely-many-rational-points} for examples.  In any case, the base change of $X_H$ to $\Q(\zeta_N)$ breaks into connected components which are geometrically connected and are each isomorphic over $\C$ to the modular curve $X_\Gamma$ obtained by taking the quotient of the extended upper half plane by $\Gamma$.
\end{remark}

\begin{remark}
\label{remark:disconnected}
For $H\le \GL_2(N)$, if $\det H = \Z(N)^\times$, then $X_H$ is geometrically connected. We mostly consider the case where $\det H = \Z(N)^\times$ in this paper, though in Appendix \ref{S:kolyvagin} we make use of the case $H = \{I\}$, where $X_H = X(N)$. When $\det H < \Z(N)^\times$ is a proper subgroup, $X_H$ is still a curve defined over $\Q$, but it is not geometrically connected; its connected components are defined over $\Q(\zeta_N)^H$, indexed (via the Weil pairing on the level structures) by cosets of $\det H$, and isomorphic (over $\Q(\zeta_N)^H$) to $X_{\Gamma}$. See Appendix \ref{S:kolyvagin} where this is worked out in more detail for $X(N)$ and Example \ref{example:non-surjective-determinant} where we discuss the isogeny decomposition of the ``Jacobian'' of such an $X_H$.
\end{remark}

%----------------------------------------------------------------------------
\subsection{Subgroup labels}
\label{ssec:subgroup-labels}
%----------------------------------------------------------------------------

To help organize our work we introduce the following labeling convention for uniquely identifying (conjugacy classes of) open subgroups $H$ of $\GL_2(\Zhat)$, which will also serve as labels of the corresponding modular curves $X_H$.  All the subgroups we shall consider satisfy $\det(H)=\Zhat^\times$ and have labels of the form
\[
\texttt{N.i.g.n}
\]
where \texttt{N}, \texttt{i}, \texttt{g}, \texttt{n} are the decimal representations of integers $N,i,g,n$ defined as follows:
\begin{itemize}\setlength{\itemsep}{2pt}
\item $N=N(H)$ is the level of $H$;
\item $i=i(H)$ is the index of $H$;
\item $g=g(H)$ is the genus of $H$;
\item $n\ge 1$ is an ordinal indicating the position of $H$ among all subgroups of the same level, index, and genus with $\det(H)=\Zhat^\times$, under the ordering defined below.
\end{itemize}
To define the ordinal $n$ we must totally order the set $S(N,i,g)$ of open subgroups of $\GL_2(\Zhat)$ of level $N$, index $i$, and genus $g$ with $\det(H)=\Zhat^\times$, and we want this ordering to be efficiently computable.  With this in mind, we define the following  invariants for each $H\in S(N,i,g)$:

\begin{itemize}\setlength{\itemsep}{2pt}
\item the \defi{parent list} of $H$ is the lexicographically ordered list of labels of the subgroups $G\leq \GL_2(\Zhat)$ of which $H$ is a maximal subgroup;
\item the \defi{orbit signature} of $H$ as the ordered list of triples $(e,s,m)$ where $m$ counts orbits of $\Z(N)^2$ under the left-action of $H$ of size $s$ and exponent $e$ (the least integer that kills every element of the orbit);\footnote{\magma efficiently computes the orbits of $\Z(N)^2$ under the right action of $H$; one applies this function to the transpose of $H$ to efficiently compute orbits under the left action of $H$.}
\item the \defi{class signature} of $H$ is the ordered set of tuples $(o,d,t,s,m)$ where $m$ counts the conjugacy classes of elements of~$H\le \GL_2(N)$ of size~$s$ whose elements have order~$o$, determinant~$d$, and trace~$t$; 
\item the \defi{minimal conjugate} of $H$ is the least subgroup of $\GL_2(N)$ conjugate to $H$, where subgroups of $\GL_2(N)$ are ordered lexicographically as ordered lists of tuples of integers $(a,b,c,d)$ with $a,b,c,d\in [0,N-1]$ corresponding to the matrix $\smallmat{a}{b}{c}{d}\in H$.
\end{itemize}
We now totally order the set $S(N,i,g,d)$ lexicographically by parent list, orbit signature, class signature, and minimal conjugate; these invariants are intentionally ordered according to the difficulty of computing them, and in most cases the first two or three suffice to distinguish every element of $S(N,i,g,d)$, meaning that we rarely need to compute minimal conjugates, which is by far the most expensive invariant to compute.

\begin{example}\label{example:label}
Let $H=N_{\ns}(11^2)$ be the normalizer of the nonsplit Cartan subgroup of $\Z(11^2)$, which has level $11^2=121$ and index $5\cdot 11^3 = 6655$.
There are just two subgroups $H\le \GL_2(11^2)$ of index 6655 (up to conjugacy), only one of which has $\det(H)=\Zhat^\times$, so computing $g(H)=511$ suffices to determine the label \texttt{121.6655.511.1} of $N_{\ns}(11^2)$.

If we instead take $H=N_{\ns}(11)$ with level 11 and index 55, we find there are two subgroups of $\GL_2(11)$ of index 55; they have genus 1 and determinant index 1 and thus  belong to the set $S(11,55,1,1)$. These subgroups $H_1\coloneqq \langle \smallmat{9}{1}{2}{9},\smallmat{7}{0}{0}{4}\rangle$ and $H_2\coloneqq \langle \smallmat{5}{10}{9}{2},\smallmat{3}{7}{10}{10}\rangle$ have parent list $(\texttt{1.1.0.1})$ and orbit signature $(1,1,1), (11,120,1)$.  Their class signatures begin $(1,1,2,1,1), (2,1,9,1,1), (2,10,0,12,1)$ but differ in the fourth tuple, which is $(3,1,10,2,1)$ for $H_1$ but $(3,1,10,8,1)$ for $H_2$, meaning that $H_1$ contains a conjugacy class of size 2 whose elements have order 3, determinant 1, and trace 10, but $H_2$ does not.  It follows that the labels of $H_1$ and $H_2$ are \texttt{11.55.1.1} and \glink{11.55.1.2}, and that $N_{\ns}(11)=H_1$.
\end{example}

\begin{remark}
This labeling system can be extended to arbitrary open $H\le \GL_2(\Zhat)$ using labels of the form \texttt{M.d.m-N.i.g.n}, where \texttt{M.d.m} uniquely identifies $\det(H)\le \Zhat^\times$ according to its level $M|N$, index $d:=[\Zhat^\times\!\!:\!\det(H)]$, and an ordinal $m$ that distinguishes $\det(H)$ among the open subgroups $D\le \Zhat^\times$ of level $M$ and index $d$.\footnote{These can be viewed as labels of the finite abelian extensions $K$ of $\Q$ by identifying $D$ with the subgroup of $(\Z/M\Z)^\times \simeq \{\zeta_M\mapsto \zeta_M^a:a\in (\Z/M\Z)^\times\}\simeq \Gal(\Q(\zeta_M)/\Q)$ whose fixed field is $K$, where $M=\cond(K)$.}  To define $m$ we order the index~$d$ subgroups $D\le (\Z/M\Z)^\times$ lexicographically according to the lists of labels $\texttt{M.a}$ of Conrey characters $\chi_M(a,\cdot)$ of modulus $M$ whose kernels contain $D$; see \cite{BestBBCCDLLRSV:Computing-classical-modular-forms}*{\S 3.2} for the definition of $\chi_M(a,\cdot)$.  We then redefine~$N$ and~$i$ to be the relative level and index of~$H$ as a subgroup of the preimage $G$ of $\det(H)$ in $\GL_2(\Zhat)$: this means that $N$ is the least positive integer for which $H=\pi_N^{-1}(\pi_N(H))$, where $\pi_N\colon G\to G(N)$ is the reduction map, and $i=[G:H]$. This convention ensures that if $E/K$ has Galois image $\rho_E(G_K)=H$, the values of $N$, $i$, $g$ in the label of $\rho_E(G_L)$  will be the same for every finite extension $L/K$.
\end{remark}

% See \cite[Table 1]{zywina:On-the-possible-images-of-the-mod-ell-representations-associated-to-elliptic-curves-over-Q-arxiv} for a list of CM elliptic curves.

%----------------------------------------------------------------------------
\section{Arithmetically maximal subgroups}
\label{sec:group-theory}
%----------------------------------------------------------------------------
\setcounter{subsection}{1}

\begin{definition}
\label{defn:arithmetically-maximal}
We define an open subgroup $H \le \GL_2(\Zhat)$ to be \defi{arithmetically maximal} if
\begin{enumerate}[{\rm(i)}]
\item $\det(H)=\Zhat^\times$,\label{item:surjdet}
\item $H$ contains an element conjugate to $\smallmat{1}{0}{0}{-1}$ or $\smallmat{1}{1}{0}{-1}$,\label{item:cc}
\item $j(X_H(\Q))$ is finite and $j(X_K(\Q))$ is infinite for all $H< K\le \GL_2(\Zhat)$.\label{item:maximal}
\end{enumerate}
\end{definition}
Properties of the Weil pairing imply that if $H = \rho_E(G_{\Q})$ then $\det(H)= \Z^{\times}$, and more generally, that if $E$ is an elliptic curve over a number field $K$ and $H=\rho_E(G_K)$, then $[\Z^\times:\det(H)]=[K:K\cap \Q^{\rm cyc}]$.  In particular, \eqref{item:surjdet} holds for every $H$ that arises as $\rho_E(G_\Q)$ for an elliptic curve $E/\Q$.
Condition \eqref{item:cc} is a necessary and sufficient condition for $X_{H}$ to have non-cuspidal real points, by Proposition 3.5 of \cite{zywina:Possible-indices-for-the-Galois-image-of-elliptic-curves-over-Q-arxiv}, and it follows from Proposition~3.1 of \cite{SutherlandZ:Modular-curves-of-prime-power-level-with-infinitely-many-rational-points} that every genus zero $H$ of prime power level $N=\ell^e$ that satisfies \eqref{item:surjdet} and \eqref{item:cc} has infinitely many $\Q$-points (the assumption $-I\in H$ is not used in the proof of Proposition~3.1 in \cite{SutherlandZ:Modular-curves-of-prime-power-level-with-infinitely-many-rational-points}). For odd $\ell$ the matrices in \eqref{item:cc} are conjugate in $\GL_2(N)$ and only the first is needed.  This definition differs slightly from that used in \cite{RZB}*{Definition 3.1} for $\ell=2$.

Let $\calS_\ell(\Q)$ be the set of arithmetically maximal subgroups of $\ell$-power level, and let $\calS_\ell^\infty(\Q)$ be the set of open subgroups of $\GL_2(\Zhat)$ of $\ell$-power level for which $j(X_H(\Q))$ is infinite.

The sets $\calS_\ell^\infty(\Q)$ are explicitly determined in \cite{SutherlandZ:Modular-curves-of-prime-power-level-with-infinitely-many-rational-points} (and in \cite{RZB} for $\ell=2$) and have cardinalities $1208,47,23,15,2,11$ for $\ell=2,3,5,7,11,13$, and cardinality $1$ for all $\ell >13$.\footnote{Seven of the $H_0\in S_2^\infty(\Q)$ never arise as $\rho_{E,2^\infty}(G_\Q)$ because $\rho_{E,2^\infty}(G_\Q)$ is strictly contained in $H_0$ for all the elliptic curves $E$ that correspond to points in $X_{H_0}(\Q)$; see \cite{RZB}*{Remark 6.3}.  This explains the discrepancy between $\#S_2^\infty(\Q)=1208$ and the figure 1201 that appears in \cite{SutherlandZ:Modular-curves-of-prime-power-level-with-infinitely-many-rational-points}*{Corollary 1.6}.}
To enumerate the set $\calS_\ell(\Q)$ it would suffice to enumerate the maximal subgroups of all the groups in $\calS_\ell^\infty(\Q)$, but we take a slightly different approach and use our knowledge of $\calS_\ell^\infty(\Q)$ to derive an upper bound on the level of the groups in  $\calS_\ell(\Q)$.

\begin{lemma}\label{L:level-bound}
Let $e_\ell\coloneqq 5,3,2,1,1,1$ for $\ell=2,3,5,7,11,13$, respectively, and let $e_\ell\coloneqq 0$ for all primes $\ell>13$.
Then $\ell^{e_\ell}$ is an upper bound on the level of every $H\in \calS_\ell^\infty(\Q)$ and $\ell^{e_\ell+1}$ is an upper bound on the level of every $H\in \calS_\ell(\Q)$.
\end{lemma}
\begin{proof}
The bound for $H\in \calS_\ell^\infty(\Q)$ follows from \cite{SutherlandZ:Modular-curves-of-prime-power-level-with-infinitely-many-rational-points}*{Corollary 1.6} and its associated data.  The bound for $H\in\calS_\ell(\Q)$ follows from \cite{SutherlandZ:Modular-curves-of-prime-power-level-with-infinitely-many-rational-points}*{Lemma 3.7} for $\ell \le 13$ (note~$e_2\ge 2$), and for $\ell>13$ we can apply \cite{serre:abelianladic}*{Lemma 3, page IV-23}, since $e_\ell=0$.
\end{proof}

Lemma~\ref{L:level-bound} reduces the problem of enumerating the sets $\calS_\ell(\Q)$ to a finite computation in $\GL_2(\ell^{e_\ell+1})$.
This computation can be further constrained by computing the maximal index~$i_\ell$ of any maximal subgroup of $H_0\in \calS_\ell^\infty(\Q)$ and then enumerating all (conjugacy classes of) subgroups of $\GL_2(\ell^{e_\ell+1})$ of index at most $i_\ell$, which can be efficiently accomplished using algorithms for subgroup enumeration that are implemented in \magma and other computer algebra systems.  The advantages of this approach are that it enumerates subgroups up to conjugacy in $\GL_2(\Zhat)$, rather than up to conjugacy in $H_0$, and it includes all subgroups that have the same level and index of any $H\in \calS_\ell(\Q)$, as well as all of the groups that contain them, which may be needed to compute the group labels defined in the previous section.

Table~\ref{table:amstats} summarizes this computation.  The ``subgroups'' row counts the number of subgroups $H\le \GL_2(\Zhat)$ of $\ell$-power level with $\det(H)=\Zhat^\times$ that satisfy the level and index bounds in the first two rows, and the last three rows list the maximum level, index, and genus that arise among the groups $H\in \calS_\ell(\Q)$.
\medskip

\begin{table}[htp!]
\begin{tabular}{r|rrrrrrrrrrrr}\toprule
$\ell$ & 2 & 3 & 5 & 7 & 11 & 13 & 17 & 19 & 23 & 29 & 31 & 37\\
level bound & 64 & 81 & 125 & 49 & 121 & 169 & 17  & 19  & 23  & 29  & 31 & 37\\
index bound & 192 & 729 & 625 & 1372 & 6655 & 728 & 153 & 285 & 276 & 1015 & 496 & 2109\\
subgroups & 11091 & 469 & 111 & 144 & 141 & 54 & 18 & 25 & 17 & 64 &45 & 100\\
$\#\calS_\ell^\infty(\Q)$ & 1208 & 47 & 23 & 15 &2 & 11 & 1 & 1 & 1 & 1 & 1 & 1\\
$\#\calS_\ell(\Q)$ & 130 & 19 & 14 & 10 & 6 & 10 & 3 & 4 & 3 & 4 & 3 & 4\\
max level & 32 & 27 & 125 & 49 & 121 & 169 & 17 & 19 & 23 & 29 & 31 & 37\\
max index & 96 & 729 & 625 & 1372 & 6655 & 182 & 153 & 285 & 276 & 1015 & 496 & 2109\\
max genus & 7 & 43 & 36 & 94 & 511 & 3 & 7 & 14 & 15 & 63 & 30 & 142\\
\bottomrule
\end{tabular}
\medskip

\caption{Summary of arithmetically maximal $H\leq\GL_2(\Zhat)$ of $\ell$-power level for $\ell\le 37$.  The level and index bounds are those implied by Lemma~\ref{L:level-bound} and the paragraph following, the max level, index, genus rows are the maximum values realized by arithmetically maximal groups.} \label{table:amstats}
\end{table}
\vspace{-20pt}

%----------------------------------------------------------------------------
\section{Bookkeeping}
\label{sec:bookkeeping}
%----------------------------------------------------------------------------
\setcounter{subsection}{1}

By \cite{Mazur:isogenies}, \cite{serre:quelquescheb}*{Section 8.4}, and \cite{biluP:Serres-uniformity-problem-in-the-split-Cartan-case}, if $\ell \geq 17$ and $\rho_{E,\ell}$ is not surjective, then either $\ell = 17$ or $37$ and $j(E)$ is listed in Table~\ref{table:exceptional-points}, or the image $\rho_{E,\ell}(G_\Q)$ is a subgroup of~$N_{\ns}(\ell)$. In the latter case,
for $\ell\equiv 1\pmod{3}$ the image must be equal to $N_{\ns}(\ell)$ (not a proper subgroup), and for $\ell\equiv 2\pmod{3}$ the image is either $N_{\ns}(\ell)$ or the index-3 subgroup
\[
G:=\big\{a^3: a \in C_{\ns}(\ell)\big\} \cup \big\{ \left(\begin{smallmatrix}1 &0  \\0 & -1 \end{smallmatrix}\right) \cdot a^3: a \in C_{\ns}(\ell) \big\},
\]
by \cite{zywina:On-the-possible-images-of-the-mod-ell-representations-associated-to-elliptic-curves-over-Q-arxiv}*{Proposition 1.13}. Recent work by le Fourn and Lemos \cite{leFournL:Residual-Galois-representations-of-elliptic-curves-with-image-contained-in-the-normaliser-of-a-nonsplit-Cartan}*{Theorem 1.2} shows that if $\ell > 1.4\times10^7$ then $\rho_{E,\ell}(G_\Q)$ cannot be a proper subgroup of $N_{\ns}(\ell)$.

Numerous individual cases remain. See Tables~\ref{table:groups3}-\ref{table:groups2} in Appendix \ref{S:tables} for a list of all arithmetically maximal $H\le \GL_2(\Zhat)$ of $\ell$-power level for $\ell \leq 37$, and see Table~\ref{table:amstats} for a numerical summary of this data. Many of the corresponding modular curves $X_H$ are already handled in the literature, and these account for most of the exceptional points listed in Table~\ref{table:exceptional-points}.  Below is a summary of prior results.

\begin{enumerate}\setlength{\itemsep}{3pt}
\item Ligozat determined the rational points on $X_0(N)$ for $N = 27,49,11,17,19$ with labels \glink{27.36.1.1}, \glink{49.56.1.1}, \glink{11.12.1.1}, \glink{17.18.1.1}, \glink{19.20.1.1} in \cite{ligozat:Courbes-modulaires-de-genre-1}*{5.2.3.1};  there are exceptional points for $N = 11,17$.

\item Ligozat addressed $X_{S_4}(11)$ with label \glink{11.55.1.2} in \cite{ligozat:courbes-modulaires-de-niveau-11}*{Prop.~4.4.8.1}.

\item Kenku determined the rational points on $X_0(13^2)$ (labeled \glink{169.182.8.1}) \cite{kenku:The-modular-curve-X0169-and-rational-isogeny}, and on $X_0(125)$ (labeled \glink{125.150.8.1}) \cite{kenku-On-the-modular-curves-X0125-X125-and-X149}; the modular curve $X_0(5,25)$ with label \glink{25.150.8.1} is isomorphic to $X_0(125)$ and thus also handled by Kenku.

\item The curves $X_{\spc}^+(25)$ and $X_{\spc}^+(49)$ have labels \glink{25.375.22.1} and \glink{49.1372.94.1}; they are isomorphic to $X_0^+(25^2)$ and $X_0^+(49^2)$, and are thus handled by Momose and Shimura \cite{momoseS:lifting-of-supersingular-points-on-X0pr-and-lower-bound-of-ramification-index}*{Theorems 0.1 and 3.14} and by Momose \cite{momose:rational-points-on-the-modular-curves-X0pluspr}*{Theorem 3.6}.
% and thus handled by Bilu--Parent--Rebolledo \cite{bilupr:rationalpoints}*{Theorem 1.1}.

\item Momose also handled $X_{\spc}^+(11)$ with label \glink{11.66.2.1} in \cite{momose:rational-points-on-the-modular-curves-Xsplitp}*{Theorem 0.1}.
  
\item The curves $X_{\ns}^+(13)$ and $X_{\spc}^+(13)$ labeled \glink{13.78.3.1} and \glink{13.91.3.1} are handled via nonabelian Chabauty by Balakrishnan et al.~in \cite{balakrishnanDMTV:cartan}.
  
\item The curves $X_{S_4}(13)$ and $X_{\ns}^+(17)$ labeled  \glink{13.91.3.2} and \glink{17.136.6.1}  are handled via nonabelian Chabauty by Balakrishnan et al.~in \cite{balakrishnanDMTV:Quadratic-Chabauty-For-Modular-Curves:-Algorithms-And-Examples}.

\item Rouse and Zureick-Brown addressed all $X_H$ of $2$-power level in \cite{RZB}.
\item The group with label \glink{7.42.1.1} is the index 2 subgroup of $N_{\spc}(7)$ containing the elements of $C_{\spc}(7)$ with square determinant and the elements of $N_{\spc}(7) - C_{\spc}(7)$ with non-square determinant, handled by  \cite{sutherland:A-local-global-principle-for-rational-isogenies-of-prime-degree}*{Section 3} and \cite{zywina:On-the-possible-images-of-the-mod-ell-representations-associated-to-elliptic-curves-over-Q-arxiv}*{Remark 4.3}. There is one exceptional $j$-invariant, which is interesting because elliptic curves with this mod 7 image admit a rational 7-isogeny locally everywhere but not globally; this is the unique example of this phenomenon over $\Q$ for any prime $\ell$.

\item Zywina addressed all remaining curves of prime level $\ell \leq 37$ except $N_{\ns}(\ell)$ in \cite{zywina:On-the-possible-images-of-the-mod-ell-representations-associated-to-elliptic-curves-over-Q-arxiv}.

\item Elkies showed that the curve \glink{9.81.1.1} has no rational points in \cite{ElkiesSurj3}*{page 5}.

\item Non-cuspidal points on the modular curves labeled \glink{49.168.12.1} and \glink{49.168.12.2} are ``$7$-exceptional'' in the sense of \cite{greenbergRSS:On-elliptic-curves-with-an-isogeny-of-degree-7}*{Definition 1.1}.\footnote{These two curves have genus 12, but are not the genus 12 curve studied in \cite{greenbergRSS:On-elliptic-curves-with-an-isogeny-of-degree-7}*{Theorem 5.4}.}
By \cite{greenbergRSS:On-elliptic-curves-with-an-isogeny-of-degree-7}*{page 3} the only $7$-exceptional elliptic curves are CM curves with $j = -15^3$ and $j = 255^3$, which give rise to two points on the curve with label \glink{49.168.12.2}.
% (\glink{49.168.12.1} also has no $\F_p$ points.)

\item The groups with labels \glink{25.150.4.1}, \glink{25.150.4.2}, \glink{25.150.4.5}, and \glink{25.150.4.6} correspond to elliptic curves $E/\Q$ that admit a rational $5$-isogeny with $5$-adic image of index divisible by $25$; their existence is ruled out by \cite{greenberg:The-image-of-Galois-representations-attached-to-elliptic-curves-with-an-isogeny}*{Theorem 2}.
% \jeremy{There's one other genus $4$ $5$-adic curve, namely \glink{25.150.4.4}, but it's not arithmetically maximal, since it maps to \glink{25.75.2.1}.}

\item The genus 2 rank 2 curves \glink{25.50.2.1} and  \glink{25.75.2.1} are handled by Balakrishnan et al. in \cite{balakrishnanDMTV:Quadratic-Chabauty-For-Modular-Curves:-Algorithms-And-Examples}*{4.3}. Each has a single exceptional $j$-invariant, and the minimal conductors for the two $j$-invariants are 396900 and 21175 (respectively).
\end{enumerate}
\medskip

Of the 210 arithmetically maximal subgroups in Table~\ref{table:amstats-maximal}, only 32 are not addressed above.  Five are the groups $N_{\rm ns}(\ell)$ for $19\le \ell \le 37$; the remaining 27 are listed in Table~\ref{table:cases}.
\smallskip

\begin{table}[bth!]
\begin{tabular}{cll}
genus & subgroups & section\\\toprule
1 & \glink{9.12.1.1}, \glink{9.54.1.1}, \glink{27.36.1.2} & \S\ref{ssec:genus-1}\\
2 & \glink{9.36.2.1}, \glink{9.54.2.2}, \glink{27.36.2.1}, \glink{27.36.2.2} & \S\ref{ssec:genus-2}\\
3 & \glink{9.108.3.1}, \glink{27.36.3.1}, \glink{27.108.3.1} & \S\ref{ssec:genus-3}\\
4 & \glink{9.108.4.1}, \glink{27.108.4.1}, \glink{27.108.4.3}, \glink{27.108.4.5} & \S\ref{ssec:genus-4}\\
&\glink{25.150.4.7}, \glink{25.150.4.8}, \glink{25.150.4.9} & \S\ref{ssec:pointless}\\
6 & \glink{27.108.6.1} & \S\ref{ssec:genus-6}\\
9 & \glink{49.147.9.1}*, \glink{49.196.9.1}* & \S\ref{ssec:genus9}\\
12 & \glink{27.243.12.1}* & \S\ref{ssec:ns27}\\
14 & \glink{25.250.14.1}* & \S\ref{ssec:nssquare}\\
36 & \glink{25.625.36.1} & \S\ref{ssec:genus-36-43}\\
41 & \glink{121.605.41.1} & \S\ref{ssec:genus-41}\\
43 & \glink{27.729.43.1} & \S\ref{ssec:genus-36-43}\\
69 & \glink{49.1029.69.1}* & \S\ref{ssec:nssquare}\\
511 & \glink{121.6655.511.1}* & \S\ref{ssec:nssquare}\\\bottomrule
\end{tabular}
\medskip

\caption{The arithmetically maximal subgroups of $\ell$-power level for $\ell\le 17$ not addressed by previous results.  Subgroups marked with asterisks are the subgroups $H$ listed in Table~\ref{table:remaining-cases} for which we are not able to determine $X_H(\Q)$.}\label{table:cases}
\end{table}

\FloatBarrier

%----------------------------------------------------------------------------
\section{\texorpdfstring{Counting $\F_q$-points on $X_H$}{Counting F\_q-points on X\_H}}
\label{sec:enumeration-of-moduli}
%----------------------------------------------------------------------------

For any prime power $q=p^e$ coprime to $N$ and $H\le\GL_2(N)$ we may consider the modular curve $X_H$ over the finite field $\F_q$.
We can count $\F_q$-rational points on $X_H$ via
\[
\#X_H(\F_q) = \#X_H^\infty(\F_q) + \#Y_H(\F_q),
\]
by counting fixed points of a right $G_{\F_q}$-action on certain double coset spaces of $\GL_2(N)$ defined in Subsection \ref{ssec:modular-curves}.
In this section we explain how to do this explicitly and efficiently, via a refinement of the method proposed in \cite{zywina:Possible-indices-for-the-Galois-image-of-elliptic-curves-over-Q-arxiv}*{\S 3}.  Applications include:
\begin{itemize}\setlength{\itemsep}{1pt}
\item checking for $p$-adic obstructions to rational points (by checking if $\#X_H(\F_p)=0$);
\item computing the zeta function of $X_H/\F_p$ (by computing $\#X_H(\F_{p^r})$ for $1\le r\le g(H)$);
\item determining the isogeny decomposition of $\Jac(X_H)$;
\item determining the analytic rank of $\Jac(X_H)$.
\end{itemize}
The last two are described in Section~\ref{sec:analytic-rank}. None require an explicit model for $X_H$.

\subsection{Counting points}
If we put $\overline{H}\coloneqq \langle H,-I\rangle\le \GL_2(N)$ we can use the permutation representation provided by the right $G_{\F_q}$-action on the coset space $[\overline{H}\backslash \GL_2(N)]$ to compute
\[
\#X_{H}^\infty(\F_q) = \#(H\backslash \GL_2(N)/U(N))^{G_{\F_q}},
\]
as the number of $\smallmat{1}{1}{0}{1}$-orbits of $[\overline{H}\backslash\GL_2(N)]$ stable under the action of $\smallmat{q}{0}{0}{1}$, which we note depends only on $q\bmod N$; if one expects to compute $\#X_H^\infty(\F_q)$ for many finite fields $\F_q$ (as when computing the $L$-function of $X_H/\Q$, for example), one can simply precompute a table of rational cusp counts indexed by $(\Z/N\Z)^\times$.

To compute
\begin{equation}\label{eq:YH}
\#Y_H(\F_q) = \sum_{j(E)\in \F_q} \#(H\backslash \GL_2(N)/A_E)^{G_{\F_q}},
\end{equation}
we first note that $\pm I\in A_E\coloneqq \{\varphi_N:\varphi\in \Aut(E_\kbar)\}$, so we may replace $H$ with $\overline{H}$ and work with the quotient $A_E/\pm I$, which is trivial for $j(E)\ne 0,1728$.  It then suffices to count elements of $[\overline{H}\backslash \GL_2(N)]$ fixed by $G_{\F_q}$ using the action of the Frobenius endomorphism $\pi$ on~$E[N]$ with respect to some basis; we are free to choose any basis we like, since the number of elements of $[\overline{H}\backslash\GL_2(N)]$ fixed by a matrix in $\GL_2(N)$ is invariant under conjugation.  By \cite{DukeT:The-splitting-of-primes-in-division-fields-of-elliptic-curves}*{Theorem~2.1}, we can use the reduction modulo~$N$ of the integer matrix
\begin{equation}\label{eq:Api}
A_\pi \coloneqq A(a,b,\Delta)\coloneqq \begin{bmatrix} \frac{a+bd}{2} & b\vspace{2pt}\\ \frac{b(\Delta-d)}{4} & \frac{a-bd}{2}\end{bmatrix},
\end{equation}
where $a,b,\Delta,d\in \Z$ are defined as follows.  Let $R_\pi\coloneqq \End(E)\cap \Q[\pi]$, let $\Delta\coloneqq \disc(R_\pi)$,
let $a\coloneqq \tr \pi$, and let $b\coloneqq\bigl[R_\pi\!:\!\Z[\pi]\bigr]$ if $\Z[\pi]\ne \Z$ and let $b\coloneqq 0$ otherwise, so that we always have
\begin{equation}\label{eq:norm}
4q = a^2-b^2\Delta,
\end{equation}
and let $d\coloneqq \Delta\bmod 2\in\{0,1\}$.  We then have $\tr A_\pi\equiv\tr \pi\bmod N$ and $\det A_\pi\equiv q\bmod N$.  We note that $A_\pi$ is uniquely determined by $\tr\pi$ and $\Delta$ and represents an element of $\GL_2(N)$ for every integer $N$ coprime to $q$.

Let $\chi_{\overline{H}}\colon \GL_2(N)\to \Z_{\ge 0}$ denote the character of the permutation representation given by the right $\GL_2(N)$-set $[\overline{H}\backslash \GL_2(N)]$.
Each term in \eqref{eq:YH} for $j(E)\ne 0,1728$ can be computed as $\chi_{\overline{H}}(A_\pi)$. 
 It follows from the theory of complex multiplication that for each imaginary quadratic discriminant $\Delta$ for which the norm equation \eqref{eq:norm} has a solution with $a>0$ coprime to $q$,
there are exactly $h(\Delta)$ ordinary $j$-invariants of elliptic curves over $\F_q$ with $\disc(R_\pi)=\Delta$, where $h(\Delta)$ is the class number; see \cite{sutherland:Computing-Hilbert-class-polynomials-with-the-Chinese-remainder-theorem}*{Prop~1}, for example.  We thus compute the number of $\F_q$-points on $X_H$ corresponding to ordinary $j(E)\ne 0,1728$ via
\begin{equation}\label{eq:XHord}
\#X_H^{\rm{ord}'}(\F_q) = \sum_{\substack{0<a<2\sqrt{q}\\\gcd(a,q)=1}}\ \sum_{\substack{4q=a^2-b^2\Delta\\\Delta <-4}} h(\Delta)\chi_{\overline{H}}\bigl(A(a,b,\Delta)\bigr).
\end{equation}
To count supersingular points in $X_H(\F_q)$ with $j\ne 0,1728$ we rely on the following lemma.

\begin{lemma}\label{lemma:ss}
Let $q=p^e$ be a prime power, let $H\le \GL_2(N)$ with $\gcd(q,N)=1$, let $h'=\lfloor h(-4p)/2\rfloor$, and let $s_0=1$ for $p\equiv 2\bmod 3$ and $s_0=0$ otherwise.  The number $\#X_H^{\rm ss'}(\F_q)$ of $\F_q$-points on $X_H$ corresponding to supersingular $j(E)\ne 0,1728$ can be computed as follows.
\begin{itemize}\setlength{\itemsep}{1pt}
\item If $p\le 3$ then $\#X_H^{\rm ss'}(\F_q)=0$.
\item If $e$ is odd and $p\equiv 1\bmod 4$ then
\[
\#X_H^{\rm ss'}(\F_q) = (h'-s_0)\chi_{\overline{H}}\bigl(A(0,p^{(e-1)/2},-4p)\bigr).
\]
\item If $e$ is odd and $p\equiv 3\bmod 4$ then
\begin{align*}
\#X_H^{\rm ss'}(\F_q) = h'\chi_{\overline{H}}\bigl(A(0,2p^{(e-1)/2},-p)\bigr) + (h'-s_0)\chi_{\overline{H}}\bigl(A(0,p^{(e-1)/2},-4p)\bigr).
\end{align*}
\item If $e$ is even then
\[
\#X_H^{\rm ss'}(\F_q) = \Biggl(\frac{p-6+2\kron{-3}{p}+3\kron{-4}{p}}{12}\Biggr) \chi_{\overline{H}}\bigl(A(2p^{e/2},0,1)\bigr).
\]
\end{itemize}
\end{lemma}
\begin{proof}
These formulas are derived from \cite{schoof:Nonsingular-plane-cubic-curves-over-finite-fields}*{Theorem 4.6} by discarding $\F_q$-isomorphism classes with $j(E)=0,1728$ and dividing by 2 to count $j$-invariants rather than counting $\F_q$-isomorphism classes.
For $p=2,3$ there are no supersingular $j(E)\ne 0,1728$, and for $p>3$ and $j(E)\ne 0,1728$ we only need to consider supersingular $\F_q$-isomorphism classes with trace $a=0$ and $\Delta=-0,-4p$ when $e$ is odd, and with trace $a=2p^{e/2}$ and $\Delta=1$ when $e$ is even.
In each case the number of supersingular $\F_q$-isomorphism classes with $j$-invariant $0,1728$ that we need to discard can be determined using Table~\ref{table:j0j1728} below.
\end{proof}

For $j(E)=0,1728$, rather than computing double cosets fixed by $G_k$, we instead compute a weighted sum of $\chi_{\overline{H}}(A_\pi)$ over $k$-isomorphism classes of elliptic curves with $j(E)=0,1728$ via the table below.  If we extend the permutation character $\chi_{\overline{H}}$ to the group ring $\Q[\GL_2(N)]$, we can compute the number of $k$-rational points of $X_H$ above $j=0,1728\in X(1)$ via
\[
\# X_H^{j=0}(\F_q)= \chi_{\overline{H}}\left(\sum_{\substack{j(E)=0\\p\ne 2}} \frac{A(a,b,\Delta)}{\#\Aut(E)} \right), \qquad \#X_H^{j=1728}(\F_q) = \chi_{\overline{H}}\left(\sum_{\substack{j(E)=1728\\p\ne 3}} \frac{A(a,b,\Delta)}{\#\Aut(E)} \right),
\]
where the sums are over $\F_q$-isomorphism classes of elliptic curves $E/\F_q$, and the values of $a,b,\Delta$ are listed in Table~\ref{table:j0j1728}.  These values are derived from \cite{waterhouse:abelianvarietiesfinitefields} and \cite{schoof:Nonsingular-plane-cubic-curves-over-finite-fields}, and we note that for each triple $(|a|,b,\Delta)$ with $|a|>0$ listed in the table with multiplicity $m$, the triples $(a,b,\Delta)$ and $(-a,b,\Delta)$ each occur with multiplicity $m/2$.

\begin{table}[htp!]
\bigskip
\begin{center}
\begin{tabular}{rcccccrc}
$\kron{-3}{p}$ & $e\bmod 2$ & $\#\!\Aut(E)$ &  $\rk\!E[2](\F_{p^e})$ & $|a|$ & $b$ & $\Delta$ & $m$\vspace{2pt}\\\toprule
$-1$ & $1$ & $2$ & $1$ & $0$ & $p^{(e-1)/2}$ & $-4p$ & 2\\
$-1$ & $0$ & $6$ & $0$ & $p^{e/2}$ & $p^{e/2}$ & $-3$ & 4\\
$-1$ & $0$ & $6$ & $2$ & $2p^{e/2}$ & $0$ & $1$ & 2\\
$0$ & $1$ & $2$ & $1$ & $0$ & $3^{(e-1)/2}$ & $-12$ & 1\\
$0$ & $1$ & $6$ & $2$ & $0$ & $2\cdot 3^{(e-1)/2}$ & $-3$ & 1\\
$0$ & $1$ & $6$ & $0$ & $3^{(e+1)/2}$ & $3^{(e-1)/2}$ & $-3$ & 2\\
$0$ & $0$ & $4$ & $1$ & $0$ & $3^{e/2}$ & $-4$ & 2\\
$0$ & $0$ & $6$ & $0$ & $3^{e/2}$ & $3^{e/2}$ & $-3$ & 2\\
$0$ & $0$ & $12$ & $2$ & $2\cdot 3^{e/2}$ & $0$ & $1$ & 2\\
$+1$ & $*$ & $6$ & $2$ &  $u$ & $v$ & $-3$ & 2\\
$+1$ & $*$ & $6$ & $0$ & $\frac{u+3v}{2}$ & $\frac{|u-v|}{2}$ & $-3$ & 2\\
$+1$ & $*$ & $6$ & $0$ & $\frac{|u-3v|}{2}$ & $\frac{u+v}{2}$ & $-3$ & 2\\\bottomrule
\end{tabular}
\vspace{2pt}

$j(E)=0,\ p\ne 2,\ 4p^e=u^2+3v^2$ with $u,v > 0$ even
\end{center}

\bigskip

\begin{center}
\begin{tabular}{rcccccrc}
$\kron{-4}{p}$ & $e\bmod 2$ & $\#\!\Aut(E)$ &  $\rk\!E[2](\F_{p^e})$ & $|a|$ & $b$ & $\Delta$ & $m$\vspace{2pt}\\\toprule
$-1$ & $1$ & $2$ & $1$ & $0$ & $p^{(e-1)/2}$ & $-4p$ & 1\\
$-1$ & $1$ & $2$ & $2$ & $0$ & $2p^{(e-1)/2}$ & $-p$ & 1\\
$-1$ & $0$ & $4$ & $1$ & $0$ & $p^{e/2}$ & $-4$ & 2\\
$-1$ & $0$ & $4$ & $2$ & $2p^{e/2}$ & $0$ & $1$ & 2\\
$0$ & $1$ & $2$ & $0$ & $0$ & $2^{(e-1)/2}$ & $-8$ & 1\\
$0$ & $1$ & $4$ & $0$ & $2^{(e+1)/2}$ & $2^{(e-1)/2}$ & $-4$ & 2\\
$0$ & $0$ & $4$ & $0$ & $0$ & $2^{e/2}$ & $-4$ & 1\\
$0$ & $0$ & $6$ & $0$ & $2^{e/2}$ & $2^{e/2}$ & $-3$ & 4\\
$0$ & $0$ & $24$ & $0$ & $2\cdot 2^{e/2}$ & $0$ & $1$ & 2\\
$+1$ & $*$ & $4$ & $2$ & $u$ & $v$ & $-4$ & 2\\
$+1$ & $*$ & $4$ & $1$ & $2v$ & $\frac{u}{2}$ & $-4$ & 2\\\bottomrule
\end{tabular}
\vspace{2pt}

$j(E)=1728,\ p\ne 3,\ 4p^e=u^2+4v^2$ with $u,v>0$ even
\medskip

\caption{$\F_{p^e}$-isomorphism classes of elliptic curves with $j$-invariant $0$ or $1728$.}\label{table:j0j1728}
\end{center}
\end{table}

To sum up, for any $H\le \GL_2(N)$ we can compute $\#X_H(\F_q)$ as
\begin{equation}\label{eq:ptsum}
\#X_H(\F_q) = \#X_H^\infty(\F_q) + \#X_H^{\rm ord'}(\F_q) +\#X_H^{\rm ss'}(\F_q) +\#X_H^{j=0}(\F_q) +\#X_H^{j=1728}(\F_q),
\end{equation}
where each term in the sum is computed using the permutation representation $[\overline{H}\backslash \GL_2(N)]$: the term $\#X_H^\infty(\F_q)$ is computed by counting $\smallmat{q}{0}{0}{q}$-stable $\smallmat{1}{1}{0}{1}$-orbits, and the remaining terms are weighted sums of values of the permutation character $\chi_{\overline{H}}$ on the matrices $A(a,b,\Delta)$ arising for elliptic curves $E/\F_q$.  A \magma implementation of the formula in \eqref{eq:ptsum} is provided by the function \magmacmd{GL2PointCount} in the file \repolink{groups}{gl2.m} in \cite{RouseSZB:magma-scripts-ell-adic-images}.

\begin{example}
Consider $H=\{\smallmat{1}{*}{0}{*}\}\le \GL_2(13)$ with label \texttt{13.168.2.1}.  The modular curve $X_H=X_1(13)$ parameterizes elliptic curves with a rational point of order 13.  Over $\F_{37}$ there are four such elliptic curves, up to $\F_{37}$-isomorphism, each with 12 points of order~13, with distinct $j$-invariants $0$, $16$, $26=1728$, $35$ (none are supersingular), and automorphism groups of order 6, 2, 4, 2, respectively. We thus expect $\#Y_H(\F_{37}) = 12/6 + 12/2 + 12/4 + 12/2 = 17$,
and one can check that 6 of the 12 cusps of $X_1(13)$ are $\F_{37}$-rational (those of width of 13).
This agrees with the fact that there are 23 rational points on the mod-37 reduction of the genus 2 curve \href{http://www.lmfdb.org/Genus2Curve/Q/169/a/169/1}{$y^2+(x^3+x+1)y=x^5+x^4$}, which is a $\Z[1/13]$-model for $X_H=X_1(13)$.

To compute $\#X_H(\F_{37})$ using the algorithm proposed above, we apply \eqref{eq:ptsum}, which does not require us to enumerate elliptic curves with a point of order 13.  We first compute $\#X_H^\infty(\F_{37})=6$ by noting that the coset space $[\overline{H}\backslash\GL_2(13)]$ has twelve $\smallmat{1}{1}{0}{1}$-orbits, six of size 1 and six of size~13, the later of which are stable under the action of $\smallmat{37}{0}{0}{1}\equiv \smallmat{11}{0}{0}{1}$.
Enumerating integers $a\in [1,\lfloor 2\sqrt{37}\rfloor]$ yields 17 triples $(a,b,\Delta)$ with $\Delta  < -4$ that satisfy the norm equation \eqref{eq:norm}.  Among these, only $A(1,1,-147)$ fixes an element of $[\overline{H}\backslash\GL_2(13)]$, and from \eqref{eq:XHord} we obtain
\[
\#X_H^{\rm{ord}'}(\F_q) = h(\Delta)\chi_{\overline{H}}\bigl(A(a,b,\Delta)\bigr) = h(-147)\chi_{\overline{H}}\bigl(A(1,1,-147)\bigr) = 2\cdot 6 = 12.
\]
To compute $\#X_H^{\rm ss'}(\F_{37})$ we apply Lemma~\ref{lemma:ss}, and find that $\chi_{\overline{H}}(A(0,0,-4\cdot 37))=0$, so $\#X_H^{\rm ss'}(\F_q) =0$.
Using Table~\ref{table:j0j1728} we compute
\begin{align*}
\#X_H^{j=0}(\F_{37}) &= \chi_{\overline{H}}(A(\pm 1,7,-3)/6 + A(\pm 10, 4,-3)/6 + A(\pm 11,3,-3)/6)=2+0+0=2\\
\#X_H^{j=1728}(\F_{37}) &= \chi_{\overline{H}}(A(\pm 2,6,-4)/4 + A(\pm 12, 1,-4)/4)=0+3=3
\end{align*}
and \eqref{eq:ptsum} then yields
\begin{align*}
\#X_H(\F_{37}) &= \#X_H^\infty(\F_{37}) + \#X_H^{\rm ord'}(\F_{37}) +\#X_H^{\rm ss'}(\F_{37}) +\#X_H^{j=0}(\F_{37}) +\#X_H^{j=1728}(\F_{37})\\
		    &= 6 + 12 + 0 + 2 + 3 = 23.
\end{align*}
In contrast to the method described in \cite{zywina:Possible-indices-for-the-Galois-image-of-elliptic-curves-over-Q-arxiv}*{\S 3}, in our computation above we did not enumerate elliptic curves or $j$-invariants over $\F_q$, nor did we compute any Hilbert class polynomials (a commonly used but inefficient way to compute $A_\pi$).  Instead we enumerated possible Frobenius traces $a>0$, solved norm equations, computed class numbers, and computed values of the permutation character $\chi_{\overline{H}}$. This can be viewed as an algorithmic generalization of the Eichler--Selberg trace formula that works for any open $H\le \GL_2(\Zhat)$.  It has roughly the same asymptotic complexity as the trace formula (the dependence on $q$ is $\tilde O(q^{1/2})$ if we ignore the cost of computing class numbers, which in practice we simply look up in a table such as those available for $|D|\le 2^{40}$ in the LMFDB).  But the constant factors are better, and even for groups that do not contain the $\GL_2$-analog of $\Gamma_1(N)$, we are always able to work at level $N$ rather than level $N^2$, as typically required by generalizations of the trace formula to handle congruence subgroups that do not contain $\Gamma_1(N)$.
\end{example}

\subsection{Local obstructions}\label{ssec:pointless}

An easy application of this point counting algorithm is checking for local obstructions to rational points on the arithmetically maximal modular curves $X_H$.
For $X_H$ of genus~$g$ it suffices to compute $\#X(\F_p)$ for $p < \lfloor 2g\sqrt{p}\rfloor$, since the Weil bounds force $\#X(\F_p)>0$ for all larger primes $p$.
Doing this for the 210 arithmetically maximal groups summarized in Table~\ref{table:amstats} yields Theorem~\ref{theorem:pointless-curves}, whose results are summarized in Table~\ref{table:empty-via-moduli}.

\begin{theorem}
\label{theorem:pointless-curves}
For $\ell\le 37$ the arithmetically maximal $H\le \GL_2(\Zhat)$ of $\ell$-power level for which $X_H$ has no $\F_{p}$-points for some prime $p\ne \ell$ are precisely those listed in Table \ref{table:empty-via-moduli}.
\end{theorem}
\vspace{-12pt}

\begin{table}[htp]
\def\arraystretch{1.25}
\begin{tabular}{lclccc}
label & level & generators & $p$ & rank & genus \\\toprule
\glink{16.48.2.17} & $2^4$ & $\smallmat{11}{9}{4}{13}, \smallmat{13}{5}{4}{11}, \smallmat{1}{9}{12}{7}, \smallmat{1}{9}{0}{5}$ & $3,11$ & 0 & 2\\
\glink{27.108.4.5} & $3^3$ & $\smallmat{4}{25}{6}{14}, \smallmat{8}{0}{3}{1}$ & $7,31$ & 0 & 4\\
\glink{25.150.4.2} & $5^2$ & $\smallmat{7}{20}{20}{7}, \smallmat{22}{2}{13}{22}$ & $2$ & 0 & 4\\
\glink{25.150.4.7} & $5^2$ & $\smallmat{24}{24}{0}{18}, \smallmat{2}{5}{0}{23}$ & $3,23$ & 4 & 4\\
\glink{25.150.4.8} & $5^2$ & $\smallmat{8}{4}{0}{23}, \smallmat{16}{7}{0}{8}$ & $2$ & 0 & 4\\
\glink{25.150.4.9} & $5^2$ & $\smallmat{2}{0}{0}{8}, \smallmat{3}{18}{0}{14}$ & $2$ & 0 & 4\\
\glink{49.168.12.1} & $7^2$ & $\smallmat{39}{6}{36}{24}, \smallmat{11}{9}{24}{2}$ & $2$ & 3 & 12\\
\glink{13.84.2.2} & $13$ & $\smallmat{3}{7}{0}{8}, \smallmat{12}{4}{0}{12}$ & $2$ & 0 & 2\\
\glink{13.84.2.3} & $13$ & $\smallmat{9}{2}{0}{7}, \smallmat{4}{4}{0}{7}$ & $3$ & 0 & 2\\
\glink{13.84.2.4} & $13$ & $\smallmat{8}{12}{0}{10}, \smallmat{8}{3}{0}{9}$ & $2$ & 0 & 2\\
\glink{13.84.2.6} & $13$ & $\smallmat{9}{0}{0}{4}, \smallmat{11}{3}{0}{10}$ & $3$ & 0 & 2\\
\bottomrule
\end{tabular}
\smallskip

\caption{Arithmetically maximal $H\le \GL_2(\Zhat)$ of $\ell$-power level for which $X_H$ has no $\F_p$-points for some prime $p\ne \ell \le 37$.}\label{table:empty-via-moduli}
\vspace{-20pt}
\end{table}

%----------------------------------------------------------------------------
\section{\texorpdfstring{Computing the isogeny decomposition and analytic rank of $J_H$}{Computing the isogeny decomposition and analytic rank of J\_H}}
\label{sec:analytic-rank}
%----------------------------------------------------------------------------
\setcounter{subsection}{1}

Let $H$ be an open subgroup of $\GL_2(\Zhat)$.
In this section we explain how the point counting algorithm described in the previous section can be used to determine the isogeny factors of the Jacobian $J_H\coloneqq\Jac(X_H)$ as modular abelian varieties; this decomposition can then be used to compute the analytic rank of $J_H$.  To simplify the exposition we shall assume that $\det(H)=\Zhat^\times$, which holds in all the cases of interest to us here, but the methods we describe can also be applied when $\det(H)\ne \Zhat^\times$; see Example \ref{example:non-surjective-determinant} for details.

The Jacobian $J_H$ is an abelian variety over $\Q$ that has good reduction at all primes $p$ that do not divide the level $N$ of $H$.
It follows from Theorem~\ref{thm:finalans} that the simple factors of $J_H$ are each $\Q$-isogenous to the modular abelian variety associated to the Galois orbit of a (not necessarily new, normalized) eigenform of weight~2 for $\Gamma_1(N)\cap \Gamma_0(N^2)$.  We may decompose this space of modular forms as
\[
S_2\bigl(\Gamma_1(N)\cap\Gamma_0(N^2)\bigr)=\bigoplus_{\cond(\chi)|N}\!\!\!S_2\bigl(\Gamma_0(N^2),[\chi]\bigr),
\]
where $[\chi]$ denotes the Galois orbit of the Dirichlet character $\chi$ and the direct sum varies over Galois orbits of Dirichlet characters of modulus $N^2$ whose conductor divides $N$.
The Galois orbit of a normalized eigenform $f\in S_2(\Gamma_0(N^2),[\chi])$ has an associated \defi{trace form} $\Tr(f)$ (see \cite{BestBBCCDLLRSV:Computing-classical-modular-forms}*{\S 4.5}), which has an integral $q$-expansion
\[
\Tr(f)(q)\coloneqq \sum_{n\ge 1} \Tr_{\Q(f)/\Q}(a_n(f))q^n
\]
with $a_1(\Tr(f))=\dim f\coloneqq [\Q(f):\Q]$, and for each prime $p$ the coefficient $a_p$ is the Frobenius trace of the abelian variety $A_f/\Q$ of dimension $\dim f$ associated to $f$ via the Eichler--Shimura correspondence; equivalently, $a_p$ is the $p$th (arithmetically normalized) Dirichlet coefficient of the $L$-function $L(A_f,s)=\prod_{f\in [f]}L(f,s)$.

Let $S(H)\coloneqq \{[f_1],\ldots,[f_m]\}$ be the Galois orbits of eigenforms $f\in S_2(\Gamma_1(N)\cap\Gamma_0(N^2))$ with $\dim f\le\dim J_H=g(H)$. Theorem~\ref{thm:finalans} implies that there is a sequence of nonnegative integers $e(H)=(e_1,\ldots,e_m)$ for which we have
\[
L(J_H,s) = \prod_{i=1}^m L(A_{f_i},s)^{e_i},
\]
and this sequence also satisfies $\sum_i e_i\dim f_i =g(H)$.  Now let $T(B)$ be the $n\times m$ integer matrix whose $i$th column is $[a_1(\Tr(f_i)),a_2(\Tr(f_i)),a_3(\Tr(f_i)),a_5(\Tr(f_i)),\ldots,a_p(\Tr(f_i)),\ldots]$, where $p$ varies over the $n-1$ primes bounded by $B$ that do not divide the level of $H$, and let $a(H;B)$ be the vector of integers $[g(H),a_2(H),a_3(H),\ldots,a_p(H),\ldots]$, where $a_p(H)$ denotes the trace of the Frobenius endomorphism of $J_H$ at a prime $p\le B$ not dividing the level of~$H$, which we can compute as $a_p(H)= p+1-\#X_H(\F_p)$ via \eqref{eq:ptsum}.  For $B\ge 1$ we have
\[
T(B)e(H)=a(H;B),
\]
and it follows from an effective form of multiplicity 1 for the Selberg class \cites{KaczorowskiP:Strong-multiplicity-one-for-the-Selberg-class,Soundararajan:Strong-multiplicity-one-for-the-Selberg-class} that for all sufficiently large $B$ the matrix $T(B)$ and the vector $a(H,B)$ uniquely determine the vector $e(H)$.  Indeed, there are only finitely many sequences of nonnegative integers that sum to $g(H)$, and the corresponding $L$-functions are all distinguished by their Dirichlet coefficients $a_p$ on the set of primes $p\nmid N$, by strong multiplicity 1; for each pair of these $L$-functions there is some smallest $p\nmid N$ for which the $a_p$ differ, and we can take $B$ to be the maximum of this $p$ over all pairs.  We note that $L$-functions of modular forms and all products of such $L$-functions not only lie in the Selberg class, they satisfy the polynomial Euler product condition used in \cite{KaczorowskiP:Strong-multiplicity-one-for-the-Selberg-class} and the weaker condition required by \cite{Soundararajan:Strong-multiplicity-one-for-the-Selberg-class}.

This suggests the following approach to determining the exponents in the factorization $L(J_H,s)=\prod_i L(A_{f_i},s)^{e_i}$, from which both the analytic rank of $J_H$ and its isogeny-factor decomposition can be derived.  We first determine the set $S(H)$, which we note depends only on the genus and level of $H$.  We then determine a value of $B$ for which the $\#S(H)$ columns of $T(B)$ become linearly independent over $\Q$, e.g., by starting with a value of $B$ that makes $T(B)$ a square matrix and successively increasing $B$ by some constant factor $\alpha>1$.  Finally, we compute the vector $a(H;B)$ by computing $p+1-\#X_H(\F_p)$ for primes $p\le B$ not dividing the level of $H$ and determine $e(H)$ as the necessarily unique solution to the linear system $T(B)x=a(H;B)$.

This approach is remarkably effective in practice. For the 210 arithmetically maximal~$H$ arising for $\ell\le 37$ the value of $B$ never exceeds 3000 and the matrix $T(B)$ with linear independent columns typically has less than a hundred rows.
To compute $T(B)$ we use precomputed trace form data stored in the LMFDB along with additional data available in the file \repolink{groups}{cmfdata.txt} in  \cite{RouseSZB:magma-scripts-ell-adic-images} that we computed to handle cases where the set $S(H)$ contains Galois orbits of eigenforms that are not currently stored in the LMFDB, using the methods described in \cite{BestBBCCDLLRSV:Computing-classical-modular-forms}.

\begin{remark}
There are special cases in which we know the isogeny factors of $J_H$ \emph{a priori} and need not apply the approach described above in order to determine the vector~$e(H)$.
When $H$ is the Borel subgroup of level $N$, we have $X_H=X_0(N)$ and the isogeny factors of~$J_H$ are the modular abelian varieties associated to the eigenforms $f\in S_2(\Gamma_0(N))$.  When~$H$ is the normalizer of a nonsplit Cartan subgroup of prime power level $\ell^n$, we have $X_H=X_{\ns}^+(\ell^n)$ and the isogeny factors of $J_H$ are the modular abelian varieties associated to the eigenforms in $f\in S_2^{\rm new}(\Gamma_0(\ell^{2r}))$ with Fricke eigenvalue 1, for $1\le r\le n$, and when~$H$ is the normalizer of a split Cartan subgroup of level $\ell^n$, we have $X_H=X_{\spc}^+(\ell^r)$ and the isogeny factors of $J_H$ are the modular abelian varieties associated to the eigenforms in $f\in S_2(\Gamma_0(\ell^{2n}))$ with Fricke eigenvalue 1, as shown by Chen \cite{Chen:Jacobians-of-modular-curves-associated-to-normalizers-of-Cartan-sugroups-of-level-pn}.
\end{remark}

Having determined the vector $e(H)=[e_1,\ldots,e_m]$ we compute the rank of $J_H$ as $\sum_i e_ir_i$ where each $r_i$ is the sum of the analytic ranks of the eigenforms in the $i$th Galois orbit~$[f_i]$ (one expects them to have the same rank as $f_i$; this is true for all of the modular forms in the LMFDB and for all of the modular forms that we computed).  The method used to rigorously compute the analytic ranks of these modular forms is described in \cite{BestBBCCDLLRSV:Computing-classical-modular-forms}*{\S 9}.

\begin{example}
\label{example:27.36.3.1}
Consider the group $H=\langle \smallmat{4}{1}{21}{2},\smallmat{11}{0}{0}{19}\rangle\le \GL_2(27)$ with label \glink{27.36.3.1}.
Eleven Galois orbits of eigenforms $f\in S_2(\Gamma_1(27)\cap\Gamma_0(27^2))$ have $\dim f\le g(H)= 3$, with LMFDB labels
\mflabel{243.2.a.a}, \mflabel{27.2.a.a}, \mflabel{243.2.a.b}, \mflabel{243.2.c.a}, \mflabel{243.2.a.c}, \mflabel{81.2.c.a}, \mflabel{243.2.a.d}, \mflabel{81.2.a.a}, \mflabel{243.2.c.b}, \mflabel{243.2.a.e}, and \mflabel{243.2.a.f}. For $B =73$ the $21\times 11$ matrix $T(B)$ has independent columns.  Computing $\#X_H(\F_p)$ for $p=2,5,7,11,\ldots, 73$ yields $a(H;B)=[3, 0, 0, 0, 0, \minus 9, 0, \minus 3, 0, 0, 9, 6, 0, 27, 0, 0, 0, 9, \minus 9, 0, 15 ]^{\mathrm{T}}$, and the linear system

\begin{small}
\[
\begin{bmatrix*}[r]
 1 & 1 & 1 & 2 & 2 & 2 & 2 & 2 & 2 & 3 & 3\\
 0 & 0 & 0 & 0 & 0 & 0 & 0 & 0 & 0 & \minus 3 & 3\\
 0 & 0 & 0 & 0 & 0 & 0 & 0 & 0 & 0  & \minus 6 & 6\\
\minus 4  & \minus 1 & 5  & \minus 5  & \minus 2 & 1 & 4 & 4 & 4  & \minus 3  & \minus 3\\
 0 & 0 & 0 & 0 & 0 & 0 & 0 & 0 & 0  & \minus 3 & 3\\
\minus 7 & 5 & 2  & \minus 2 & 10  & \minus 5  & \minus 2  & \minus 2 & 7  & \minus 3  & \minus 3\\
 0 & 0 & 0 & 0 & 0 & 0 & 0 & 0 & 0  & \minus 9 & 9\\
\minus 1  & \minus 7 & 8 & 16 & \minus 2 & \minus 14 & \minus 2 & 4  & \minus 2  & \minus 3  & \minus 3\\
 0 & 0 & 0 & 0 & 0 & 0 & 0 & 0 & 0  & \minus 6 & 6\\
 0 & 0 & 0 & 0 & 0 & 0 & 0 & 0 & 0 & \minus 12 & 12\\
11  & \minus 4  & \minus 7 & 7 & 10 & 4  & \minus 2 & 16 & \minus 11 & \minus 12 & \minus 12\\
10 & 11 & \minus 1  & \minus 2  & \minus 2 & 22 & 16 & \minus 14 & \minus 20  & \minus 3  & \minus 3\\
 0 & 0 & 0 & 0 & 0 & 0 & 0 & 0 & 0 & 3  & \minus 3\\
 5 & 8 & \minus 13 & 13  & \minus 2  & \minus 8 & 22 & 4  & \minus 5 & \minus 12 & \minus 12\\
 0 & 0 & 0 & 0 & 0 & 0 & 0 & 0 & 0 & 6  & \minus 6\\
 0 & 0 & 0 & 0 & 0 & 0 & 0 & 0 & 0 & \minus 18 & 18\\
 0 & 0 & 0 & 0 & 0 & 0 & 0 & 0 & 0 & 21 & \minus 21\\
\minus 1  & \minus 1  & \minus 1 & 1 & 4 & 1 & 10 & \minus 14 & 1 & 6 & 6\\
 5 & 5 & 5  & \minus 5 & 16  & \minus 5 & \minus 14 & \minus 20 & \minus 5 & 6 & 6\\
 0 & 0 & 0 & 0 & 0 & 0 & 0 & 0 & 0 & 9 & \minus 9\\
\minus 7  & \minus 7  & \minus 7 & \minus 14 & 4 & \minus 14 & 22 & \minus 14 & \minus 14 & 6 & 6
\end{bmatrix*}e(H) = \begin{bmatrix*}[r] 3\\ 0\\ 0\\ 0\\ 0\\ \minus 9\\ 0\\ \minus 3\\ 0\\ 0\\ 9\\ 6\\ 0\\ 27\\ 0\\ 0\\ 0\\ 9\\ \minus 9\\ 0\\ 15\end{bmatrix*}
\]
\end{small}

\noindent
has the unique solution $e(H)=[1,0,0,0,0,0,1,0,0,0,0]^{\mathrm{T}}$, corresponding to the Galois orbits \mflabel{243.2.a.a} and \mflabel{243.2.a.d} of dimensions 1 and 2 with analytic ranks 1 and 0, respectively; thus $J_H$ has analytic rank 1.
\end{example}

We used the method described above to compute the analytic ranks and the decomposition of $J_H$ up to isogeny for all but one of the arithmetically maximal subgroups of $\ell$-power level for $\ell\le 37$.  The sole exception is the group $H$ with label \glink{121.605.41.1} of level $121$, where we were unable to compute all of the elements of $S(H)$ (we were able to do this for the group \glink{121.6655.511.1} corresponding to $N_{\rm ns}(121)$ by restricting to eigenforms $f\in S_2(\Gamma_0(\ell^{2n}))$ via the optimization noted above).  Computing all of the elements of $S(H)$ in this case would require decomposing the newspaces $S_2^{\rm new}(\Gamma_0(11^4),[\chi])$ for all Dirichlet characters of modulus $11^4$ and conductor dividing $11^2$, which we found impractical in the two cases where $\cond(\chi)=11^2$ (the corresponding newspaces have $\Q$-dimensions 10550 and 42200).

For the group \glink{121.605.41.1} we apply the alternative strategy of verifying a candidate value for $e(H)$ by checking Fourier coefficients up to the Sturm bound.
We first use the subset of $S(H)$ that we know to find a candidate $e(H)=[e_1,\ldots,e_m]$ that works even for~$B$ much larger than necessary to ensure that the columns of $T(B)$ are linearly independent.  To rigorously verify this candidate it suffices to compare the Dirichlet coefficients $b_n,c_n\in\Z$ of the $L$-functions
\[
L(J_H,s)= \sum_{n=1}^\infty b_n n^{-s}\qquad\text{and}\qquad \prod_i L(A_{f_i},s)^{e_i}=\sum_{n=1}^\infty c_nn^{-s},
\]
for $n$ up to the Sturm bound for the space $S_2\bigl(\Gamma_1(N)\cap\Gamma_0(N^2)\bigr)$, which is $B\coloneqq 322\,102$.
The integers $c_n$ uniquely determine the Fourier coefficients $a_n$ that appear in the $q$-expansion of the modular from $F\coloneqq \sum_i\Tr(f_i)=\sum_{n=1}^\infty a_nq^n\in S_2\bigl(\Gamma_1(N)\cap\Gamma_0(N^2)\bigr)$ for $n\le B$, which uniquely determines $F$.\footnote{The sequences $(c_n)_{n\le B}$ and $(a_n)_{n\le B}$ are not equal in general (although we do have $c_p=a_p$), but they uniquely determine each other, which is all we are using; see \cite{BestBBCCDLLRSV:Computing-classical-modular-forms}*{\S 8.6}.}

Both $L$-functions are defined by an Euler product, so it suffices to check $b_q=c_q$ for prime powers $q\le B$.  We compute the $c_q$ by computing Euler factors of $\prod_{f\in [f_i]}L(f,s)$ at all primes $p\le B$ using the modular symbols functionality in \magma. For $11\!\nmid\! q$ we compute the $b_q$ by computing $\#X_H(\F_{p^r})$ for $1\le r\le \log_p B$ via the point counting algorithm described in the previous section (these point counts determine enough coefficients of the degree-82 $L$-polynomial of $J_H$ at $p$ to determine the $a_q$ for $q\le B$).  To compute the Euler factor of $L(J_H,s)$ at 11 (which determines the $a_q$ for $11|q$), we use the $\GL_2$-modular symbols functionality in the \magma package \href{https://github.com/assaferan/ModFrmGL2}{\texttt{ModFrmGL2}} recently developed by Assaf \cite{Assaf}.
The most time consuming part of this process is computing the $c_q$ for $q\le B=322\,102$, which takes 15-20 hours.  It takes about 15 minutes to compute the $b_q$ for $11\nmid q$ and another 15 or~20 minutes to compute the Euler factor of $L(J_H,s)$ at 11.  \magma scripts that reproduce all three computations and verify that the results match are available at \cite{RouseSZB:magma-scripts-ell-adic-images}.

\begin{example}
\label{example:non-surjective-determinant}
We remark that these methods apply to groups with nonsurjective determinant as well. While for such $H$, $X_H(\Q)$ is empty (see the discussion after Definition~\ref{defn:arithmetically-maximal}) and $X_H$ is connected but not geometrically connected, the sheafification of the functor $\Pic^0 X_H$ is still representable by an abelian variety $J_H$ over $\Q$; see Remark \ref{remark:disconnected} and Appendix \ref{S:kolyvagin}, where it is also proved that the isogeny factors of $J_H$ are twists $A_{f \otimes \chi}$ of factors of $J_1(N^2)$, where~$N$ is the level of $H$ and $\chi$ is a Dirichlet character of modulus $N$.

One can compute the decomposition of $J_H$ via the methods of this section. At primes $p \nmid N$ whose reduction mod $N$ is not an element of $\det(H)$ the point count $X_H(\F_p)$ is zero (again by the Weil pairing), and one can show that the trace of Frobenius acting on $H^1$ is also zero; for consistency, we also remark that since, for such $p$ none of the geometric connected components of $X_{H,\F_p}$ are defined over $\F_p$, the trace of Frobenius acting on $H^0$ (and by duality, also on $H^2$) are zero.

For example, consider the group $H=\langle \smallmat{8}{0}{0}{1},\smallmat{1}{3}{0}{8}\rangle\le \GL_2(9)$ with label \texttt{9.3.1-9.108.1.1}. One has $\det(H) = \{\pm 1\}$, which has index 3; $X_H$ is connected, but geometrically disconnected, and the geometric components are each isomorphic (over $\Qbar$) to $E_{\Qbar}$, where $E/\Q$ is the elliptic curve with newform \mflabel{27.2.a.a}. Following the procedure outlined in Example \ref{example:27.36.3.1} we find that over $J_H$ decomposes (over $\Q$) as $E \times A$, where $A$ is the simple abelian surface corresponding to the newform orbit \mflabel{81.2.c.a}.  The abelian surface $A$ is not geometrically simple, in fact $A_{\Qbar} \sim E_{\Qbar}^2$, and the newform orbit \mflabel{81.2.c.a} is the twist of the newform \mflabel{27.2.a.a} by the Dirichlet character orbit \href{https://www.lmfdb.org/Character/Dirichlet/9/c}{9.c}.
\end{example}

%----------------------------------------------------------------------------
\section{\texorpdfstring{Computing equations for $X_H$ with $-I \in H$}{Computing equations for X\_H with -I in H}}
\label{sec:comp-equat-x_h}
%----------------------------------------------------------------------------
\setcounter{subsection}{1}

We now discuss the computation of equations for $X_H$ for arithmetically maximal $H$ containing $-I$ of $\ell$-power level for $3\le \ell \le 13$; see Section \ref{sec:universal-curve} for the case $-I \not\in H$.

In \cite{RZB}, Rouse and Zureick-Brown construct the modular curves $X_{H}$ as covers of $X_{\widetilde{H}}$, where
$\widetilde{H}$ is a supergroup of~$H$ for which $[\widetilde{H} : H]$ is minimal. Because we are only recording equations for arithmetically maximal subgroups, in any case where we compute the covering $X_{H} \to X_{\widetilde{H}}$, the genus of~$X_{\widetilde{H}}$ is zero or one. In all cases we consider the genus of $X_{\widetilde{H}}$ is zero and we know a specific function $x \in \Q(X_{\widetilde{H}})$ so that $\Q(X_{\widetilde{H}}) = \Q(x)$. The models are computed by constructing elements of the function field of $\Q(X_{H})$ built using products of weight~$2$ Eisenstein series whose Fourier expansions are simple to compute and for which it is easy to describe the action of $\GL_2(N)$ on them. In this way we find a generator~$y$ for~$\Q(X_{H})$ and then compute the action of $H$-coset representatives of $\widetilde{H}$ on it.
This allows us to compute the minimal polynomial for $y$ over $\Q(x)$, which gives a plane model for $X_{H}$. We use the method of Novocin and van Hoeij \cite{vHoeijN:A-Reduction-Algorithm-for-Algebraic-Function-Fields} to reduce the resulting algebraic function field and obtain a simple model. 

We now describe another approach for computing the canonical model of a high genus curve, which we use to compute equations of $X_{\ns}^{+}(27)$ (label \glink{27.243.12.1}), and $X_{H}$, where~$H$ has label \glink{27.729.43.1} or \glink{25.625.36.1}. Given $H \leq \GL_{2}(N)$, there is a natural action of~$H$ on the field of modular functions for $\Gamma(N)$ with coefficients in $\Q(\zeta_{N})$. In particular, $\smallmat{a}{b}{c}{d} \in \SL_{2}(N)$ acts via $f | \smallmat{a}{b}{c}{d} = f\left(\frac{az+b}{cz+d}\right)$, while the matrix $h = \smallmat{1}{0}{0}{d}$ acts by sending $f = \sum a(n) q^{n}$ to $f | h = \sum \sigma_{d}(a(n)) q^{n}$ where $\sigma_{d}$ is the automorphism of $\Q(\zeta_{N})$ sending $\sigma_{d}(\zeta_{N}) = \zeta_{N}^{d}$. One can extend this map to a natural map on $M_{2k}(\Gamma(N), \Q(\zeta_{N}))$ of the space of weight $2k$ modular forms for $\Gamma(N)$ with Fourier coefficients in $\Q(\zeta_{N})$.  To do so, given a form $F \in M_{2k}(\Gamma(N), \Q(\zeta_{N}))$, define $f = \frac{E_{4}(z)^{k} E_{6}(z)^{k}}{\Delta(z)^{k}} F$, a modular function for $\Gamma(N)$, and
\[
(F | M) \colonequals (f | M) \cdot \frac{\Delta(z)^{k}}{E_{4}(z)^{k} E_{6}(z)^{k}}.
\]
Here $E_{4}(z) \colonequals 1 + 240 \sum_{n=1}^{\infty} \sigma_{3}(n) q^{n}$ (with $q = e^{2 \pi i z}$) and $E_{6}(z) \colonequals 1 - 504 \sum_{n=1}^{\infty} \sigma_{5}(n) q^{n}$ are the classical weight $4$ and $6$ Eisenstein series, respectively.

Given $H$, our goal is to find $q$-expansions for a basis for $S_{2}(\Gamma(N), \Q(\zeta_{N}))^{H}$. This problem has been considered before in \cite{BanwaitCremona} and recently in \cite{ZywinaNew} and \cite{Assaf}. Instead, we continue our approach from \cite{RZB} of first computing the space of weight~$2$ Eisenstein series for $H$. Compared to other methods, our methods has the disadvantage
of requiring a much larger number of Fourier coefficients. However, it only requires working with modular forms defined on $H$, and our method may be substantially more efficient if the index of $\Gamma(N)$ in $H$ is large. The addition
of modular forms functionality to PARI/GP in \cite{belabascohen:modformsinPARIGP} uses the strategy of multiplying Eisenstein series.

The next task is to construct the Fourier coefficients of weight~$2$ cusp forms from those of the Eisenstein series. We do this by multiplying triples of weight~$2$ Eisenstein series together to obtain a large collection of weight $6$ modular forms. Because we know how each weight~$2$ Eisenstein series transforms under matrices in $\GL_{2}(N)$, we can keep track of the values of these products of Eisenstein series at each cusp of $X_{H}$. There is no reason to expect that this is the entirety of the space of weight $6$ modular forms, because weight~$2$ modular forms must have a zero at each elliptic point of order~$2$ and a double zero at each elliptic point of order $3$. One can use the Riemann--Roch theorem to compute that the dimension of the subspace of weight $6$ modular forms with a zero of order $\geq 3$ at each elliptic point of order~$2$ and zero of order $\geq 6$ at each elliptic point of order $3$ is $3 \nu_{\infty}(\Gamma_{H}) +
5(g-1)$. This allows one to avoid computing products of all triples of weight~$2$ Eisenstein series.

Next, one applies Hecke operators to compute the full space of weight $6$ modular forms. In \cite{KatoModForms}, Kato defines a family of Hecke operators $T(n) \smallmat{n}{0}{0}{1}^{*}$ for modular forms on $\Gamma(N)$ and gives their action on Fourier expansions for $\gcd(n,N) = 1$. As noted by Kato, this normalization commutes with the action of $\GL_{2}(N)$, but it does not preserve the connected components of~$X(N)$.

Let $w$ be the width of the cusp at $\infty$ for $\Gamma_H$, and suppose that $f(z) = \sum_{n=0}^{\infty} a(n) q^{n/w}$ is a modular form of weight $k$ for $\Gamma_{H}$, $q = e^{2 \pi i z}$. Then Kato's formulas show that if $p$ is a prime with $p \equiv \pm 1 \pmod{N}$
\[
\left(\sum_{n=0}^{\infty} a(n) q^{n/w}\right) | T(p) =
\sum_{n=0}^{\infty} a(pn) q^{n/w} + p^{k-1} \sum_{n \equiv 0 \pmod{p}}^{\infty} a(n/p) q^{n/w}.
\]
The case $p \equiv -1 \pmod{N}$ requires $M \coloneqq \smallmat{1}{0}{0}{-1} \in H$, so we conjugate~$H$ to contain $M$.

We choose the smallest prime $p \equiv \pm 1 \pmod{N}$ and apply the Hecke operator $T(p)$ to all the products of three weight~$2$ Eisenstein series. While we are not guaranteed that we obtain the entire space $M_{6}(\Gamma(N), \Q(\zeta_{N}))^{H}$, in the cases under consideration we do. It is not difficult to see how the Hecke operator affects values at the cusps and in this way, we obtain a basis for the space $S_{6}(\Gamma(N), \Q(\zeta_{N}))^{H}$. We apply the same tools to obtain a basis for the space $S_{8}(\Gamma(N), \Q(\zeta_{N}))^{H}$.

Let $h$ be a meromorphic modular form on $\Gamma(N)$.  We use the elementary observation that, since the zero sets of $E_{4}(z)$ and $E_{6}(z)$ are disjoint, $h\in S_{2}(\Gamma(N), \Q(\zeta_{N}))^{H}$ if and only if $h E_{4} \in S_{6}(\Gamma(N), \Q(\zeta_{N}))^{H}$ and $h E_{6} \in S_{8}(\Gamma(N), \Q(\zeta_N))^{H}$. In this way, one obtains $q$-expansions for the weight~$2$ cusps forms for $H$. From these we can easily obtain the canonical map for $X_{H}$ and its image in $\P^{g-1}$, and then check if the curve is hyperelliptic or not. Let $\Omega$ be the sheaf of holomorphic $1$-forms on the curve $X_{H}$. If $g(H) \geq 3$ and $X_{H}$ is not hyperelliptic, the canonical ring $R = \oplus_{d \geq 0} H^{0}(X_{H}, \Omega^{\otimes d})$ is generated in degree~$1$ \cite{VoightZB-The-canonical-ring-of-a-stacky-curve}*{Chapter 2}.

Finally, we seek to identify $j(z)$ as a ratio of two elements of the canonical ring. From a computational standpoint, it is easier to write
\[
j(z) = \frac{1728 E_{4}(z)^{3}}{E_{4}(z)^{3} - E_{6}(z)^{2}}
\]
and find representations of $E_{4}(z)$ and $E_{6}(z)$ as ratios of elements of the canonical ring. Let~$h(z)$ be an element of the canonical ring of weight $k-4$. Then $E_{4}(z) = \frac{f(z)}{h(z)}$ for some $f(z)$ in the canonical ring if and only if $h(z) E_{4}(z) dz^{k/2}$ is a holomorphic $k/2$-form. Diamond and Shurman \cite{DiamondS:modularForms}*{Page 86} show that
\[
{\rm div}(f(z) dz^{k/2}) = {\rm div}(f) - \sum \frac{k}{4} x_{2,i} -
\sum \frac{k}{3} x_{3,i} - \sum \frac{k}{2} x_{i}
\]
where $\{x_{2,i} \}$ is the set of elliptic points of order~$2$, $\{ x_{3,i} \}$ is the set of elliptic points of order $3$, and $\{ x_{i} \}$ is the set of cusps on $X_{H}$. Let $D = -\sum x_{2,i} - \sum x_{3,i} - \sum 2x_{i}$. Then, $h(z) E_{4}(z) dz^{k/2}$ is holomorphic if and only if $h(z) dz^{(k-4)/2} \in L^{(k-4)/2}(D)$, the space of meromorphic $(k-4)/2$-forms with poles at most at $D$. We have (by Lemma VII.4.11 from Miranda \cite{Miranda}) that $L^{(k-4)/2}(D) \simeq L(D + ((k-4)/2)K)$, where $K$ is the canonical divisor. Riemann--Roch now guarantees the existence of such a $D$ provided that
\[
k \geq \frac{2\nu_{\infty}(\Gamma_{H}) + \nu_{2}(\Gamma_{H}) + \nu_{3}(\Gamma_{H}) + 5g(H)-4}{g(H)-1}.
\]
A similar calculation shows that $E_{6}$ is a ratio of a holomorphic $(k/2)$-form and a holomorphic $((k-6)/2)$-form provided that
\[
k \geq \frac{3 \nu_{\infty}(\Gamma_{H}) + \nu_{2}(\Gamma_{H}) + 2 \nu_{3}(\Gamma_{H}) + 7g(H)-6}{g(H)-1}.
\]
With these bounds, one can use the $q$-expansions of the weight~$2$ cusp forms to compute the $q$-expansions of the modular forms of weight $k$ that correspond to elements of the canonical ring, and use linear algebra to compute a ratio that is equal to $E_{4}(z)$ and $E_{6}(z)$.

This technique is used to compute the canonical models of $X_{\ns}^{+}(25)$,
$X_{\ns}^{+}(27)$, a genus~$43$ curve with label
\glink{27.729.43.1}, and a genus~$36$ curve with label \glink{25.625.36.1}.
A variant of this technique was used to compute
a model of a genus $3$ modular curve defined over $\Q(\zeta_{3})$ (see Subsection \ref{ssec:ns27}). We compute models for all of the arithmetically maximal modular
curves with $3 \leq \ell \leq 13$ with the exceptions of $X_{{\rm s}}^{+}(49)$ with label \glink{49.1372.94.1} which is known not to have any non-cuspidal non-CM rational points, \glink{121.605.41.1} which is shown not to have any rational points in Section~\ref{ssec:genus-41}, and $X_{\ns}^{+}(121)$ with label \glink{121.6655.511.1}. For the genus $69$ curve $X_{\ns}^{+}(49)$, we obtain a singular plane model, but we are unable to apply the method of Novocin and van Hoeij to simplify the equation.

\begin{example}
Consider the modular curve $X_H= X_{\ns}^{+}(27)$ with label \glink{27.243.12.1}. We compute the Fourier expansions of the $8$ Eisenstein series in $M_{2}(\Gamma(27),\Q(\zeta_{27}))^{H}$, which are elements of $\Q(\zeta_{27})[[q^{1/27}]]$, and multiply triples of these together. There are $120$ such triples, however, the subspace of $M_{6}(\Gamma(27),\Q(\zeta_{27}))^{H}$ with triple zeros at the $19$ elliptic points of order~$2$ has dimension $82$. The first $82$ triples that we try are linearly independent in $M_{6}(\Gamma(27),\Q(\zeta_{27}))^{H}$.

The smallest prime $p \equiv \pm 1 \pmod{27}$ is $p = 53$. We apply the Hecke operator $T(53)$ to the first $19$ weight $6$ forms we computed in the last step and find that these are linearly independent, while for the remaining $63$ forms $g$, $g | T(53)$ is a linear combination of the forms we already have. This gives us a basis for $M_{6}(\Gamma(27),\Q(\zeta_{27}))^{H}$, which has dimension $101$. We then identify the subspace of cusp forms, which has dimension $92$.

Next we find a basis for $M_{8}(\Gamma(27),\Q(\zeta_{27}))^{H}$ by multiplying together $4$-tuples of weight~$2$ Eisenstein series. We obtain $113$ linearly independent forms this way, and obtain $38$ additional forms by applying $T(53)$. This gives us a basis for $M_{8}(\Gamma(27),\Q(\zeta_{27}))^{H}$ which has dimension $151$. The subspace of cusp forms has dimension $142$.

Next, we compute the subspace of $\Q(\zeta_{27})[[q^{1/27}]]$ spanned by the forms $g_{6}/E_{4}$, where $g_{6} \in S_{6}(\Gamma(27),\Q(\zeta_{27}))^{H}$ and intersect this with the subspace of $\Q(\zeta_{27})[[q^{1/27}]]$ spanned by the forms $g_{8}/E_{6}$, where $g_{8} \in S_{8}(\Gamma(27),\Q(\zeta_{27}))^{H}$. This subspace has dimension $12$ and so it must equal $S_{2}(\Gamma(27),\Q(\zeta_{27}))^{H}$. We compute an LLL-reduced basis for this latter space, and compute an LLL-reduced basis for the space of quadratic relations and find that it has dimension $45$. This shows that the image of the canonical map is not a rational normal curve (for in this case, the space of quadratic relations would have dimension $55$), and hence $X_H$ is not hyperelliptic and the curve $X_H$ is given by the intersection of our $45$ quadrics in $\P^{11}$.

The Riemann--Roch calculation described above shows that $E_{4}$ is a ratio of an element of $H^{0}(X_H, \Omega^{5})$ and one in $H^{0}(X_H,\Omega^{3})$, while $E_{6}$ is a ratio of elements in $H^{0}(X_H,\Omega^{6})$ and one in $H^{0}(X_H,\Omega^{3})$. We obtain Fourier expansions of the modular forms of weight $2r$ corresponding to elements of $H^{0}(X_H,\Omega^{r})$ by multiplying together forms in $S_{2}(\Gamma(27),\Q(\zeta_{27}))^{H}$. A linear algebra calculation then finds (many) representations of $E_{4}$ and $E_{6}$ in the desired form. From this, the map $j \colon X_H \to \P^{1}$ can be constructed.
 
We do not reproduce the model in this paper (the file \repolink{canmodels}{modelfile27.243.12.1.txt} in \cite{RouseSZB:magma-scripts-ell-adic-images} contains \magma code to compute it), but the coefficients are quite small: the largest absolute value of a coefficient is $4$. Point searching on this model finds $8$ rational points. Two of these $8$ points are in the base locus of the map from $X_H \to \P^{1}$ and some more care is needed to compute their image on the $j$-line (one maps to
$j = 1728$ and one maps to $j = -884736000$). The remaining~$6$ points map to $j = -3375$, $j = 287496$, $j = -884736$, $j = 16581375$, $j = -147197952000$ and $j = -262537412640768000$, all of which are CM points. The total time required for the computation of the model and the map to the $j$-line is about $12$ minutes.
\end{example}

%----------------------------------------------------------------------------
\section{Analysis of Rational Points}
\label{sec:analysis-of-rational-points-small-genus}
%----------------------------------------------------------------------------

In this section we determine the rational points on $X_H(\Q)$ for each of the 21 arithmetically maximal subgroups that appear in Table~\ref{table:cases} without an asterisk.

%----------------------------------------------------------------------------
\subsection{Genus 1}
\label{ssec:genus-1}
%----------------------------------------------------------------------------

The three curves of genus 1 listed in Table~\ref{table:cases} have labels
\[
 \glink{9.12.1.1},\hspace{10pt} \glink{9.54.1.1}, \hspace{10pt} \glink{27.36.1.2},
\]
whose rational points are as follows:
\begin{itemize}
\item \glink{9.12.1.1} is a degree $3$ cover of $X_{0}(3)$ isomorphic to the elliptic curve \href{https://www.lmfdb.org/EllipticCurve/Q/27/a/4}{$y^{2} + y = x^{3}$}.  It has three rational points, two of which are cusps; the other is a CM point with $j = 0$.

\item \glink{9.54.1.1} also known as $X_{\spc}^{+}(9)$ is also isomorphic to $y^{2} + y = x^{3}$ (different $j$-map). In this case, the curve has one cusp and two CM points: $j = 8000$ and $j = -32768$.

%\item \glink{9.81.1.1} is isomorphic to the plane cubic given on the bottom of page 5 of \cite{ElkiesSurj3}: $x^{3} - 3x^{2} y - 6xyz + 6xz^{2} - 2y^{3} + 3y^{2} z - 9 yz^{2} + 5z^{3} = 0$.  It has no $3$-adic points. 

\item \glink{27.36.1.2} is isomorphic to the elliptic curve \href{https://www.lmfdb.org/EllipticCurve/Q/243/b/2}{$y^2 + y = x^3 + 2$}; it has three rational points, two of which are cusps; the other is a CM point with $j = 0$.
\end{itemize}

The file \repolink{ratpoints}{check-for-sporadic-points.m} in \cite{RouseSZB:magma-scripts-ell-adic-images} verifies these computations.

%----------------------------------------------------------------------------
\subsection{Genus 2}
\label{ssec:genus-2}
%----------------------------------------------------------------------------

The curves corresponding to the labels
\[
  \glink{9.36.2.1}, \hspace{10pt} \glink{27.36.2.1}, \hspace{10pt} \glink{27.36.2.2}, \hspace{7pt} \text{and} \hspace{7pt} \glink{9.54.2.2}
\]
have genus 2 and rank 0. The first two curves are isomorphic and have minimal Weierstrass model $y^{2} + (x^{3} + 1)y = -5x^{3} - 7$, and  $y^{2} + (x^{3} + 1)y = -2x^{3} - 7$ is a  minimal Weierstrass model for the third curve. Using \magma's \magmacmd{RankBound} command, it is straightforward to see that the Jacobians of these curves have rank zero and using the command \magmacmd{Chabauty0} one finds that each has two rational points. In each case the rational points are cusps.

The fourth curve, with minimal Weierstrass model
\[
y^{2} + (x^{2} + x + 1)y = 3x^{6} + 2x^{4} + 10x^{3} + 6x^{2} - 32x + 14,
\]
has no $3$-adic points.
\medskip

The file \repolink{ratpoints}{genus-2-analysis.m} in \cite{RouseSZB:magma-scripts-ell-adic-images} verifies these computations.

%----------------------------------------------------------------------------
\subsection{Genus 3}
\label{ssec:genus-3}
%----------------------------------------------------------------------------

The curves corresponding to the labels
\[
\glink{9.108.3.1}, \hspace{10pt} \glink{27.36.3.1}, \hspace{7pt} \text{and} \hspace{7pt} \glink{27.108.3.1}
\]
have genus 3 and are not hyperelliptic. The first has canonical model
\[
  X \colon x^{3} z - 6x^{2} z^{2} + 3xy^{3} + 3xz^{3} + z^{4} = 0.
\]
Since $X$ is a Picard curve (which is clear from the $y$ coordinate), we can use \magma's \magmacmd{RankBound} command to compute that the rank of $J_X(\Q)$ is zero. In principle, one can now simply compute $J_X(\Q) = J_X(\Q)_{\tors}$ and then compute $X(\Q)$ as the preimage of an Abel--Jacobi map $X(\Q) \hookrightarrow J_X(\Q)$; see \cite{derickxEHMZB}*{Subsection 5.2} for a longer discussion of such computations.

In practice, we were not able to compute $J_X(\Q)_{\tors}$ exactly.  By \cite{katz:galois-properties-of-torsion}*{Appendix}, reduction modulo an odd prime $p$ gives an injection $J_X(\Q)_{\tors} \hookrightarrow J_X(\F_p)$. We compute (using \magma's \magmacmd{ClassGroup} command) that $J_X(\F_5) \simeq \Z/198\Z$ and $J_X(\F_{13}) \simeq \Z/45\Z \times \Z/45\Z$. (See \cite{derickxEHMZB}*{Subsection 4.1} for a longer discussion of computing torsion on modular Jacobians).  On the other hand, $X$ has two visible rational points with coordinates $[0 : 1 : 0]$,  $[1 : 0 : 0]$ whose difference is a point of order 3 in $J_X(\Q)$. Thus the only two possibilities for $J_X(\Q)$ are $\Z/3\Z$ and $\Z/9\Z$. A long search for low degree points on $X$ does not produce a point of order 9 on $J_X(\Q)$, nor do further local computations (including refined techniques such as \cite{ozman:quadraticpoints}*{Section 4}) give a better local bound than $\Z/9\Z$. This suggests a local to global failure for $J_X(\Q)_{\tors}$.

We circumvent this as follows. We enumerate $X(\F_2)$ and find that for any $P,Q \in X(\F_2)$, the order of $P - Q$ is not equal to 9. Since $J_X(\Q)$ has rank 0 and odd order, the reduction map $J_X(\Q)_{\tors} \hookrightarrow J_X(\F_2)$ is injective by  \cite{katz:galois-properties-of-torsion}*{Appendix}. Thus, there also do not exist points  $P,Q \in X(\Q)$ such that the order of $P - Q$ is equal to 9 (but note that this does not exclude the possibility that $J_X(\Q) \simeq \Z/9\Z$).

Finally, we compute the preimage of the remaining points under an Abel--Jacobi map. Fix $P = [0 : 1 : 0]$, $P' = [1 : 0 : 0]$, and $D = P'-P$. Let $\iota_P \colon X \hookrightarrow J_X$ be the map $Q \mapsto Q-P$. By the previous paragraph, for any $Q \in X(\Q)$, $\iota_P(Q)$ has order 1 or 3. By definition, $\iota_P(P) = 0\cdot D$, and $\iota_P(Q) = 1\cdot D$. Using \magma's \magmacmd{RiemannRoch} command, we find that the linear system $|2\cdot D + P|$ is empty, so there is no $Q \in X(\Q)$ such that $\iota_P(Q) = 2\cdot D$. Since this exhausts $J_X(\Q)_{\tors}$, we have computed $X(\Q)$. 
Using our equations for the $j$-map, we determine that the two known points are both CM points with $j = 0$.

The second curve has canonical model 
\[
  X \colon -x^{3} y + x^{2} y^{2} - xy^{3} + 3xz^{3} + 3yz^{3} = 0
\]
and is also a Picard curve (which is clear from the $z$ coordinate).  It has an automorphism $\iota \colon X \to X$ that swaps $x$ and $y$. The quotient is the elliptic curve $E\colon y^{2} + y = x^{3} + 20$ which has rank one and corresponding newform \mflabel{243.2.a.a}. Applying \magma's \magmacmd{RankBound} command reveals that the rank of $J_X(\Q)$ is also one. In particular, $J_X$ is isogenous to $E \times A$ where $A$ is the abelian surface $A$ corresponding to the newform \mflabel{243.2.a.d}; see Example~\ref{example:27.36.3.1}. The analytic rank of $A(\Q)$ is zero; by Corollary \ref{coro:kolyvagin} the algebraic rank is also zero. 

One could in principle apply Chabauty's method, but this is not directly implemented in \magma for non-hyperelliptic curves. We present an alternative argument. For any point $P \in X(\Q)$, $P$ and $\iota(P)$ have the same image in $E(\Q)$, so $P - \iota(P)$ has trivial image in $E(\Q)$. Thus, since $A(\Q)$ is torsion, the image of $P - \iota(P)$ in $E(\Q) \times A(\Q)$ is also torsion; since the map $J(\Q) \to E(\Q) \times A(\Q)$ is finite, we conclude that $P - \iota(P)$ is an element of $J_X(\Q)_{\tors}$.

It thus suffices to compute $J_X(\Q)_{\tors}$ and the preimage of the map
\begin{align*}
\tau\colon X(\Q) &\to J_X(\Q)_{\tors}\\
P &\mapsto P - \iota(P).
\end{align*}
Local computations show that $J_X(\Q)_{\tors}$ is isomorphic to a subgroup of $\left(\Z/3\Z\right)^2$, and differences of rational points generate a group of order~9.

To compute the preimage of $\tau$ we first note that the $\tau$ is injective away from the fixed points of $\iota$. Indeed, if $P - \iota(P) = Q - \iota(Q)$ as divisors, then either $P = Q$, or $P = \iota(P)$ and $Q = \iota(Q)$. Now suppose that $P - \iota(P) \neq Q - \iota(Q)$ as divisors, but $P - \iota(P) \sim Q - \iota(Q)$. If $P = \iota(P)$ but $Q \neq \iota(Q)$, this equivalence gives a $g^1_1$, which is a contradiction since $X$ is not rational (and similarly if $Q = \iota(Q)$ but $P \neq \iota(P)$).
If $Q \neq \iota(Q)$ and $P \neq \iota(P)$, then instead this gives a $g^1_2$ on a non-hyperelliptic curve, which is also a contradiction. Thus $\tau$ is injective away from the fixed points of $\iota$.

Next, differences of rational points give an explicit basis $D_1, D_2$ for $J_X(\Q)_{\tors}$. For each $a,b \in \Z/3\Z$ such that $(a,b) \neq (0,0)$, we either find an explicit point $P$ for which we have $P - \iota(P) = aD_1 + bD_2$, or we find a prime $p$ such that there is no $P \in X(\F_p)$ such that $P - \iota(P) = aD_1 + bD_2$ in $J_X(\F_p)$ (essentially, ``sieving'') and thus no such $P \in X(\Q)$. If $(a,b) = (0,0)$, then $P - \iota(P) = aD_1 + bD_2$ if and only if $P = \iota(P)$, i.e., if and only if $P$ is a fixed point of the involution $\iota$; there are finitely many such points and they are straightforward to compute.
Using our equations for the $j$-map we find that two of the known points are cusps and one is a CM point with $j = 0$.

The third curve is isomorphic to the first,
% this isomorphism can be explained by isogenies, using the theory in Shimura's book - Section 6.7
which is easy to verify with the explicit models. The 2 rational points are not exceptional; they are both CM points with $j = 0$.

\medskip
The files \repolink{ratpoints}{9.108.3.1.m}, \repolink{ratpoints}{27.36.3.1.m}, \repolink{ratpoints}{27.108.3.1.m} in \cite{RouseSZB:magma-scripts-ell-adic-images} verify these computations.

%----------------------------------------------------------------------------
\subsection{Genus 4}
\label{ssec:genus-4}
% ----------------------------------------------------------------------------

The curves corresponding to the labels
\[
  \glink{9.108.4.1}, \hspace{10pt} \glink{27.108.4.1}, \hspace{10pt} \glink{27.108.4.3}, \hspace{7pt} \text{and} \hspace{7pt} \glink{27.108.4.5}
\]
have genus 4. The fourth curve (with label \glink{27.108.4.5}) has no $\F_7$ points (see Section \ref{sec:enumeration-of-moduli}), and computations with \magma show that the other 3 are all isomorphic. Using the method of \cite{vHoeijN:A-Reduction-Algorithm-for-Algebraic-Function-Fields}, we find that they are all isomorphic to the Picard curve
\[
X\colon t(t+1) y^{3} = t^{3} - 3t^{2} - 6t - 1.
\]
(We note that while \cite{vHoeijN:A-Reduction-Algorithm-for-Algebraic-Function-Fields} is unpublished, and thus has not been refereed, in our example it is straightforward to verify that our curve is isomorphic to this one.)
The Jacobian of the curve $X$ is isogenous to the modular abelian variety \mflabel{81.2.c.b}, which has analytic rank 0.

From here the argument is identical to those in Subsection \ref{ssec:genus-3}. Local computations show that $J_X(\Q)_{\tors}$ is isomorphic to a subgroup of $\left(\Z/3\Z\right)^2$, and differences of rational points generate a group of order 9. Using \magma's \magmacmd{RiemannRoch} command, we compute the preimage of $J(\Q)$ under an Abel--Jacobi map, and find that $\#X(\Q) = 3$; via Section \ref{sec:enumeration-of-moduli} we compute a priori that for each of the 3 curves there are 3 rational cusps. 

% Note: \magma has routines that will compute the fake $\phi$-Selmer set, where $\phi$ is the endomorphism of the Jacobian given by $\phi = 1 - \zeta_{3}$. This finished for g = 3, but not for the g = 6 case in the next section

The files \repolink{ratpoints}{9.108.4.1.m}, \repolink{ratpoints}{27.108.4.1.m}, and \repolink{ratpoints}{27.108.4.3.m} in \cite{RouseSZB:magma-scripts-ell-adic-images} verify these computations.

%----------------------------------------------------------------------------
\subsection{Genus 6}
\label{ssec:genus-6}
% ----------------------------------------------------------------------------

Let $H$ be the subgroup labeled \glink{27.108.6.1}. 
The curve $X_H$ has genus $6$; it is a Picard curve, and admits an automorphism $\iota$ of order~$2$ defined over $\Q$. The quotient curve $Y$ has genus $3$ and appears to have $11$ rational points; its Jacobian $J_Y$ is geometrically simple and is isogenous to the modular abelian variety \mflabel{243.2.a.f}, which has analytic ran~3. The Jacobian $J_H$ of $X_H$ is isogenous to $J_Y \times A$ where $A$ is the geometrically simple modular abelian variety \mflabel{243.2.a.e} which has analytic rank 0.

We are thus in a similar situation to the group with label \glink{27.36.3.1} from Subsection \ref{ssec:genus-3}: every rational point $P \in X_H(\Q)$ has the property that $P - \iota(P)$ is torsion in $J(X_H)$. Local computations show that $J_H(\Q)_{\tors}$ is isomorphic to a subgroup of $\left(\Z/3\Z\right)^2$. (In fact, we can show that $J_H(\Q)_{\tors} \simeq \left(\Z/3\Z\right)^2$, but we will not need this fact.)  

As in Subsection \ref{ssec:genus-3}, the map 
\begin{align*}
\tau\colon X(\Q) &\to J_H(\Q)_{\tors}\\
P &\mapsto P - \iota(P).
\end{align*}
is injective away from the fixed points of $\iota$. Additionally, since $\#J_H(\Q)_{\tors}$ is odd, the reduction map
\[
J_H(\Q)_{\tors} \to J_H(\F_2)
\]
is injective \cite{katz:galois-properties-of-torsion}*{Appendix}. The involution $\iota$ has two fixed points $P$ and $Q$, both of which are rational. Let $U = X_H - \{P,Q\}$ and consider the commutative diagram
\[
\xymatrix{
U(\Q) \ar[r]^{\tau} \ar[d]_{\red_{U,2}} & J_H(\Q)_{\tors} \ar[d]^{\red_{J,2}} \\
U(\F_2) \ar[r]^{\tau_2} & J_H(\F_2).
}
 \]
Since $\tau$ and $\red_{J,2}$ are injective, $\red_{U,2}$ is also injective. Finally, we compute in \magma that $\#U(\F_2) = 2$ and that both points lift to $\Q$. We conclude that $\#X(\Q) = 4$.

Finally, using our equations for the $j$-map, we compute that the four known points are CM with $j = 0$ and $j=54000$.

The file \repolink{ratpoints}{27.108.6.1.m} in \cite{RouseSZB:magma-scripts-ell-adic-images} verifies these computations.

%----------------------------------------------------------------------------
\subsection{Genus 36 and 43:}
\label{ssec:genus-36-43}
%----------------------------------------------------------------------------

Let $H$ be the group with label \glink{27.729.43.1}. The curve $X_H$ has genus $43$ and is a degree~$27$ cover of the curve with label \texttt{9.27.0.1}, the modular curve that is the subject of the paper \cite{ElkiesSurj3}. If $E/\Q$ is an elliptic curve and the image of $\rho_{E,3^{\infty}}$ is contained in $H$, then $\Q(E[27]) = \Q(E[3],\zeta_{27})$. This remarkable situation is incompatible with most known possibilities for $\rho_{E,3^{\infty}}(G_{\Q_3})$, which suggests that $X_H$ has no $3$-adic points.

We verify this as follows. Using the methods described in Section~\ref{sec:comp-equat-x_h}, we compute the canonical model for $X_H$ in $\P^{42}$. We successfully compute a basis for the space of weight $6$ and weight $8$ modular forms for $H$ (which have dimensions $318$ and $438$ respectively). In general, a non-hyperelliptic curve of genus $g$ contains $\binom{g-2}{2}$ quadrics as generators for the canonical ideal, and these generate the canonical ideal provided $g \geq 4$, the curve is not trigonal, and if $g = 6$, the curve is not a smooth plane quintic \cite{Petri1923}, \cite{GreenLazarsfeld:Simple}*{Corollary 1.7, Remark 1.9}. Indeed, the curve $X_H$ is not trigonal, and so our canonical model is given by the vanishing of $\binom{41}{2} = 820$ quadrics in $\P^{42}$.

The model computation takes about 10 hours. We are able to represent $E_{4}$ as a ratio of a degree $4$ and a degree $2$ element in the canonical ring, and represent $E_{6}$ as a ratio of a degree $5$ and a degree $2$ element in the canonical ring. However, memory limitations do not allow us to use these to represent $j$ as a ratio of elements in the canonical ring. Using the expressions for $E_{4}$ and $E_{6}$ to compute $j$ gives a ratio of elements each of degree $18$. Note that a generic homogeneous degree $18$ polynomial in $43$ variables is a linear combination of $\binom{60}{18} = 925029565741050$ monomials.

We use this model to show that $X_H$ has no $3$-adic points using a suggestion
of David Zywina. First we find that the reduction mod $3$ of our model has $19$ points over $\F_{3}$; this reduction is not one dimensional. For each point $(a_{0} : \cdots : a_{42})$ on $X_{H}(\F_{3})$, we plug $a_{i} + 3x_{i}$ into the defining equations for $X_{H}/\Z$; here $a_{i}$ and $x_{i}$ represent the first two digits in the $3$-adic expansion of the $i$th coordinate. Dividing by $3$ and reducing modulo $3$ represents the mod $9$ lifts of
$(a_{0} : \cdots : a_{42})$ as a system of linear equations in $x_{0}, \ldots, x_{42}$. We find that in all cases this system of linear equations has no solutions
and so $X_{H}(\Z/9\Z)$ is empty. The total running time to verify this
is approximately 20 seconds. This technique computes the fibers
of the map ${\rm Gr}_{2}(X_{H})(\F_{p}) \to X_{H}(\F_{p})$, where ${\rm Gr}_{2}$ denotes the level 2 Greenberg functor. For a short introduction to the
Greenberg transform, see \cite{poonen:computing-torsion-points}*{Section~3}.

The group $H$ with label \glink{25.625.36.1} is similarly remarkable and has the following properties: $\#H(25) = 5\cdot\#H(5) = \#\GL_2(\F_5)$ (i.e., the kernel reduction map $H(25)\to H(5)$ has order $5$), the image of $H(5)$ in $\PGL_{2}(\F_5)$ is $S_4$, and if $E/\Q$ is an elliptic curve such that $\rho_{E,5^{\infty}}(G_{\Q}) \leq H$, then $\Q(E[25]) = \Q(E[5], \zeta_{25})$ or $\zeta_{25} \in \Q(E[5])$. The procedure described above for the group \glink{27.729.43.1} also works here, and shows that $X_H(\Z/25\Z)$ is empty.

The files \repolink{ratpoints}{27.729.43.1.m} and \repolink{ratpoints}{25.625.36.1.m} in \cite{RouseSZB:magma-scripts-ell-adic-images} verify these computations.

%----------------------------------------------------------------------------
\subsection{Genus 41: an equationless Mordell--Weil sieve}
\label{ssec:genus-41}
%----------------------------------------------------------------------------

Let $H$ be the subgroup with label \glink{121.605.41.1}. The associated curve $X_H$ has genus 41 and its Jacobian has analytic rank~41. We were unable to compute a model for $X_H$, but will show that $X_H(\Q)$ is empty by Mordell--Weil sieving without using equations (in other words, sieving ``via moduli'').

The reduction $H(11)$ mod 11 is a subgroup of $N_{\ns}(11)$, so there is a map $\pi\colon X_H \to X_{\ns}^+(11)$. The curve $X_{\ns}^+(11)$ is an elliptic curve, with Mordell--Weil group $X_{\ns}^+(11)(\Q) \simeq \Z$. We will perform a Mordell--Weil sieve (see \cite{bruin2010MWsieve}), in the following sense: for a set $S$ of primes of good reduction (i.e., $11 \not \in S$) the diagram
\[
\xymatrix{
X_H(\Q) \ar[r]^{\pi} \ar[d]_{\alpha} & X_{\ns}^+(11)(\Q) \ar[d]^{\beta} \\
\prod_{p\in S} X_H(\F_p) \ar[r]^{\pi_S} & \prod_{p\in S}X_{\ns}^+(11)(\F_p).
}
 \]
is commutative. We claim that for $S = \{13, 307\}$ the intersection of the images of the maps $\pi_S$ and $\beta$ is empty, and in particular $X_H(\Q)$ is empty.

In principle, this is a straightforward computation to implement. In practice, we do not have explicit equations for the map $\pi\colon X_H \to X_{\ns}^+(11)$. Nonetheless, for $p \neq 11$, we can determine whether a point $P \in X_{\ns}^+(11)(\F_p)$ is the image under $\pi$ of some point of $X_H(\F_p)$ by working directly with the (generalized) elliptic curve $E$ corresponding to $P$, as follows. The mod 11 image $\rho_{E,11}(G_{\F_p})$ is a subgroup of $N_{\ns}^+(11)$. To prove there does not exist a point $Q \in X_H(\F_p)$ such that $\pi(Q) = P$, it suffices to show that the mod~121 image $\rho_{E,121}(G_{\F_p})$ is not a subgroup of $H$. These images are cyclic groups generated by the image of Frobenius.

The reduction mod $N$ of the matrix $A_{\pi}$ from Equation \ref{eq:Api} is the matrix of Frobenius acting on $E[N](\overline{\F}_p)$ (with respect to some basis), and is computed by the function \magmacmd{GL2nTorsionFrobenius} in the file \repolink{groups}{gl2.m} in \cite{RouseSZB:magma-scripts-ell-adic-images}. We compute this and check whether it (or a conjugate) is contained in $H$; if not, then $P$ is not in the image of $\pi$ (mod $p$).

We can thus perform the Mordell--Weil sieve ``at the level of moduli''. For our chosen generator $R$ of $X_{\ns}^+(11)(\Q)$ we find that, for $p = 13$, a hypothetical point of $X_H(\Q)$ maps to $n\cdot R$, where $n$ is 1 or 5 mod 7, and for $p = 307$, any point of $X_H(\Q)$ maps to $n\cdot R$, where $n$ is 2, 3, 4, 7, 10 or 13 mod 14; we conclude that $X_H(\Q)$ is empty.

The file \repolink{ratpoints}{121.605.41.1.m} in \cite{RouseSZB:magma-scripts-ell-adic-images} verifies these computations.

%----------------------------------------------------------------------------
\section{Comments on the remaining cases}
\label{section:comments-on-the-remaining-cases}
%----------------------------------------------------------------------------

In this section we discuss our (unsuccessful) efforts to determine the rational points on the six curves marked with asterisks in Table~\ref{table:cases}, which also appear in Table \ref{table:remaining-cases}. Each of these curves has no rational cusps and  at least one rational CM point; our analysis suggests that there are no exceptional points on any of these curves, but we are unable to prove this.

%----------------------------------------------------------------------------
\subsection{\texorpdfstring{$\boldsymbol{X_{{\rm ns}}^{+}(27)}$}{X\_ns+(27)}}
\label{ssec:ns27}
% ----------------------------------------------------------------------------

While we are not able to determine the rational points on $X = X_{{\rm ns}}^{+}(27)$ (labeled \glink{27.243.12.1}), we will describe our attempts. As described in Section~\ref{sec:comp-equat-x_h}, we compute the canonical model of $X_{{\rm ns}}^{+}(27)$, which has genus $12$.

Rather than try to work directly with $X$, we work with a modular
curve $X_{H}$ for which $\det H$ is not surjective and whose geometric connected components are defined over $\Q(\zeta_{3})$, rather than
$\Q$. Although $N_{{\rm ns}}(27)$ is a maximal subgroup of index $9$
in $N_{{\rm ns}}(9)$, if we let $D = \{ g \in \GL_{2}(\Z_{3}) \colon
\det(g) \equiv 1 \pmod{3} \}$, then there is a non-trivial subgroup
between $N_{{\rm ns}}(27) \cap D$ and $N_{{\rm ns}}(9) \cap D$. This
is the subgroup $H = \left\langle \smallmat{0}{26}{4}{6}, \smallmat{10}{1}{25}{26}\right\rangle \le \GL_{2}(27)$ with label \texttt{3.2.1-27.81.3.1}. Let $X_{H}'$ denote one of the connected components of $X_{H, \Q(\zeta_{3})}$ (the components are all isomorphic).
The curve $X_{H}'$ has
genus $3$, and we use a variant of the method from
Section~\ref{sec:comp-equat-x_h} to compute a model. We find that
$X_{H}'$ is the smooth plane quartic
\begin{align*}
a^{4} &+ (\zeta_{3} - 1) a^{3} b + (3 \zeta_{3} + 2) a^{3} c - 3a^{2}c^{2} +
(2 \zeta_{3} + 2) ab^{3} - 3 \zeta_{3} ab^{2} c\\
&+ 3 \zeta_{3} abc^{2} - 2 \zeta_{3} ac^{3} - \zeta_{3} b^{3} c + 3 \zeta_{3} b^{2} c^{2} + (-\zeta_{3} + 1) bc^{3} + (\zeta_{3} + 1) c^{4} = 0.
\end{align*}
There is a degree $3$ morphism $X \to X_{H}'$ defined over $\Q(\zeta_{3})$,
and so every $\Q(\zeta_{3})$ point of $X$ maps to a $\Q(\zeta_{3})$ point of $X_{H}'$. There are at least $13$ points in $X_{H}'(\Q(\zeta_{3}))$. These include
the images of the $8$ rational points on $X_{{\rm ns}}^{+}(27)$, four more CM points with discriminant $-4$, and a non-CM point corresponding to an elliptic curve with $j$-invariant $2^{3} \cdot 5^{3} \cdot (1-\zeta)^{6} \cdot (5+3\zeta)^{-27} \cdot (27+13\zeta)^{3} \cdot (54+49\zeta)^{3} \cdot (227+173\zeta)^{3}$. The restriction of scalars $\Res_{\Q(\zeta_{3})/\Q} J(X'_H)$ is isogenous to $J_H$, and Theorem~\ref{thm:finalans} shows that $J_{H}$ is isogenous to the $\Q$-simple abelian
variety associated to \mflabel{729.2.a.c}, with analytic rank~$6$.
  
We searched for points on $X_{{\rm ns}}^{+}(27)$ by enumerating rational points on $X_{{\rm ns}}^{+}(9) \cong \P^{1}$ and seeing if they lift to $X_{H}'$. We tested
all points on $X_{{\rm ns}}^{+}(9)$ with height $\leq 1000$, and the only points
we found that lifted were $8$ CM points (one each for $D=-4$,$-7$, $-16$, $-19$, $-28$, $-43$, $-67$, $-163$). The computation takes approximately two hours.

It is not hard to find a divisor $D$ supported at $4$ of the $\Q(\zeta_{3})$ points of $X_{H}'$ whose image in $J(X_{H}')$ has order $3$. Using this point one
can construct a family $\{ D_{\delta} \}$ of 9 \'etale triple covers of $X_{H}'$. Counting points on these genus $7$ curves over finite fields strongly suggests that each of these genus $7$ curves $D_{\delta}$ maps to an elliptic
curve $E/\Q(\zeta_{3})$. One can search for elliptic curves defined over
$\Q(\zeta_{3})$ with bad reduction only at $\langle 1 - \zeta_{3} \rangle$
and use this to predict, for each twist $D_{\delta}$, which elliptic curve it
maps to. For $8$ of the $9$ twists $D_{\delta}$, the elliptic curve in question has rank $0$ or $1$, which should in principle make an elliptic curve Chabauty computation feasible.

Six of the $13$ known $\Q(\zeta_{3})$ points of $X_{H}'$ lift to the final
twist $D_{1}$. Since the cover $D_{1} \to X_{H}'$ is \'etale, each of
those six points on $X_{H}'$ has $3$ $\Q(\zeta_{3})$-rational preimages on
$X_{K}$, and so $D_{1}$ has at least $18$ $\Q(\zeta_{3})$-rational points.
For $D_{1}$, the elliptic curve in question is $E :
y^{2} = x^{3} - 48$, and for this curve \href{https://www.lmfdb.org/EllipticCurve/2.0.3.1/6561.1/CMa/1}{$E(\Q(\zeta_{3})) \simeq \Z/3\Z\times \Z^{2}$}. We explicitly compute the map from $D_{1} \to E$ over
$\Q(\zeta_{3})$ as follows.  We attempt to find the composition
$D_{1} \to E \to \mathbb{P}^{1}$ given by the $x$-coordinate map. If the degree of the map from $D_{1} \to E$ is $d$, then the map $x \colon D_{1} \to \mathbb{P}^{1}$
has degree $2d$ and each pole of $x$ has even multiplicity. First we
pick $\mathfrak{p} = \langle 1 - 2\zeta_{3} \rangle$ to be one of the
prime ideals above $7$ in $\Q(\zeta_{3})$ and we enumerate effective
divisors $D$ of degree $d$ on $D_{1}/\F_{7}$. We search for divisors such
that the Riemann--Roch space of $2D$ is non-trivial. If $f \in L(2D)$
is non-constant, then build the function field of $D_{1}/\F_{7}$ as an
extension of $\F_{7}(f)$ and search to see if there are intermediate
fields. If so, this produces a non-trivial morphism from $D_{1}$ to a
lower genus curve. This technique succeeds for $d = 3$ and we find a
map from $D_{1} \to E$ defined over $\F_{7}$. The $x$-coordinate map
can be written as a ratio of degree $1$ elements of the canonical ring
for $D_{1}$, and the $y$-coordinate map can be written as a ratio of
degree $2$ elements.

From here, we write down a scheme $X$ that parameterizes maps from $D_{1} \to E$
of the form that we found, with the additional constraint that
one particular $\Q(\zeta_{3})$-rational point on $D_{1}$ maps to $(0 : 1 : 0)$.
The map from $D_{1} \to E$ gives rise to a point on $X$ and we use Hensel's
lemma to lift this point to a point modulo $\langle 1 - 2 \zeta_{3} \rangle^{512}$. We then use \magma's \magmacmd{Reconstruct} function to find simple elements
of $\Q(\zeta_{3})$ that reduce modulo $\langle 1-2 \zeta_{3} \rangle^{512}$ to
the point found on $X$. This yields the desired map to $E$.

Using the map $\phi \colon D_{1} \to E$, we can create an \'etale
triple cover of $D_{1}$ via the fiber product
\[
\xymatrix{
D_{1} \times_{E} E \ar[r] \ar[d] & D_{1} \ar[d]^{\phi}\\
E \ar[r]^{\psi} & E}  
\]
where $\psi \colon E \to E$ is a cyclic $3$-isogeny defined over $\Q(\zeta_{3})$ (which exists because $E$ has complex multiplication). The $18$ known
$\Q(\zeta_{3})$ rational points become equidistributed among the nine
twists of $D_{1} \times_{E} E$. Our hope was that we could find a map
from $D_{1} \times_{E} E$, which has genus~$19$, to a lower genus curve.
We can compute the zeta function of $D_{1} \times_{E} E$ over $\F_{4}$
by counting points on it over $\F_{4^{k}}$ for $1 \leq k \leq 12$,
since we already know a $7$-dimensional factor (namely $J(D_{1})$)
of $J(D_{1} \times_{E} E)$. We find that the dimension $12$ ``new'' part of
the Jacobian of $D_{1} \times_{E} E$ is irreducible. We have not been able to
provably determine $D_{1}(\Q(\zeta_{3}))$ or $X_{{\rm ns}}^{+}(27)(\Q)$.

%----------------------------------------------------------------------------
\subsection{\texorpdfstring{$\boldsymbol{X_{{\rm ns}}^{+}(\ell^2)$, $\ell = 5,7,11}$}{X\_ns+(ell*ell), ell = 5,7,11}}
\label{ssec:nssquare}
%----------------------------------------------------------------------------

Using the techniques of Section~\ref{sec:comp-equat-x_h} we have
computed the canonical model of $X_{{\rm ns}}^{+}(25)$, which has
genus $14$. We searched for points on $X_{{\rm ns}}^{+}(25)$ by enumerating
points on $X_{{\rm ns}}^{+}(5) \cong \P^{1}$ with height $\leq 1000$, writing
down the corresponding elliptic curve $E$ and testing if the
pairs $(p+1-\# E(\F_{p}) \mod 25, p \mod 25)$ are in $\{ (\tr g, \det g) : g \in N_{{\rm ns}}(25) \}$. This process rules out all non-CM $j$-invariants
in the range we tested and takes approximately 6 hours.

Point searching finds $8$ rational points all of which are
CM with discriminants $-3$, $-7$, $-12$, $-27$, $-28$, $-43$, $-67$,
and $-163$. The Jacobian has three $2$-dimensional factors and this raises
the possibility that $X_{{\rm ns}}^{+}(25)$ might map to a genus~$2$ hyperelliptic curve; two genus~$2$ curves $X_H$ whose Jacobians are factors of $J(X_{{\rm ns}}^{+}(25))$ are those with labels \glink{25.50.2.1} and \glink{25.75.2.1}. The automorphism group of $X_{{\rm ns}}^{+}(25)$ is trivial over $\Q$.

If $f \colon X_{{\rm ns}}^{+}(25) \to C$ were a map to a genus~$2$ curve,
it seems plausible that the induced map on differentials $\Omega_{C}
\to \Omega_{X_{{\rm ns}}^{+}(25)}$ should respect the action of the Hecke operators.
We have identified the two-dimensional Hecke stable subspaces of $\Omega_{X_{{\rm ns}}^{+}(25)}$ (canonically identified with weight~$2$ cusp forms) and checked
that these do not give rise to a map to a genus~$2$ curve.

For $X_{{\rm ns}}^{+}(49)$ we applied the same technique to search for points
as for $X_{{\rm ns}}^{+}(25)$ and found no non-CM rational points.

For $X_{{\rm ns}}^{+}(121)$ with genus $511$ we did not compute a model.
However, $X_{{\rm ns}}^{+}(11)$ is isomorphic to the elliptic curve $E\colon$\href{https://www.lmfdb.org/EllipticCurve/Q/121b1/}{$y^{2} + y = x^{3} - x^{2} - 7x + 10$} with Mordell--Weil group $E(\Q) \cong \Z$ generated by
$P = (4 : 5 : 1)$. We use the isomorphism computed in \cite{ChenCummins:nonsplitmod11}. We used this to Mordell--Weil sieve (as in Section~\ref{ssec:genus-41}). As a result, if $Q \in X_{{\rm ns}}^{+}(121)$ is a rational point and its image on $X_{{\rm ns}}^{+}(11)$ is $kP$,
then either $k \in \{ -5, -4, -3, -2, -1, 0, 1 \}$ which implies that $Q$ is CM,
or $|k| \geq \lcm(1,2,3,\ldots,233)$. This implies that if $Q$ is a non-CM
point on $X_{{\rm ns}}^{+}(121)$, then the height of $j(Q)$ would need to be roughly on the order of $10^{7.92 \cdot 10^{202}}$ (or larger).

%----------------------------------------------------------------------------
\subsection{The curves \texttt{49.147.9.1} and \texttt{49.196.9.1}}\label{ssec:genus9}
%----------------------------------------------------------------------------

The curve \glink{49.147.9.1} is a degree~$7$ cover of $X_{{\rm ns}}^{+}(7)$, which has genus zero. Using this, we have a singular degree $7$ plane model
of the modular curve. We searched for points on our plane model with height up to
$10^{7}$ and found a unique rational point (above $j = 0$). The Jacobian is irreducible and has trivial torsion subgroup
over $\Q$. The automorphism group of the modular curve is also trivial over $\Q$. We have not made any further progress with this curve.

The modular curve \glink{49.196.9.1} is a degree $7$ cover of
$X_{{\rm s}}^{+}(7)$ which also has genus zero. We likewise have a
singular degree $7$ plane model for it and searched for points with
height up to $10^{7}$. Again, it appears to have a unique rational
point (above $j = 0$) on it.  The curve does not have any non-trivial
automorphisms over $\Q$, and the torsion subgroup of its Jacobian is
trivial.  We have not made any further progress with this curve.

\section{\texorpdfstring{The cases with $-I \not\in H$}{The cases with -I not in H}}
\label{sec:universal-curve}
%----------------------------------------------------------------------------
\setcounter{subsection}{1}

We now suppose that $H \leq \GL_{2}(N)$ has genus 0, $-I \not\in H$ and $X_{\langle H, -I \rangle} \cong \P^{1}$. Let
$U$ be the complement of the cusps and preimages on $X_{H}$ of $j = 0$
and $j = 1728$. Then $U \subseteq \Spec \Q[t]$ is an affine scheme, and the moduli problem of $H$-level structures is rigid. Therefore,
by \cite{katzmazur}*{Scholie 4.7.0, p.~111} the moduli problem
there is representable and there is a universal curve $\mathscr{E} \to
U$ such that an elliptic curve $E/K$ with $j(E) \ne 0, 1728$ has $\im \rho_{E,N} \leq H$ if and only if $E \simeq \mathscr{E}_{t}$ for some (possibly more than one) $t \in U(K)$.  The following is \cite{RZB}*{Lemma 2.4}.
\begin{lemma}
\label{L:universal-weierstrass-embedding}
Let $f \colon \scrE \to U$ be as above and assume that $U \subset \A^1$. Then there exists a closed immersion $\scrE \hookrightarrow \P^2_U$ given by a homogeneous polynomial
\[
Y^2Z - X^3 - aXZ^2 - bZ^3
\]
 where $a,b \in \Z[t]$.
\end{lemma}

In \cite{RZB}*{Section 5.2}, the authors proved \cite{RZB}*{Lemma 5.1} that if $E/K$ is an elliptic curve for which
$\im \rho_{E,N} \leq \langle H, -I \rangle$, then there is a unique
quadratic twist $E_{d}$ of $E$ so that the image of $\rho_{E_{d},N}$
is contained in $H$ (where $\rho_{E_{d},N}$ is defined using a basis
for $E_{d}[N]$ that is compatible with the choice for $E$). See also \cite{sutherland:Computing-images-of-Galois-representations-attached-to-elliptic-curves}*{Section 5.6}. It follows that
if $E$ has $\ell$-adic image $\langle H, -I \rangle$, then the number of quadratic twists $E_{d}$ for which the $\ell$-adic image is $H$ is equal to
the number of index two subgroups of $\langle H, -I \rangle$ that
are $\GL_{2}(\Z_{\ell})$-conjugate to $H$. It is straightforward to see that
this is $[N_{\GL_{2}(\Z_{\ell})}(\langle H, -I \rangle) : N_{\GL_{2}(\Z_{\ell})}(H)]$.
If $j \ne 1728$, then $E_{d} \cong E_{d'}$ if and only if $d/d' \in (K^{\times})^{2}$.

The authors of \cite{RZB} gave a procedure
(based on a general theory of resolvent polynomials from \cite{dokchitsers:frobenius}) to represent $\mathscr{E} \to U$ as an elliptic curve over $\Q(t)$.

This technique works in general, and the same techniques are used to compute
models for the universal elliptic curves over $X_{H}$ for all arithmetically maximal $H$ for which $X_{\langle H, -I \rangle}(\Q)$ is infinite. In all such cases, $X_{\langle H, -I \rangle} \simeq \mathbb{P}^{1}$. As an example,
if $E/\Q$ is an elliptic curve with $E[5] \simeq \Z/5\Z \times \mu_{5}$ as Galois
modules, then there is a $t \in \Q$ with $t \ne 0$ so that $E \simeq E_{t}$, where
\begin{align*}
  E_{t} \colon y^{2} = x^{3} &+ (-27t^{20} - 6156t^{15} - 13338t^{10} + 6156t^{5} - 27) x\\
  & + (54t^{30} - 28188t^{25} - 540270t^{20} - 540270t^{10} + 28188t^{5} + 54).
\end{align*}
\normalsize
This corresponds to the $5$-adic image with label \texttt{5.120.0.1}.

These models allow us to efficiently compute the $\ell$-adic image of Galois for non-CM elliptic curves $E/\Q$; see Section \ref{sec:computing-images} for details.

\section{\texorpdfstring{Computing $\ell$-adic images of non-CM elliptic curves over $\Q$}{Computing l-adic images of non-CM elliptic curves over Q}}\label{sec:computing-images}
%----------------------------------------------------------------------------

In this section we present the algorithm promised in the introduction, which takes as input an elliptic curve $E/\Q$ without complex multiplication and outputs a list of labels that identify all groups of non-trivial index that arise as $\rho_{E,\ell^\infty}(G_\Q)$ for some prime $\ell$.  Serre's open image theorem guarantees that this list is finite, and we use an extension of an algorithm due to Zywina \cite{zywina:On-the-surjectivity-of-images-of-mod-ell-representations-associated-to-elliptic-curves-over-Q-arxiv}*{Algorithm 1.1} to compute a finite set $S$ that is guaranteed to contain all the primes $\ell$ we must consider.  The set $S$ consists of the primes $\ell=2,3,5,7,13$, any prime $\ell=11,17,37$ for which $j(E)$ is exceptional and listed in Table~\ref{table:exceptional-points}, and possibly a small number of additional primes. While this ``small number'' can be arbitrary large, it is typically zero unless~$E$ has bad or supersingular reduction at many small primes, and in fact is zero for all non-CM $E/\Q$ of conductor up to 500000. The function \magmacmd{PossiblyNonsurjectivePrimes} in the file \repolink{groups}{gl2.m} in \cite{RouseSZB:magma-scripts-ell-adic-images} implements this computation.

\subsection{Computing Frobenius matrices}
The next step of the algorithm is to compute a sequence of matrices $A_p \in \Z^{2\times 2}$ for primes $p\le B$ of good reduction for $E$ whose reductions modulo any integer $N$ coprime to $p$ give the action of the Frobenius endomorphism $\pi$ of $E_p:=E\bmod p$ on a basis for $E_p[N]$.  From \eqref{eq:Api}, we may take $A_p=A_\pi=A(a,b,\Delta)$, where~$a$ is the trace of Frobenius, $\Delta$ is the discriminant of the ring $R_\pi=\End(E_p)\cap \Q[\pi]$, and $b=[R_\pi:\Z[\pi]]$ when $\Z[\pi]\ne \Z$, and $b=0$ otherwise.

In \cite{centeleghe:Integral-Tate-modules-and-splitting-of-primes-in-torsion-fields-of-elliptic-curves}, Centeleghe gives an algorithm to compute $A(a,b,\Delta)$ using Hilbert class polynomials which is implemented in the \magma package \magmacmd{IntegralFrobenius}. The time and space complexity of this algorithm are both superlinear in $p$, which means that the cost of computing $A_p$ for good primes $p\le B$ grows quadratically with $B$, which makes the algorithm unsuitable for our intended application.  Here we describe a more efficient approach whose complexity is subexponential in $\log p$ under the generalized Riemann hypothesis (GRH), meaning that in practice we can compute $A_p$ for all good primes $p\le B$ in quasi-linear time.

For $j(E_p)=0,1728$ the matrix $A(a,b,\Delta)$ can be read off from Table~\ref{table:j0j1728} after computing $\Aut(E_p)$ and $\#E_p[2]$, which takes time $(\log p)^{2+o(1)}$.  If $j(E_p)\ne 0,1728$ and $E_p$ is supersingular (which can be determined in expected time $(\log p)^{3+o(1)}$ and heuristically in time $(\log p)^{2+o(1)}$ on average for $p\le B$ via \cite{sutherland:Identifying-supersingular-elliptic-curves}), then $a=0$, and $D=a^2-4p=-4p$ implies that either $b=1$ and $\Delta=-4p$ (which must occur if $p\equiv 1\bmod 4$), or $b=2$ and $\Delta=-p$, which occurs precisely when $E_p$ is on the surface of its 2-volcano (the connected components of its 2-isogeny graph over $\F_p$, see \cite{sutherland:Isogeny-volcanoes} for details).  We can determine which case applies in $(\log p)^{2+o(1)}$ time by checking whether $\#E_p[2]=4$ or not.

In all remaining cases $E_p$ is ordinary and $j(E_p)\ne 0,1728$, which occurs for 100 percent of the primes $p\le B$ asymptotically.  In this situation our first step is to compute the trace of Frobenius $a\colonequals a_p$, which can be accomplished using Schoof's algorithm \cite{schoof:Elliptic-curves-over-finite-fields-and-the-computation-of-square-roots-mod-p} in $(\log p)^{5+o(1)}$ time \cite{ShparlinksiS:On-the-distribution-of-Atkin-and-Elkies-primes-for-reductions-of-elliptic-curves-on-average}*{Corollary 11}.
Alternatively, one can apply the Schoof-Elkies-Atkin (SEA) algorithm \cites{schoof:Counting-points-on-elliptic-curves-over-finite-fields,elkies:elliptic-and-modular-curves-over-finite-fields-and-related-computational-issues}, which takes $(\log p)^{4+o(1)}$ expected time on average under the generalized Riemann hypothesis (GRH) \cite{ShparlinksiS:On-the-distribution-of-Atkin-and-Elkies-primes-for-reductions-of-elliptic-curves-on-average}*{Corollary 14}, or the average polynomial-time algorithm of \cite{HarveyS:Computing-Hasse-Witt-matrices-of-hyperelliptic-curves-in-average-polynomial-time-II}, which takes $(\log p)^{4+o(1)}$ time on average unconditionally.  For small values of $p$ it is practically faster to use Mestre's algorithm (see \cite{schoof:Counting-points-on-elliptic-curves-over-finite-fields}*{Sections 2--3}); the \magma function \magmacmd{TraceOfFrobenius} switches between Mestre's algorithm and the SEA algorithm depending on the size of $p$.

Having computed $a$, it suffices to compute $b=\bigl[\End(E_p):\Z[\pi]\bigr]$, since $\Delta = (a^2-4p)/b^2$.  The integer $b$ is computed by the algorithm in \cite{BissonSutherland:Computing-the-endomorphism-ring-of-an-ordinary-elliptic-curve-over-a-finite-field} in subexponential expected time $\exp((\log p)^{1/2+o(1)})$ under heuristic assumptions that can be reduced to GRH via \cite{bisson:Computing-endomorphism-rings-of-elliptic-curves-under-the-GRH}.
The first step of this algorithm involves an isogeny volcano computation to determine the $\ell$-adic valuation of $b$ at any small primes~$\ell$ that divide $v:=[\calO_K:\Z[\pi]]$, where $\calO_K$ is the ring of integers of $K:=\Q(\sqrt{a^2-4p})\simeq\End(E)\otimes\Q$. To compute $v$ and its prime divisors~$\ell$ we factor $D:=a^2-4p$ to determine the fundamental discriminant $D_0\colonequals\disc K$ for which $D=v^2D_0$, which can be accomplished in subexponential expected time $\exp((\log p)^{1/2+o(1)})$ via \cite{LenstraP:A-rigorous-time-bound-for-factoring-integers} (there are also heuristically faster algorithms).

The $\ell$-adic valuation of $b$ is equal to the distance from $j(E)$ to the floor of its $\ell$-volcano, which can be computed using the \textsc{FindShortestPathToFloor} algorithm in \cite{sutherland:Isogeny-volcanoes}.  In the typical case where $D$ is not divisible by $\ell^4$ this amounts to checking whether the instantiated modular polynomial $\Phi_\ell(j(E),x)\in \F_p[x]$ splits completely in $\F_p[x]$ or not.  Typically~$v$ will be quite small.  Indeed, for a random integer $-D\in [1,4p]$ the average size of the positive~$v$ for which $D=v^2D_0$ with $D_0$ squarefree is $\frac{3}{\pi^2}\log (4p)$, and when $v$ is $O(\log p)$ we are happy to compute the $\ell$-adic valuation of $b$ for every prime $\ell|v$ in this fashion.  This yields $b\in \Z$ in expected time bounded by $(\log p)^{5+o(1)}$, a bound  which includes the time to compute the modular polynomial $\Phi_\ell(x,y)$ via~\cite{brokerLS:Modular-polynomials-via-isogeny-volcanoes}.  We may gain a large constant factor speedup (up to~1728 in the best case) by using alternative modular polynomials such as those provided by the \magmacmd{AtkinModularPolynomial} function in \magma, or one of the modular polynomials considered in \cite{brokerLS:Modular-polynomials-via-isogeny-volcanoes}.
This complexity bound can be improved to $(\log p)^{4+o(1)}$ via \cite{kohel:thesis}*{Lemma 25}, but with the constant factor improvements noted above, it is typically not worth doing this when $v$ is $O(\log p)$ (in \cite{kohel:thesis} $v$ may be as large as $p^{1/6}$).

Even if $v$ is not small, as long as its prime factors $\ell$ are all small, say $\ell=O(\log p)$, we can achieve the same complexity bound for computing $b$, which is comparable to the time required to compute $a$.  In the uncommon case that $v$ is divisible by a large prime $\ell$ we instead use the volcano climbing algorithm of \cite{sutherland:Computing-Hilbert-class-polynomials-with-the-Chinese-remainder-theorem}*{\S 4.1} to obtain an isogenous elliptic curve $E_p'$ whose endomorphism ring is typically isomorphic to $\calO_K$ or a suborder of index~$\ell$ (when there are multiple large prime divisors of $v$ one needs to consider additional suborders).  These cases can be distinguished by searching for smooth relations in the ideal class group of $\calO$ and testing them in the isogeny graph of $E_p$ as described in \cite{BissonSutherland:Computing-the-endomorphism-ring-of-an-ordinary-elliptic-curve-over-a-finite-field}.  As an additional optimization, when $|D_0|$ is particularly small, one may instead simply test whether $j(E_p)$ is a root of the Hilbert class polynomial $H_{D_0}(x)$ modulo $p$ or not (or better, use the class polynomials provided by the \magmacmd{WeberClassPolynomial} function in \magma or one of the class polynomials considered in \cite{EngeSutherland:Class-invariants-by-the-CRT-method}).

The function \magmacmd{EndomorphismRingData} in \repolink{groups}{gl2.m} in \cite{RouseSZB:magma-scripts-ell-adic-images} computes the triple $(a,b,\Delta)$ using a simple \magma implementation of the algorithm described above that is optimized for the case where $v$ is small and provides reasonably good performance in the range $p \in [1,2^{30}]$, which is more than sufficient for our intended application.  This is used by the function \magmacmd{GL2FrobeniusMatrices} in \repolink{groups}{gl2.m} in \cite{RouseSZB:magma-scripts-ell-adic-images} to compute the matrices $A_p$ for primes $p$ of good reduction for a given elliptic curve $E/\Q$ up to a specified bound $B$.  An asymptotically fast implementation of the algorithm in \cite{BissonSutherland:Computing-the-endomorphism-ring-of-an-ordinary-elliptic-curve-over-a-finite-field} is available in the C library \href{https://math.mit.edu/~drew/smooth_relation_v1.3.tar}{\texttt{smoothrelation}}.

Table~\ref{table:FrobeniusMatrices} compares the times to compute the Frobenius trace $a_p$ and the Frobenius matrix~$A_p$ for all primes $p\le B$ of good reduction for a non-CM elliptic curve $E/\Q$ using the function \magmacmd{GL2FrobeniusMatrices} described above, and using the \magma package \magmacmd{IntegralFrobenius}, which implements the algorithm in \cite{centeleghe:Integral-Tate-modules-and-splitting-of-primes-in-torsion-fields-of-elliptic-curves}; both algorithms use the built-in \magma function \magmacmd{TraceOfFrobenius} to compute $a_p$.  

\begin{table}[tbh!]
\begin{tabular}{crrrrrrrrrr}
    & $2^8$ & $2^{10}$ & $2^{12}$ & $2^{14}$ & $2^{16}$ & $2^{18}$ \\\toprule
\magmacmd{TracesOfFrobenius} & 0.000 & 0.001 & 0.006 &  0.026 & 0.118 & 0.52\\% & 2.34 & 10.6 & 48.1\\ % 223, 1001, 4694
\magmacmd{FrobeniusMatrices} & 0.011 & 0.031 & 0.089 &  0.315 & 1.162 & 4.24\\% & 15.83 & 60.9 & 238\\ % 842, 3512, 15669
\magmacmd{IntegralFrobenius} & 0.036 & 0.226 & 2.82\phantom{1} & 63.4\phantom{00} & $1910\phantom{.000}$ & $59200\phantom{.00}$\\\bottomrule % & $-\phantom{0}$ & $-\phantom{.}$ & $-\phantom{.}$\\\bottomrule
\end{tabular}
\bigskip\smallskip

\begin{tabular}{crrrrrrrrrr}
    & $2^{20}$ & $2^{22}$ & $2^{24}$ & $2^{26}$ & $2^{28}$ & $2^{30}$ \\\toprule
\magmacmd{TracesOfFrobenius} & 2.34 & 10.6 & 48.1 & 223 & 1010 & 4690\\
\magmacmd{FrobeniusMatrices} & 15.8 & 60.9 & 238 & 842 & 3510 & 15700\\\midrule
relative cost & 6.75 & 5.75 & 4.95 & 3.78 & 3.48 & 3.35\\
\end{tabular}
\bigskip

\caption{Time to compute $a_p$ and $A_p$ for non-CM $E/\Q$ at good~$p\le 2^n$. Times in CPU-seconds on a 4.0GHz Intel i9-9960X core, averaged over the first 100 nonisogenous elliptic curves of conductor greater than~250000 (just the first ten curves for times over 100s and just the first curve for times over~1000s).
}\label{table:FrobeniusMatrices}
\vspace{-8pt}
\end{table}

In Table~\ref{table:FrobeniusMatrices} one sees the quasi-linear time complexity of the functions \magmacmd{TracesOfFrobenius} and \magmacmd{FrobeniusMatrices} versus the super-quadratic time complexity of \magmacmd{IntegralFrobenius}.
In fact the running times for \magmacmd{FrobeniusMatrices} grows sub-linearly up to $B=2^{26}$ and its performance relative to \magmacmd{TracesOfFrobenius} continues to improve even up to $B=2^{30}$.

\subsection{Proving surjectivity}\label{ssec:surjective}

Having computed an initial set of primes $\ell\in S$ where $\rho_{E,\ell^\infty}$ might be nonsurjective via \magmacmd{PossiblyNonsurjectivePrimes}, we use \magmacmd{FrobeniusMatrices} to compute Frobenius matrices $A_p$ for good primes $p$ up to an initial bound $B_{\rm min}$ (we used $B_{\rm min}=256$ in our computations).  Let $e_\ell=3,2$ for $\ell=2,3$, respectively, and let $e_\ell=1$ for all $\ell\ge 5$.  Then $\rho_{E,\ell^\infty}$ is surjective if and only if $\rho_{E,\ell^{e_\ell}}$ is surjective; this follows from \cite{SutherlandZ:Modular-curves-of-prime-power-level-with-infinitely-many-rational-points}*{Lemma 3.7} for $\ell=2,3$ and \cite{serre:abelianladic}*{Lemma 3, page IV-23} for $\ell \ge 5$.  By reducing each $A_p$ modulo $\ell^{e_\ell}$ we obtain a set of $\GL_2(\ell^{e_\ell})$-conjugacy classes $C_\ell$ that we can typically use to prove that $\rho_{E,\ell^\infty}$ is surjective in any case where this is true.  
For every maximal subgroup $H\le \GL_2(\ell^{e_\ell})$ there are several $\GL_2(\ell^{e_\ell})$-conjugacy classes that do not intersect~$H$.  If any of these lie in $C_\ell$, then $\rho_{E,\ell^{e_\ell}}(G_\Q)$ cannot lie in $H$, and if this holds for every maximal subgroup~$H$, then $\rho_{E,\ell^\infty}$ must be surjective.  By precomputing sets of conjugacy classes for the maximal subgroups of $\GL_2(\ell^{e_\ell})$ for $\ell\le 13$ this test can be performed very efficiently (the set~$S$ typically contains no primes $\ell>13$).  We remove from $S$ any primes~$\ell$ for which we can prove $\rho_{E,\ell^\infty}$ is surjective using the sets of conjugacy classes $C_\ell$.

\subsection{\texorpdfstring{Determining nonsurjective $\ell$-adic images}{Determining nonsurjective l-adic images}}

Having reduced $S$ to a set of primes~$\ell$ for which we expect $\rho_{E,\ell^\infty}$ to be nonsurjective, we first check whether $j(E)$ is one of the 20 exceptional $j$-invariants listed in Table~\ref{table:exceptional-points}, and if so, whether $E$ is isomorphic to any of the corresponding exceptional models listed in Table~\ref{table:exceptional-points}.  Whenever this applies we can read off the image of $\rho_{E,\ell^\infty}$ from the appropriate row of the table.

Otherwise we expect that $j(E)$ is not exceptional (we conjecture that this is always true), in which case it corresponds to a point on a modular curve $X_H$ for which $j(X_H(\Q))$ is infinite (possibly more than one such $H$).  Except for the four 2-adic groups with labels
\[
\gtarget{8.96.1.9},\quad \gtarget{8.96.1.18},\quad \gtarget{8.96.1.75},\quad \gtarget{8.96.1.102},
\]
which were shown to have no non-cuspidal rational points in \cite{RZB}, there is no open subgroup $H\le \GL_2(\Zhat)$ of prime power level for which $j(X_H(\Q))$ is infinite that has a maximal subgroup $H'\le H$ for which $j(X_{H'}(\Q))$ is finite that contains the same set of $\GL_2(\ell^e)$-conjugacy classes as $H$, where $\ell^e$ is the level of $H'$.  It follows that if we compute sufficiently many Frobenius matrices $A_p$, for each $\ell\in S$ we should be able to prove that $j(E)$ is not exceptional and obtain a small list of groups $H$ of $\ell$-power level for which $j(E)$ is one of infinitely many elements of $j(X_H(\Q))$.  Typically the initial set of Frobenius matrices~$A_p$ for good $p\le B=B_{\rm min}$ is already sufficient, but if not we successively double~$B$ and compute more Frobenius matrices $A_p$ until this condition holds or we reach a preset maximum~$B_{\rm max}$ (we used $B_{\rm max}=2^{20}$ in our computations), at which point we are convinced that we must have found a new exceptional point and compute $\rho_{E,\ell^\infty}(G_\Q)$ directly by explicitly computing the torsion fields $\Q(E[\ell^n])$ for sufficiently many powers of $\ell$ (we must eventually have $[\Q(E[\ell^{n+1}]):\Q(E[\ell^n])]=\ell^4$ with $\ell>2$ or $n>1$, at which point we can stop, by \cite{SutherlandZ:Modular-curves-of-prime-power-level-with-infinitely-many-rational-points}*{Lemma 3.7}).  We conjecture that this will never happen, and it has not happened for any non-CM elliptic curve $E/\Q$ whose $\ell$-adic Galois images we have computed; as described in Section~\ref{ssec:database-computations}, this includes more than 380 million elliptic curves.

Assuming we do not reach $B=B_{\rm max}$, for each remaining prime $\ell\in S$ we have a small list of candidate groups $H$ of $\ell$-power level for which $j(X_H(\Q))$ is infinite that we know includes the image of $\rho_{E,\ell^\infty}$.  These groups $H$ can have genus at most 1, and we have precomputed explicit models ($\P^1$ or a positive rank elliptic curve over $\Q$) and maps $j_H\colon X_H\to X(1)$ for each of them, as described in Sections \ref{sec:comp-equat-x_h} and~\ref{sec:universal-curve}.  There is necessarily a unique $H$ of maximal index containing $-I$ for which $j(E)\in j_H(X_H(\Q))$, which we can determine by testing whether $j(E)$ lies in the image of $j_H\colon X_H\to X(1)$.  If $H$ has genus 0 then $j_H\in \Q(t)$ and $j(E)$ lies in $j_H(X_H(\Q))$ if and only if the numerator of $j_H(t)-j(E)\in \Q(t)$ has a rational root or $j_H(1/t)(0) = j(E)$.  When $H$ has genus 1 we have a plane Weierstrass model $X_H\colon y^2+h(x)y=f(x)$ and a rational map $j_H(x,y):X_H\to X(1)$, and we can use the equations $j_H(x,y)=j(E)$ and $y^2+h(x)=f(x)$ to solve for $y$ as a function $g(x)$ of $x$; for each root $x_0$ of $j_H(x,g(x))=j(E)$ we then check whether $g(x_0)^2+h(x_0)=f(x_0)$ holds.  These computations are implemented in the functions \magmacmd{OnGenusZeroCurve}, \magmacmd{OnGenusOneCurve} in the file \repolink{groups}{gl2data.m} in \cite{RouseSZB:magma-scripts-ell-adic-images}.

It may happen that the unique $H$ of maximal index containing $-I$ with $j(E)\in j_H(X_H(\Q))$ actually has index one, in which case we have proved that the Galois representation $\rho_{E,\ell^\infty}$ is surjective; this is rare but may occasionally happen when the initial set of Frobenius matrices $A_p$ for $p\le B_{\rm min}$ are not sufficient to rule out every maximal subgroup of $\GL_2(\ell^{e_\ell})$ as described in Section~\ref{ssec:surjective}.  When we encounter this case we remove $\ell$ from $S$ and continue.

Having determined $H=\langle \rho_{E,\ell^\infty}(G_\Q),-I\rangle$ it remains only to check the index-2 subgroups $H'\le H$ for which $\langle H',-I\rangle=H$.  For each $H'$ we need to test whether $E$ is isomorphic to the universal elliptic curve $E_t$ for $H'$ instantiated at a value of $t$ for which $j(E_t)=j(E)$ (there are no such $H'$ for any $H$ of genus 1 so we may assume the universal elliptic curve has a single parameter).
This amounts to finding the rational roots of an integer polynomial and checking isomorphism of elliptic curves over $\Q$, a computation that is implemented in the function \magmacmd{OnGenusZeroCurveTwist} in \repolink{groups}{gl2data.m} in \cite{RouseSZB:magma-scripts-ell-adic-images}.
There will necessarily be at most one subgroup~$H'$ for which this holds.  If there is one then we have proved that $\rho_{E,\ell^\infty}(G_\Q)=H'$, and if not we have proved that $\rho_{E,\ell^\infty}(G_\Q)=H$.

The function \magmacmd{GL2EllAdicImages} in \repolink{groups}{gl2data.m} in \cite{RouseSZB:magma-scripts-ell-adic-images} implements the algorithm we have described in this section.  It takes as input a non-CM elliptic curve $E/\Q$ and a collection of precomputed data for the relevant subgroups $H$, models $X_H$, and maps $j_H$ that can be loaded using the function \magmacmd{GL2Load} in \repolink{groups}{gl2data.m} in \cite{RouseSZB:magma-scripts-ell-adic-images}.  It outputs a list of labels that uniquely identify (up to conjugacy) all of the groups $H$ corresponding to the image of a nonsurjective $\ell$-adic Galois representation $\rho_{E,\ell^\infty}\colon G_\Q\to \GL_2(\Z_\ell)$ attached to~$E$.

\subsection{Database computations}\label{ssec:database-computations}

We used the function \magmacmd{GL2EllAdicImages} to compute $\ell$-adic images of Galois for the non-CM elliptic curves over $\Q$ in three large databases:
\begin{itemize}
\setlength{\itemsep}{4pt}
\item The $L$-functions and Modular Forms Database (LMFDB) \cite{lmfdb}, which includes all elliptic curves $E/\Q$ of conductor up to 500000, of which approximately 3 million do not have complex multiplication (the LMFDB contains additional $E/\Q$ whose $\ell$-adic images we also computed, but the results are not included in Table~\ref{table:dbstats}).

\item The Stein--Watkins Database \cite{SteinWatkins:A-database-of-elliptic-curves-first-report}, which includes elliptic curves $E/\Q$ that are isogenous to an elliptic curve whose discriminant $|\Delta|\le 10^{12}$ is minimal within its quadratic-twist family and satisfies $|c_4|\le 1.44\times 10^{12}$ with conductor $N\le 10^8$. This includes nearly 137 million elliptic curves of conductor up to $10^8$ without complex multiplication (this database includes additional $E/\Q$ of larger prime conductor whose $\ell$-adic images we also computed, but the results are not included in Table~\ref{table:dbstats}).

\item The Balakrishnan--Ho--Kaplan--Spicer--Stein--Weigandt Database \cite{balakrishnanHKSSW:Databases-of-elliptic-curves-ordered-by-height-and-distributions-of-Selmer-groups-and-ranks}, which includes all elliptic curves $E/\Q$ of naive height up to $2.7\times 10^{10}$, including nearly 239 million elliptic curves without complex multiplication.
\end{itemize}

The total time to compute all $\ell$-adic Galois images of every non-CM elliptic curve over~$\Q$ in the LMFDB was less than 45 hours, an average of about 42 milliseconds per curve.  Most of this time is spent on elliptic curves whose $\ell$-adic Galois representation is nonsurjective for one or more primes~$\ell$.  The method given in Section~\ref{ssec:surjective} for proving nonsurjectivity is very efficient, and for elliptic curves whose $\ell$-adic Galois representations are surjective at every prime it typically takes less than a millisecond to prove this.  As can be seen in Table~\ref{table:dbstats}, such curves make up almost all of the Balakrishnan--Ho--Kaplan--Spicer--Stein--Weigandt database, and it did not take much longer to compute the $\ell$-adic image data for this dataset than it did for the LMFDB.  The Stein--Watkins database took substantially longer, nearly 2 core-years (about five days of parallel computation on 128 cores).

\newcommand{\lmfdbcnt}[2]{\href{https://www.lmfdb.org/EllipticCurve/Q/?conductor=1-500000\&cm=noCM\&nonmax_primes=#1}{#2}}

\begin{table}[bth!]
\small
\setlength{\tabcolsep}{4.5pt}
\begin{tabular}{lrrrrrrrrrr}
& \multicolumn{9}{c}{nonsurjective primes}\\
\cmidrule{2-10}
database & 2 & 3 & 5 & 7 & 11 & 13 & 17 & 37 & none & total \\\toprule
LMFDB & \lmfdbcnt{2}{1357468} & \lmfdbcnt{3}{266426} & \lmfdbcnt{5}{20238} & \lmfdbcnt{7}{3984} & \lmfdbcnt{11}{156} & \lmfdbcnt{13}{536} & \lmfdbcnt{17}{40} & \lmfdbcnt{37}{80} & \href{https://www.lmfdb.org/EllipticCurve/Q/?conductor=1-500000\&cm=noCM\&nonmax_quantifier=exclude\&nonmax_primes=2,3,5,7,11,13,17,37}{1467623} & \href{https://www.lmfdb.org/EllipticCurve/Q/?conductor=1-500000\&cm=noCM}{3058813}\\
SW & 35585316 & 3671438 & 181222 & 43966 & 2048 & 7444 & 368 & 1024 & 97776484 & 136789294\\
BHKSSW & 242540 & 8750 & 400 & 108 & 0 & 2 & 44 & 2 & 238447364 & 238698578\\\bottomrule
\end{tabular}
\bigskip\smallskip

\begin{tabular}{lrrrrrrrrrr}
& \multicolumn{9}{c}{nonsurjective pairs and triples of primes}\\
\cmidrule{2-10}
database & $\{2,3\}$ & $\{2,5\}$ & $\{2,7\}$ & $\{2,11\}$ & $\{2,13\}$ & $\{3,5\}$ & $\{3,7\}$ & $\{2,3,5\}$ & $\{2,3,7\}$ \\\toprule
LMFDB & \lmfdbcnt{2,3}{53168} & \lmfdbcnt{2,5}{3354} & \lmfdbcnt{2,7}{800} & \lmfdbcnt{2,11}{148} & \lmfdbcnt{2,13}{44} & \lmfdbcnt{3,5}{788} & \lmfdbcnt{3,7}{240} & \lmfdbcnt{2,3,5}{564} & \lmfdbcnt{2,3,7}{240}\\
SW & 424566 & 38790 & 11044 & 2048 & 640 & 10832 & 3272 & 7904 & 3272\\
BHKSSW & 382 & 154 & 62 & 2 & 22 & 42 & 16 & 32 & 16\\\bottomrule
\end{tabular}
\bigskip\smallskip

\caption{Summary of $\ell$-adic image data for non-CM elliptic curves $E/\Q$ in the LMFDB, Stein--Watkins (SW), and Balakrishnan--Ho--Kaplan--Spicer--Stein--Weigandt (BHKSSW) databases.  Nonsurjective counts are inclusive and may include curves that are also nonsurjective at another prime.}\label{table:dbstats}
\vspace{-16pt}
\end{table}

%----------------------------------------------------------------------------
\section{\texorpdfstring{Computing $\ell$-adic images of CM elliptic curves over $\Q$}{Computing l-adic images of CM elliptic curves over Q}}\label{sec:computing-cm-images}
%----------------------------------------------------------------------------

Let $E$ be an elliptic curve over the number field $F=\Q(j(E))$ with \defi{potential CM by~$\calO$}, by which we mean that $\End(E_{\overline{F}})\simeq \calO$ is an order in an imaginary quadratic field $K$.  The possible images of $\rho_{E,\ell^\infty}$ are determined by Lozano-Robledo in \cite{lozanoRobledo:Galois-representations-attached-to-elliptic-curves-with-complex-multiplication}.  In this section we give an efficient algorithm to compute these images.  We are primarily interested in the case $F=\Q$ where $\calO$ is one of the 13 imaginary quadratic orders of class number one, but we shall treat the more general case of a CM elliptic curve defined over its minimal field of definition, as this requires little additional effort.

\subsection{Adelic Cartan groups}

For each imaginary quadratic order $\calO$ we define
\[
C_\calO\coloneqq \varprojlim_N \left(\calO/N\calO\right)^\times,
\]
where $N$ varies over positive integers ordered by divisibility.  We embed $C_{\calO}$ in $\GL_2(\Zhat)$ as follows.
Let $f\coloneqq [\calO_K:\calO]$, let $D\coloneqq \disc \calO= f^2\disc \calO_K$, and let $\phi\coloneqq f$ if $D$ is odd, with $\phi\coloneqq 0$ otherwise.  Let
\[
\omega\coloneqq \frac{\phi+\sqrt{D}}{2}\qquad\text{and}\qquad \delta \coloneqq \frac{D-\phi^2}{4},
\]
so that $\calO=\Z[1,\omega]$ with $\omega^2-\phi\omega-\delta=0$.  We now define the \defi{Cartan subgroup} $\calC_\calO\coloneqq\calC_\calO(\Zhat)$ as the inverse limit of the groups
\[
\calC_\calO(N)\coloneqq \left\{\begin{bmatrix}a+b\phi & b\\ \delta b & a\end{bmatrix}:a,b\in \Z(N), a^2+ab\phi-\delta b^2\in \Z(N)^\times\right\}\subseteq\GL_2(N).
\]
It is a closed subgroup of $\GL_2(\Zhat)$ that is isomorphic to $C_\calO$ via the map induced by
\begin{equation}\label{eq:oembed}
a+b\omega \mapsto \begin{bmatrix}a+b\phi & b\\ \delta b & a\end{bmatrix}.
\end{equation}

It follows from the theory of complex multiplication that if the CM field $K\coloneqq\calO\otimes \Q$ lies in the field of definition of an elliptic curve $E$ with CM by $\calO$ then the image of $\rho_E$ lies in $\calC_\calO\le \GL_2(\Zhat)$ (recall that we view subgroups of $\GL_2(\Zhat)$ as defined only up to conjugation); see \cite{serre-tate:good-reduction-of-abelian-varieties}*{Thm.~5}, for example. This is not true when $E$ is defined over its minimal field of definition $F=\Q(j(E))$, or over any field that does not contain $K$. In this case the image of $\rho_E$ lies in the group $\calN_\calO\le \GL_2(\Zhat)$, which we define as the inverse limit of the groups
\[
\calN_\calO(N)\coloneqq \left\langle \calC_\calO(N),\smallmat{-1}{0}{\phi}{1}\right\rangle.
\]
The group $\calN_\calO$ is nonabelian and contains the abelian group $\calC_\calO$ with index 2.  We shall use $\calC_\calO(\Z_\ell)$ and $\calN_\calO(\Z_\ell)$ to denote the images of these groups in $\GL_2(\Z_\ell)$ under the projection from $\GL_2(\Zhat)\simeq \prod_\ell \GL_2(\Z_\ell)$.

\begin{remark}
Our definition of $\calC_\calO$ above differs slightly from that in \cite{lozanoRobledo:Galois-representations-attached-to-elliptic-curves-with-complex-multiplication} for odd $N$ when $D\equiv 1\bmod 4$, but the groups are conjugate.  This difference simplifies our presentation here and has no impact on our results or those in \cite{lozanoRobledo:Galois-representations-attached-to-elliptic-curves-with-complex-multiplication} on which we rely.
\end{remark}

\begin{remark}
The group $\calN_\calO$ is sometimes referred to as the ``normalizer" of the Cartan subgroup $\calC_\calO$ in the literature.  The group $\calN_\calO(\Z_\ell)$ is the normalizer of $\calC_\calO(\Z_\ell)$ in $\GL_2(\Z_\ell)$, by \cite{lozanoRobledo:Galois-representations-attached-to-elliptic-curves-with-complex-multiplication}*{Prop.~5.6(2)}, but $\calN_\calO$ is not the normalizer of $\calC_\calO$ in $\GL_2(\Zhat)$; indeed, $\calC_\calO$ has infinite index in its normalizer.
The group $\calN_\calO(N)$ is properly contained in the normalizer of $\calC_\calO(N)$ in $\GL_2(N)$ for most values of $N$.  This can occur even when $N$ is a prime (if it divides $\disc \calO$), and for infinitely many powers of the same prime; see \cite{lozanoRobledo:Galois-representations-attached-to-elliptic-curves-with-complex-multiplication}*{Prop.~5.6} for examples.
\end{remark}

\begin{proposition}\label{prop:CMindex}
Let $E$ be a elliptic curve over $F=\Q(j(E))$ with potential CM by $\calO$.  There is a choice of $\iota_E\colon\Aut(E_{\rm tor}({\overline{F}}))\overset{\sim}{\to} \GL_2(\Zhat)$ for which the induced Galois representation $\rho_E\colon G_F\to \GL_2(\Zhat)$ satisfies the following:
\begin{enumerate}[{\rm(i)}]
\item $\rho_E(G_F)$ is an open subgroup of $\calN_\calO$,
\item $[\calN_\calO:\rho_E(G_F)]$ divides $\#\calO^\times$,
\item $[\rho_E(G_F):\rho_E(G_F)\cap \calC_\calO]=2$.
\end{enumerate}
\end{proposition}
\begin{proof}
Parts (i) and (ii) follow from \cite{lozanoRobledo:Galois-representations-attached-to-elliptic-curves-with-complex-multiplication}*{Theorem 1.1}.  For (iii) we note  that, as in \cite{zywina:On-the-possible-images-of-the-mod-ell-representations-associated-to-elliptic-curves-over-Q-arxiv}*{Lemma 7.2}, if we pick $\beta\in \calO-\Z$ with $\beta^2\in \Z$ squarefree, then for any $\sigma\in G_F-G_{FK}$ we have $\sigma(\beta P)=\sigma(\beta)\sigma(P)=-\beta\sigma(P)$ for all $P\in E[N]$ and all $N\ge 1$.  It follows that $\rho_E(\sigma)\pi(\beta)=-\pi(\beta)\rho_E(\sigma)$, where $\pi$ is the natural map $\calO\to C_\calO\overset{\sim}{\rightarrow} \calC_\calO$, thus $\rho_E(G_F)$ is nonabelian.  It therefore contains the abelian group $\rho_E(G_G)\cap \calC_\calO$ with index at least 2, and (i) implies that the index is at most 2, since $[\calN_\calO:\calC_\calO]=2$.
\end{proof}

\begin{proposition}\label{prop:CMlevel}
Let $E$ be an elliptic curve over $F=\Q(j(E))$ with potential CM by $\calO$, let~$\ell$ be a prime, and let
\[
e = \begin{cases}
4&\text{if }\ell=2,\\
3&\text{if }\ell=3,\\
1&\text{otherwise}.
\end{cases}
\]
Then $\rho_{E,\ell^\infty}(G_F)$ is the inverse image of $\rho_{E,\ell^e}(G_\ell)$ under the projection $\calN_\calO(\Z_\ell)\to \calN_\calO(\ell^e)$.
\end{proposition}
\begin{proof}
This follows from Theorems 1.2, 1.4, 1.5, 1.6, 1.7 in \cite{lozanoRobledo:Galois-representations-attached-to-elliptic-curves-with-complex-multiplication} once one notes that for every imaginary quadratic order $\calO$ (including $\Z[\zeta_3],\Z[i]$), the possibilities for $\rho_{E,\ell^\infty}(G_F)\le \calN_\calO(\Z_\ell)$ enumerated in these theorems are non-conjugate modulo $\ell^e$.
\end{proof}

\begin{remark}
The values of $e$ in Proposition \ref{prop:CMlevel} are best possible, even if one restricts to $F=\Q$.  If one excludes $\calO=\Z[\zeta_3]$ one can take $e=1$ for $\ell=3$, but $e=4$ is necessary for $\ell=2$ and infinitely many imaginary quadratic orders $\calO$.
\end{remark}

Proposition~\ref{prop:CMlevel} implies that to identify $\rho_{E,\ell^\infty}(G_F)$ it suffices to specify the order $\calO$ (via its discriminant) and the subgroup $\rho_{E,\ell^e}(G_F)\le \calN_\calO(\ell^e)\le \GL_2(\ell^e)$ (via its label defined in Section~\ref{ssec:subgroup-labels}).  Note that in the CM setting these labels identify the inverse image of a subgroup of $\calN_\calO(\ell^e)\le \GL_2(\ell^e)$ in $\calN_\calO(\Z_\ell)$ rather than in $\GL_2(\Z_\ell)$.

\begin{remark}
For each prime $\ell$ there are infinitely many nonconjugate $\calN_\calO(\Z_\ell)\le \GL_2(\Z_\ell)$ as $\calO$ varies, even though there are only finitely many possibilities for $\calN_\calO(\ell^e)$.  It is the pair $(\calO,\rho_{E,\ell^e}(G_F))$ that uniquely determines $\rho_{E,\ell^\infty}(G_F)\le \calN_\calO(\Z_\ell)\le \GL_2(\Z_\ell)$.  We should note that $\calN_\calO(Z_\ell)$ and $\calN_\calO'(\Z_\ell)$ may be conjugate even when $\calO\ne \calO'$; this happens when $\disc\calO/\disc\calO'=u^2$ with $u\in \Z_\ell^\times$ (to see this, conjugate \eqref{eq:oembed} by $\smallmat{u}{0}{0}{u^{-1}}$).
\end{remark}

For any number field $F$, we say that an elliptic curve $E/F$ with potential CM by $\calO$ has \defi{maximal $\ell$-adic image} if $\rho_{E,\ell^\infty}(G_F)$ contains $\SL_2(\Z_\ell)\cap \calC_\calO(\Z_\ell)$.

\begin{proposition}\label{prop:CMmaximal}
Let $E$ be an elliptic curve over $F=\Q(j(E))$ with potential CM by $\calO\ne \Z[\zeta_3]$.  Then $E$ has maximal $\ell$-adic image for all primes $\ell\nmid 2\disc\calO$.
\end{proposition}
\begin{proof}
This is \cite{lozanoRobledo:Galois-representations-attached-to-elliptic-curves-with-complex-multiplication}*{Theorem 1.2(4)}.
\end{proof}

The following lemma allows us to immediately compute the label of $\rho_{E,\ell}(G_F)$ when $E$ has maximal $\ell$-adic image at a prime $\ell>3$ not dividing $\disc\calO$.

\begin{lemma}\label{lem:CMmaxlabel}
Let $E$ be an elliptic curve over $F=\Q(j(E))$ with potential CM by $\calO$ and maximal $\ell$-adic image at a prime $\ell>3$ not dividing $D:=\disc\calO$. Then $H\coloneqq \rho_{E,\ell}(G_F)$ has index $i=2\ell(\ell+(\frac{D}{\ell}))$ in $\GL_2(\ell)$ and genus
\[
g = \begin{cases}
(11+(\ell-8)\ell-4(\frac{-3}{\ell}))/24&\text{if $(\frac{D}{\ell})=+1$},\\
(23+(\ell-10)\ell+6(\frac{-1}{\ell})+4(\frac{-3}{\ell}))/24&\text{if $(\frac{D}{\ell})=-1$}.
\end{cases}
\]
The subgroup $H$ has label $\ell.i.g.1$ under the labeling system defined in \S\ref{ssec:subgroup-labels}.
\end{lemma}
\begin{proof}
When $E$ has maximal $\ell$-adic image $\rho_{E,\ell}(G_F)$ contains an index-2 Cartan subgroup of $\GL_2(\ell)$ that is split if $(\frac{D}{\ell})=+1$ and nonsplit if $(\frac{D}{\ell})=-1$.  We have $(\calO/\ell\calO)^\times\simeq \F_\ell^\times \times \F_\ell^\times$ of order $(\ell-1)^2$ in the split case and $(\calO/\ell\calO)^\times\simeq \F_{\ell^2}^\times$ of order $\ell^2-1$ in the nonsplit case, which implies the index formula, since $\#\GL_2(\ell)=\ell(\ell-1)(\ell^2-1)$.  The genus formulas for these groups appear on page 117 of \cite{mazur:rationalPointsOnModular}.  For $\ell > 31$ there is a unique conjugacy class of subgroups of index $i=2(\ell-1)(\ell\pm 1)$ in $\GL_2(\ell)$ (this follows form Dickson's classification \cite{Dickson:gl2}), and this implies that the last number in the label of~$H$ must be $1$.  A direct calculation shows that this is also true for $3\le \ell\le 31$.
\end{proof}

\subsection{CM twists}

For each imaginary quadratic order $\calO$ we define
\[
z_\calO\coloneqq \begin{cases}
\smallmat{1}{1}{-1}{0} &\text{if }\disc \calO = -3,\\
\smallmat{0}{-1}{1}{0} &\text{if }\disc \calO = -4,\\
\smallmat{-1}{0}{0}{-1}&\text{otherwise,}
\end{cases}
\]
so that  $\langle z_\calO\rangle$ is the image of $\calO^\times\le C_\calO$ under the isomorphism defined in \eqref{eq:oembed}.
We note that $\Aut(E_{\overline{F}})\simeq \calO^\times\simeq \langle z_\calO\rangle $ via the restriction of $\iota_E\colon \Aut(E_{\rm tor}(\overline{F}))\simeq \GL_2(\Zhat)$ in Proposition~\ref{prop:CMindex}.  We also note that $\langle z_\calO\rangle$ is a normal subgroup of $\calN_\calO$.

Recall that a \defi{twist} of an elliptic curve $E/F$ is an elliptic curve $E'/F$ for which $E'_{\overline{F}}\simeq E_{\overline{F}}$, equivalently, for which $j(E')=j(E)$.  

\begin{proposition}\label{prop:CMtwist}
Let $E$ be a elliptic curve over $F=\Q(j(E))$ with potential CM by $\calO$.
There is a choice of $\iota_E\colon\Aut(E_{\rm tor}({\overline{F}}))\overset{\sim}{\to} \GL_2(\Zhat)$ for which the induced Galois representation $\rho_E\colon G_F\to \GL_2(\Zhat)$
has image $H=\rho_E(G_F)$ satisfying $\calN_\calO = \langle H,z_\calO\rangle$.  Moreover, every $H\le \calN_\calO$ with $\langle H,z_\calO\rangle=\calN_\calO$ occurs as $\rho_{E'}(G_F)$ for some twist $E'$ of $E$ and choice of $\iota_E'$.

In particular, for every positive integer $N$, the images $\rho_{E',N}(G_F)\le \calN_\calO(N)$ realized by twists $E'$ of $E$ are the subgroups $H\le \calN_\calO(N)$ for which $\langle H,z_\calO\rangle = \calN_\calO(N)$.
\end{proposition}
\begin{proof}
It follows from \cite{lozanoRobledo:Galois-representations-attached-to-elliptic-curves-with-complex-multiplication}*{Corollary 4.3} that if we choose $\iota_E$ as in Proposition~\ref{prop:CMindex}, then $\langle H,z_\calO\rangle = \calN_\calO$, which proves the first statement in the proposition.

By \cite{lozanoRobledo:Galois-representations-attached-to-elliptic-curves-with-complex-multiplication}*{Theorem~6.5} we may assume $\rho_{E}(G_F)=\calN_\calO$ (replace $E$ by an appropriate twist and choose $\iota_E\colon \Aut(E_{\rm tor}(\overline F))\to \GL_2(\Zhat)$ as above) with $\iota_E(\Aut(E_{\overline F}))=\langle z_\calO\rangle \triangleleft \calN_\calO$.  Let~$H$ be any subgroup of $\calN_\calO$ for which $\langle H,z_\calO\rangle =\calN_\calO$.
Let $A=\langle a\rangle$ be the cyclic subgroup of $\langle z_\calO\rangle$ of order $m\coloneqq [\calN_\calO:H]$ so that $\calN_\calO=A\rtimes H$ (note that $A$ is a subgroup of a cyclic normal subgroup, hence normal), and define the character $\chi\colon G_F\twoheadrightarrow \mu_m=\calO^\times[m]$ by $\sigma\mapsto\zeta_m^{e_\sigma}$, where $e_\sigma\in [0,\ldots,m-1]$ is uniquely determined by $\rho_E(\sigma)\in h_{\sigma}a^{e_\sigma}$, with $h_\sigma\in H$.  Now consider the twist $E'$ of $E$ by $\chi^{-1}$; it is isomorphic to $E$ over the fixed field of the kernel of $\chi$, which is a cyclic extension of $F$ of degree  $m$.  For all $\sigma\in G_F$ we have $\rho_{E'}(\sigma)=\rho_{E}(\sigma)\iota_E(\chi(\sigma))=h_\sigma a^{e_\sigma}a^{-e_\sigma}=h_\sigma$, and it follows that $\rho_{E'}(G_F)=H$.
\end{proof}

In view of Proposition~\ref{prop:CMtwist} we shall call the subgroups $H\le \calN_\calO$ for which $\langle H,z_\calO\rangle = \calN_\calO$ \defi{twists} of $\calN_\calO$, and similarly for $\calN_\calO(\Z_\ell)$ and $\calN_\calO(N)$.  The proposition states that the images of Galois that arise among the twists of an elliptic curve $E/\Q(j(E))$ with potential CM by~$\calO$ are precisely the twists of $\calN_\calO$.  It will also be convenient to distinguish twists $H$ of $\calN_\calO$ for which $\langle H,-I\rangle = \calN_\calO$ as \defi{quadratic twists}, and similarly for $\calN_\calO(N)$.  If $\calO \ne \Z[\zeta_3],\Z[i]$ then every twist of $\calN_\calO$ is a quadratic twist.

\begin{proposition}\label{prop:CMramified}
Let $E$ be an elliptic curve over $F=\Q(j(E))$ with potential CM by $\calO$, and let $\ell>3$ be a prime divisor of $\disc \calO$.
All but two $F$-isomorphism classes of quadratic twists of~$E$ have maximal $\ell$-adic image $\calN_\calO(\ell^\infty)$.
The two that do not have mod-$\ell$ image equal to one of the two index-$2$ subgroups $H_1,H_2\le \calN_\calO(\ell)$ that do not contain $-1$.
These three possibilities are distinguished by their orbit signatures as follows:
\begin{itemize}
\item $\calN_\calO(\ell)$ has two non-trivial $\ell$-torsion orbits, one of size $\ell-1$ and one of size $\ell(\ell-1)$,
\item $H_1$ has three non-trivial $\ell$-torsion orbits, one of size $\ell-1$ and two of size $\ell(\ell-1)/2$,
\item $H_2$ has three non-trivial $\ell$-torsion orbits, two of size $(\ell-1)/2$ and one of size $\ell(\ell-1)$.
\end{itemize}
The cases $H_1$ and $H_2$ each arise for a quadratic twist of $E$ that is unique up to $F$-isomorphism.
\end{proposition}
\begin{proof}
This follows from \cite{lozanoRobledo:Galois-representations-attached-to-elliptic-curves-with-complex-multiplication}*{Theorem 1.5}, \cite{sutherland:Computing-images-of-Galois-representations-attached-to-elliptic-curves}*{Corollary~5.25}, and an analysis of the $\ell$-torsion orbits of $\calN_\calO(\ell),H_1,H_2$, which are explicitly described in \cite{lozanoRobledo:Galois-representations-attached-to-elliptic-curves-with-complex-multiplication}*{Theorem 1.5}.
\end{proof}

In practice we can quickly determine when we are in the typical case $\rho_{E,\ell}(G_F)=\calN_\calO(\ell)$ of Proposition~\ref{prop:CMramified} by sampling Frobenius elements at small primes of good reduction, since $\calN_\calO(\ell)$ contains many elements that are not $\GL_2(\ell)$-conjugate to any element in $H_i$ for $i=1,2$.  But if we do not find such an element we cannot rule out the possibility that $\rho_{E,\ell}(G_F)=\calN_\calO(\ell)$, nor can we hope to distinguish $H_1$ from $H_2$ by sampling Frobenius elements, since these groups are \emph{Gassmann equivalent} subgroups of $\GL_2(\ell)$, they contain the same number of elements in each $\GL_2(\ell)$-conjugacy class.

However, we can use the $\ell$-division polynomial of $E$ to compute the orbit-signature of $\rho_{E,\ell}(G_F)$ and thereby uniquely determine which of the three cases in Proposition~\ref{prop:CMramified} applies. Let $y^2=f(x)$ be an integral model for $E$ with $f\in \calO_F[x]$ a cubic polynomial, and let $h\in \calO_F[x]$ be the $\ell$-division polynomial of degree $(\ell^2-1)/2$, which can be efficiently computed using well known recursive formulas \cite{McKee:computing-division-polynomials}; the roots of $h$ are the distinct $x$-coordinates of the non-trivial points in $E[\ell](\overline F)$.  Each irreducible factor $g$ of $h$ in $F[x]$ corresponds to a Galois orbit of the set $\{x(P):P\in E[\ell](\overline F) \text{ of order }\ell\}$, each of which corresponds to either two Galois orbits of $E[\ell](\overline F)$ of the same size or a single Galois orbit of $E[\ell](\overline F)$ of twice the size.  These two possibilities can be distinguished by checking whether $f$ is a square in the field $F[x]/(g(x))$ or not.  This computation is implemented by the function \magmacmd{TorsionOrbits} in the file \repolink{groups}{gl2.m} in \cite{RouseSZB:magma-scripts-ell-adic-images}, which can be used to compute the Galois orbit signature of $E[N](\overline F)$ for any positive integer $N$.

The algorithm sketched above runs in polynomial-time (we can factor polynomials over number fields in polynomial time \cites{lenstra:factoring-polynomials,landau:factoring-polynomials} and then compute square-roots by Hensel lifting), but it is time consuming for large values of $\ell$, including $\ell=163$, which is a necessary case to consider for $F=\Q$.  A more expedient approach is to determine once and for all the finite set of $F$-isomorphism classes of elliptic curves $E/F$ with potential CM by an imaginary quadratic order $\calO$ that have nonmaximal $\ell$-adic image a prime $\ell>3$ that divides $\disc \calO$.  For $F=\Q$ this is done by Zywina in \cite{zywina:On-the-possible-images-of-the-mod-ell-representations-associated-to-elliptic-curves-over-Q-arxiv}*{Proposition~1.14}.  Table~\ref{table:CMramified} lists the LMFDB labels of these elliptic curves and the label of their mod-$\ell$ image~$H_i$.  With this table in hand, determining which case in Proposition \ref{prop:CMramified} applies is reduced to an isomorphism test.

\begin{table}[b]
\begin{tabular}{rrrrrrr}
$\ell$ & $D$ & $H_1$ & $E_1$ & $H_2$ & $E_2$\\\toprule
  7 &   -7 & \glabel{7.48.0.3} & \eclabel{49.a2} & \glabel{7.48.0.6} & \eclabel{49.a4}\\
    &  -28 & \glabel{7.48.0.3} & \eclabel{49.a1} & \glabel{7.48.0.6} & \eclabel{49.a3}\\
 11 &  -11 & \glabel{11.120.1.5} & \eclabel{121.b2} & \glabel{11.120.1.10} & \eclabel{121.b1}\\
 19 &  -19 & \glabel{19.360.7.9} & \eclabel{361.a2} & \glabel{19.360.7.18} & \eclabel{361.a1}\\
 43 &  -43 & \glabel{43.1848.57.21} & \eclabel{1849.b2} & \glabel{43.1848.57.42} & \eclabel{1849.b1}\\
 67 &  -67 & \glabel{67.4488.155.31} & \eclabel{4489.b2} & \glabel{67.4488.155.64} & \eclabel{4489.b1}\\
163 & -163 & \glabel{163.26568.1027.57} & \eclabel{26569.a2} & \glabel{163.26568.1027.138} & \eclabel{26569.a1}\\
\bottomrule
\end{tabular}
\bigskip

\caption{Nonmaximal $\ell$-adic images of $E/\Q$ with potential CM by $\calO$ of discriminant $D$ divisible by a prime $\ell>3$.}\label{table:CMramified}
\end{table}

\begin{proposition}\label{prop:CMj0}
Let $E$ be an elliptic curve over $\Q$ with integral model $y^2=x^3+a$ and potential CM by $\calO=\Z[\zeta_3]$, and let $\ell > 3$ be prime. Then $E$ has maximal $\ell$-adic image unless $4a\ell^e$ is an integer cube, where $e\in\{1,2\}$ satisfies $\ell\equiv \pm 2e\bmod 9$, in which case $\rho_{E,\ell}(G_\Q)$ is the unique index-$3$ subgroup of $\calN_\calO(\ell)$ that does not contain $z_\calO^2$.
\end{proposition}
\begin{proof}
This follows from \cite{zywina:On-the-possible-images-of-the-mod-ell-representations-associated-to-elliptic-curves-over-Q-arxiv}*{Proposition 1.16} and Proposition~\ref{prop:CMtwist} above.
\end{proof}

\begin{corollary}\label{cor:CMj0label}
Let $E/\Q$ be an elliptic curve with potential CM by $\calO=\Z[\zeta_3]$ that has nonmaximal $\ell$-adic image with $\ell>3$.
Then $H\coloneqq \rho_{E,\ell}(G_\Q)$ has index $i=2\ell(\ell-(\frac{-3}{\ell}))/3$ in $\GL_2(\ell)$ and genus $g=3g'-3+(\ell-(\frac{-1}{\ell}))/4+2$ where $g'$ is the genus of $\calN_\calO(\ell)$.

The subgroup $H$ has label $\ell.i.g.1$ under the labeling system defined in \S\ref{ssec:subgroup-labels} .
\end{corollary}
\begin{proof}
  The index formula follows immediately from Proposition~\ref{prop:CMj0}, since $\calN_\calO(\ell)$ contains an index-2 Cartan subgroup of order $(\ell-1)(\ell-(\frac{-3}{\ell}))$, and the genus formula follows from a calculation similar to \cite{baran:normalizersClassNumber}*{Proposition 7.10}.
  % DZB: here is a short explanation of the genus calculation.
  % Baran's calculation shows that the map $X_G \to X_{\ns}^+$ is unramified at the cusps.
  % It is also clearly unramified at elliptic points of order 2 (since the index is 3).
  % It is ramified at some of the elliptic points of order 3, and one must repeat her calculation of coset representatives and of stabilier intersections to finish the computation.
  % Baran only handles the nonsplit case, but the split case is similar.

For $\ell>37$ the group $\GL_2(\ell)$ contains a unique conjugacy class of index $i$ (this follows form Dickson's classification \cite{Dickson:gl2}), which implies that the label for $H$ is as claimed, and a direct computation verifies that this is also true for $3<\ell\le 37$.
\end{proof}

\subsection{\texorpdfstring{$\boldsymbol\ell$-adic CM images for $\boldsymbol{\ell=2,3}$}{l-adic CM images for l=2,3}}

With Propositions~\ref{prop:CMmaximal}, \ref{prop:CMramified}, \ref{prop:CMj0} in hand, the problem of computing $\rho_{E,\ell^\infty}(G_F)$ for an elliptic curve $E$ over $F=\Q(j(E))$ with potential CM is reduced to the case $\ell=2,3$ where we must distinguish the subgroups of $\GL_2(16)$ and $\GL_2(27)$ that can arise for an elliptic curve with potential CM.

\begin{proposition}
There are $59$ subgroups of $\GL_2(16)$ that arise as the reduction of a twist of $\calN_\calO$ for some imaginary quadratic order $\calO$, of which $14$ arise as maximal $2$-adic images (the reduction of some $\calN_\calO$); if one restricts to imaginary quadratic orders $\calO$ of class number $1$, there are $28$, of which $7$ are maximal.

There are $26$ subgroups of $\GL_2(27)$ that arise as the reduction of a twist of $\calN_\calO$ for some imaginary quadratic order $\calO$, of which $7$ are maximal; if one restricts to imaginary quadratic orders $\calO$ of class number $1$, there are $17$, of which $4$ are maximal.
\end{proposition}
\begin{proof}
It is clear from the definition of $\calN_\calO$ that $\calN_\calO(16)$ depends only on the residue class of~$D$ modulo $64$.  It thus suffices to enumerate a representative list of 32 imaginary quadratic discriminants $D$ that realize the residue classes modulo 64 that are squares modulo~4, compute $\calN_\calO(16)$ for each $D$, and then reduce this collection to a set of non-conjugate subgroups of $\GL_2(16)$.  This yields 14 distinct subgroups of $\GL_2(16)$ that arise as maximal twists, and it then suffices to compute the quadratic twists of these subgroups, along with all twists of $\calN_\calO(16)$ for $\calO=\Z[i],\Z[\zeta_3]$ to obtain a complete list of 59 non-conjugate subgroups of $\GL_2(16)$.  Repeating this process for the 13 discriminants of class number 1 yields the first statement of the theorem.  The second statement is proved similarly, but now $\calN_\calO(27)$ depends only on the residue class of $D$ modulo 27.  The function \magmacmd{GL2CMTwists} in the file \repolink{groups}{gl2.m} in \cite{RouseSZB:magma-scripts-ell-adic-images} was used to perform these computations.
\end{proof}

\begin{remark}
That 28 subgroups of $\GL_2(16)$ arise as reductions of twists of $\calN_\calO$ with $\calO$ of class number~$1$ is also noted in \cite{lozanoRobledo:Galois-representations-attached-to-elliptic-curves-with-complex-multiplication}*{page 2}.
\end{remark}

\begin{corollary}
Let $E$ be an elliptic curve with potential CM by an imaginary quadratic order $\calO$ of discriminant $D$ defined over $F=\Q(j(E))$.  If $D\ge -4$ then $\rho_{E,16}(G_F)$ is listed in Table~\ref{table:stwists16} and $\rho_{E,27}(G_F)$ is listed in Table~\ref{table:stwists27}.  If $D<-4$ then $\rho_{E,16}(G_F)$ is listed in Table~\ref{table:CM16a} or Table~\ref{table:CM16b} and $\rho_{E,27}(G_F)$ is listed in Table~\ref{table:CM27}.  Moreover, every group listed in these tables arises for some elliptic curve with potential CM defined over its minimal field of definition.
\end{corollary}
%
%\begin{table}
%\begin{tabular}{rrrrrrrrrrrrrrrr}
%$\ell^e$: & $2^4$ & $3^3$ & 5 & 7 & 11 & 13 & 17 & 19 & 23 & 29 & 31 & 37 & 43 & 67 & 163\\\toprule
%CM $\ell$-adic images: & 59 & 26 & 4 & 6 & 6 & 4 & 3 & 5 & 6 & 4 & 6 & 3 & 6 & 6 & 5\\
%max CM $\ell$-adic images: & 14 & 7 & 3 & 3 & 3 & 3 & 3 & 3 & 3 & 3 & 3 & 3 & 3 & 3 & 3\\
%CM $\ell$-adic images/$\Q$: & 28 & 17 & 3 & 6 & 6 & 3 & 2 & 5 & 3 & 3 & 3 & 2 & 6 & 6 & 5\\
%max CM $\ell$-adic images/$\Q$: & 7 & 4 & 2 & 3 & 3 & 2 & 2 & 3 & 2 & 2 & 2 & 2 & 3 & 3 & 3\\
%\bottomrule
%\end{tabular}
%\bigskip
%
%\caption{Counts of distinct $\rho_{E,\ell^e}(G_F)\le \GL_2(\ell^e)$ arising for elliptic curves~$E$ over $F=\Q(j(E))$ with potential CM, where $e$ is given by Proposition~\ref{prop:CMlevel}.}\label{table:CMimagecounts}
%\end{table}

As can be seen in Tables~\ref{table:stwists16}-\ref{table:CM27}, for any given imaginary quadratic order $\calO$, the 4-orbits and 8-orbits of the possible subgroups $H\le \GL_2(16)$ that arise as twists of $\calN_\calO(16)$ are sufficient to determine $\rho_{E,16}(G_F)$ except for 11 pairs of subgroups.  For each of these 11 pairs there are many conjugacy classes of $\GL_2(16)$ that arise in one element of the pair and not the other.  We can thus compute $\rho_{E,16}(G_F)$ by using division polynomials to compute the 4-torsion and 8-torsion orbits of $E$ as described at the end of the previous subsection and then sampling Frobenius elements to eliminate any remaining ambiguity.
Similar comments apply to $\calN_\calO(27)$ using 3-orbits and 9-orbits to distinguish all but 3 pairs of subgroups that can then be distinguished by sampling Frobenius elements.
The script \repolink{groups}{CMImages.m} in \cite{RouseSZB:magma-scripts-ell-adic-images} verifies these claims.  This yields the following algorithm, which is implemented in the function \magmacmd{GL2CMEllAdicImages} in the file \repolink{groups}{gl2.m} in \cite{RouseSZB:magma-scripts-ell-adic-images}.

\begin{algorithm}
Given an elliptic curve $E$ with potential CM defined over its minimal field of definition $F=\Q(j(E))$ we may compute $\rho_{E,\ell^\infty}(G_F)$ for all primes $\ell$ as follows:
\begin{enumerate}
\item Compute the discriminant $D$ of $\calO\simeq \End(E_{\overline F})$.\footnote{This can be done using the \magma function \magmacmd{HasComplexMultiplication}.}
\item Compute $\rho_{E,16}(G_F)$ and $\rho_{E,27}(G_F)$ using Tables~\ref{table:stwists16}-\ref{table:CM27} as described above.
\item Compute $\rho_{E,\ell}(G_F)$ for $\ell>3$ dividing $D$ via Proposition~\ref{prop:CMramified} (Table~\ref{table:CMramified} if $F=\Q$).
\item If $j(E)=0$ compute $\rho_{E,\ell}(G_F)$ for $\ell$ of bad reduction via Proposition~\ref{prop:CMj0}.
\item $\rho_{E,\ell}(G_F)=\calN_\calO(\ell)$ for all $\ell$ not addressed in steps (1) to (4), by Proposition~\ref{prop:CMmaximal}.
\end{enumerate}
Having determined $\rho_{E,\ell^e}(G_F)$ for all primes $\ell=2,3,7,\ldots$ with $e=4,3,1,1,\ldots$, $\rho_{E,\ell^\infty}(G_F)$ is the inverse image of $\rho_{E,\ell^e}(G_F)$ under the projection $\calN_\calO(\Z_\ell)\to \calN_\calO(\ell^e)$, by Proposition~\ref{prop:CMlevel}.
\end{algorithm}

%----------------------------------------------------------------------------
\subsection*{Acknowledgments}
%----------------------------------------------------------------------------

We thank Jennifer Balakrishnan, Maarten Derickx, Wei Ho, \'Alvaro Lozano-Robledo, Jackson Morrow, Pierre Parent, Bjorn Poonen, William Stein, Mark van Hoeij, John Voight, and David Zywina for useful conversations.  We also thank the referee for helpful comments.  The second author was supported by NSF grant DMS-1522526 and Simons Foundation grant 550033. The third author was supported in part by NSF grant DMS-1555048.
We are grateful to the organizers of, ``Modular Forms, Arithmetic, and Women in Mathematics (Atlanta 2019)'', ``Rational Points (Lichtenfels 2019)'', ``Simons Collaboration on Arithmetic Geometry, Number Theory, and Computation Annual Meeting (New York City 2019)'', ``Torsion groups and Galois representations of elliptic curves (Zagreb 2018)'' and ``Arithmetic Aspects of Explicit Moduli Problems (Banff 2017)'' for providing stimulating environments where progress was made, and the Simons Collaboration on Arithmetic Geometry, Number Theory, and Computation for funding a visit hosted by~MIT.

\FloatBarrier

%----------------------------------------------------------------------------
\vfill\newpage
\appendix
\section{\texorpdfstring{Jacobians of arbitrary modular curves (with John Voight)}{Jacobians of arbitrary modular curves (with John Voight)}}\label{S:kolyvagin}
%----------------------------------------------------------------------------
\numberwithin{theorem}{section}
\setcounter{subsection}{1}

\newcommand{\jv}[1]{{\color{pink} \sf JV: [#1]}}
\newcommand{\psmod}[1]{~(\textup{\text{mod}}~{#1})}

Ribet \cite{ribet:Abelian-varieties-over-Q-and-modular-forms}*{\S 3} observed that the simple abelian varieties attached to newforms on $\Gamma_1(N)$ are of $\GL_2$-type over $\Q$.  In this appendix, we generalize this to any modular curve.  One may anticipate this result from the fact that $\Gamma(N)$ is $\GL_2(\Q)$-conjugate to a subgroup of $\Gamma_1(N^2)$, but the proof requires attention to descent. As a corollary, we show that if $A$ is an isogeny factor of $J_H$ with $r_{\an}(A) = 0$, then $r_{\alg}(A) = 0$.

By a \defi{curve} over a perfect field $k$ in this appendix we mean a smooth, nonsingular scheme of finite type over $k$, possibly disconnected, but with all components of dimension $1$.  Isogenies of abelian varieties are by definition defined over the ground field of the objects, but we will sometimes remind the reader of this convention, for clarity.  

Let $A$ be a simple abelian variety over a number field $F$, and let $\End(A)_\Q$ be the $\Q$-algebra of endomorphisms of $A$ (defined over $F$).  We say $A$ is \defi{of $\GL_2$-type} if there exists a subfield $E \subseteq \End(A)_\Q$ such that $[E:\Q]=\dim A$. Evidently, if $A$ is of $\GL_2$-type over $F$ then so is $A_K$ for any field extension $K \supseteq F$ (taking the same subfield $E$): in fact, $A_K$ is isogenous over $K$ to a power of a simple abelian variety of $\GL_2$-type over $K$ (called a \defi{building block}), by Pyle \cite{MR2058652}*{Proposition 1.3}.

Let $N \in \Z_{\geq 3}$.  Let $\XNgdis$ be the modular curve over $\Q$ attached to the trivial subgroup of $\GL_2(N)$, parameterizing generalized elliptic curves with ``full level $N$ structure'' in the sense of Katz--Mazur \cite{katzmazur}*{(3.1)}.  (When $N=1,2$, in fact $\XNgdis$ is a stack not a scheme, but its components are of genus $0$ and our applications, e.g., Proposition~\ref{prop:j1n2}, hold trivially.)  Then $\XNgdis$ is connected over $\Q$ but becomes disconnected over $\Q(\zeta_N)$, with geometrically connected components indexed by $\Gal(\Q(\zeta_N)\,|\,\Q) \simeq (\Z/N\Z)^\times$ (see also the proof of Lemma~\ref{lem:ohthatsit}).  More precisely, the group $(\Z/N\Z)^\times$ acts geometrically by diamond operators, letting $d \in (\Z/N\Z)^\times$ act via $\smallmat{1}{0}{0}{d}$, which on moduli problems sends $[(E,P,Q)] \mapsto [(E,P,dQ)]$.  Each connected component of $\XNgdis_\C$ is attached to the congruence subgroup $\Gamma(N) \leq \SL_2(\Z)$.  Since these diamond automorphisms are defined on the moduli problem, they are defined over $\Q$, giving $(\Z/N\Z)^\times \hookrightarrow \Aut(\XNgdis)$.

Let $\JNgdis$ be the Jacobian of $\XNgdis$.  Let $\Pic^0 \XNgdis$ be the functor of line bundles whose restriction to each geometric component of $\XNgdis$ has degree $0$.  Although $\XNgdis$ is geometrically disconnected, the functor $\Pic^0 \XNgdis$ is representable by an abelian variety $\JNgdis$ over $\Q$: see e.g.\ Kleiman \cite{kleiman:the-picard-scheme-fga}*{Remark 5.6}. (We note that even though $\XNgdis$ does not admit a degree one divisor, $\Pic^0$ is representable as long as one sheafifies \cite{poonen:rational-points-on-varieties}*{Warning~5.7.10}).   The formation of the Jacobian is functorial, so $\JNgdis_{\Q(\zeta_N)}$ is isomorphic over~$\Q(\zeta_N)$ to the product of the Jacobians of the geometric connected components of $\XNgdis$.  

Next, let $\XNconn \colonequals X_{H_0}$ where $H_0 \leq \GL_2(\Zhat)$ is the preimage of $\{\smallmat{*}{0}{0}{1}\}\le \GL_2(\Z/N\Z)$. (In Derickx--Sutherland \cite{derickxs:quintic-sextic-torsion}*{Section 2},
$\XNconn$ is denoted $X_{0,1}(N,N)$.) We have $\XNconn=X(N)/H_0(N)$, and the moduli problem attached to $\XNconn$ has several equivalent descriptions.  Directly from the description as a quotient, we see that $\XNconn$ parameterizes triples $(E, \phi, P)$, where $E$ is a generalized elliptic curve, $P$ is a point of exact order $N$, and $\phi$ is a cyclic $N$-isogeny such that $E[N]$ is generated by $P$ and $\ker \phi$.    The curve $\XNconn$ is geometrically connected and attached to the congruence subgroup $\Gamma(N) \leq \SL_2(\Z)$.  The diamond operators above are defined by the same formula and their image $(\Z/N\Z)^{\times} \hookrightarrow \Aut(\XNconn)$ gives a subgroup $C$.  

\begin{remark}
See also Poonen--Schaefer--Stoll \cite{mr2309145}*{\S\S 4.1--4.2} who describe $\XNconn$ (for $N=7$ and away from the cusps) as parameterizing isomorphism classes of pairs $(E,\psi)$ where $E$ is a generalized elliptic curve and $\psi \colon \mu_N \times \Z/N\Z \xrightarrow{\sim} E[N]$ is a symplectic isomorphism.
\end{remark}

\begin{lemma} \label{lem:ohthatsit}
The curve $\XNgdis$ over $\Q$ is isomorphic to $\XNconn_{\Q(\zeta_N)}$ \emph{as a curve over $\Q$}.
\end{lemma}

By \emph{as a curve over $\Q$}, we mean the composition $\XNconn_{\Q(\zeta_N)} \to \Spec \Q(\zeta_N) \to \Spec \Q$ as structure morphism, i.e., the product $\XNconn \times_{\Spec \Q} \Spec \Q(\zeta_N)$ viewed as a scheme over~$\Q$.

\begin{proof}
Let $Y(N) \subseteq X(N)$ and $Y(N)' \subseteq \XNconn$ denote the affine curves obtained by restricting the moduli problem to elliptic curves.  In the terminology of Katz--Mazur \cite{katzmazur}*{(5.6)}, for a point $[(E,P,Q)] \in Y(N)(R)$ over a $\Q$-algebra $R$ where $E$ is an elliptic curve, we have its determinant $e_N(P,Q)$ (obtained equivalently by the Weil pairing), an element of $R[\zeta_N] \colonequals R[x]/(\Phi_N(x))$ where $\Phi_N(x)$ is the $N$th cyclotomic polynomial.  Refining the natural assignment above, we have a map $[(E,P,Q)] \mapsto [(E,\phi,P;e_N(P,Q))]$ on points which induces a map $Y(N) \to Y(N)'_{\Q[\zeta_N]}$ of nonsingular curves over $\Q$.  Over any geometric point, this map is bijective, so as a map of nonsingular curves it is an isomorphism.  Thus we obtain a (unique) map $\XNgdis \to \XNconn_{\Q(\zeta_N)}$ of projective, nonsingular curves over $\Q$.
\end{proof}

For a Dirichlet character $\chi \colon (\Z/N\Z)^\times \to \C^\times$ of modulus $N$, let $\XNconn_\chi$ be the twist of $\XNconn$ attached to $\chi$ (with respect to $C \leq \Aut(\XNconn)$) and let $\JNconn_\chi$ be its Jacobian.  For a precise description of this twisting operation, see Kida \cite{kida:Galois-descent-and-twists-of-an-abelian-variety}*{Section 2}.

\begin{corollary} \label{cor:val}
Up to isogeny over $\Q$, we have
\[
\JNgdis \sim \Res_{\Q(\zeta_N)|\Q}(\JNconn_{\Q(\zeta_N)}) \sim \prod_{\chi} \JNconn_\chi,
\]
where $\Res_{\Q(\zeta_N)|\Q}$ is restriction of scalars and $\chi$ ranges over Dirichlet characters of modulus~$N$.
\end{corollary}

\begin{proof}
By the universal property of the restriction of scalars, the Jacobian of $\XNconn_{\Q(\zeta_N)}$ as a curve over $\Q$ is the restriction of scalars $\Res_{\Q(\zeta_N)|\Q}(J(N)')$---a line bundle on $\XNconn_{\Q(\zeta_N)}$ as a curve over $\Q$ is just a line bundle on $\XNconn_{\Q(\zeta_N)}$---so the first isogeny follows from Lemma~\ref{lem:ohthatsit}.  The second isogeny follows from a result of Kida \cite{kida:Galois-descent-and-twists-of-an-abelian-variety}*{Theorem, p.\ 53}.
\end{proof}

Let $\Qhat \colonequals \Zhat \otimes_\Z \Q$ and let 
\[ \varpi = \varpi_N \colonequals \begin{bmatrix} 0 & 1 \\ N & 0 \end{bmatrix} \in \GL_2(\Qhat). \]
To prove our main result, we study the effect of conjugating $H$-level structures by $\varpi$, when this is defined.  
%an analogue of Lemma \ref{lem:x1xmu} for $\Gamma_0(N^2) \cap \Gamma_1(N)$. To setup we prove the following more general fact. Let $N$ be an integer and let $\gamma = \smallmat{N}{0}{0}{1} \in \GL_2(\Zhat)$. 
Conjugation by $\varpi$ in $\GL_2(\Qhat)$ is given by the formula
\[
\varpi \begin{bmatrix} a & b \\ c & d \end{bmatrix} \varpi^{-1} = \begin{bmatrix} d & N^{-1}c \\ Nb & a \end{bmatrix}.
\]  
Let $H \leq \GL_2(\Zhat)$ be an open subgroup and suppose that its reduction $H(N) \in \GL_2(\Z/N\Z)$ is contained in the Borel subgroup $B(N)$ of upper-triangular matrices.  Then the level of $H$ is divisible by $N$.  We define 
\[
H^{\varpi} \colonequals \varpi H \varpi^{-1} = \{\smallmat{d}{c}{Nb}{a} : \smallmat{a}{b}{Nc}{d} \in H\} \leq \GL_2(\Zhat). \]
The group $H^{\varpi}$ is again an open subgroup.  
% Swapping the basis conjugates $H^{\varpi}$ to $\{\smallmat{d}{b}{Nc}{a} : \smallmat{a}{c}{Nb}{d} \in H\}$, which, mod $N$, is a subgroup of $B(N)$.
If $H$ has level $MN$, then visibly $H^{\varpi}$ has level dividing $MN^2$, and this bound is in general sharp: for example, if $H$ is the level $N$ preimage of $\{\smallmat{*}{0}{0}{*}\} \leq \GL_2(\Z/N\Z)$, then $H^{\varpi}=\{\smallmat{*}{*}{0}{*}\} \leq \GL_2(\Z/N^2\Z)$, which has level $N^2$.

\begin{lemma}\label{lemma:conjugation}
Suppose that $H(N) \leq B(N)$.  Then the modular curves $X_H$ and $X_{H^\varpi}$ are isomorphic over $\Q$.  
\end{lemma}

\begin{proof}
We apply a result of Deligne--Rapoport \cite{DeligneRapoport}*{Propositions IV-3.16 and IV-3.19} with $K=H$ and $g=\varpi$: the isomorphism provided is defined on the (open) moduli problem and extends to an isomorphism of curves over $\Q$.  
\end{proof}

The preceding lemma explains in many cases why modular curves are in fact isomorphic over $\Q$ (even though they may look like twists!).  The following example illustrates how conjugation by $\varpi$ transforms one moduli problem into an isomorphic moduli problem (over~$\Q$).

\begin{example} \label{exm:x1xmu}
Let
\[ H_1 \colonequals \{\smallmat{1}{*}{0}{*}\} \leq \GL_2(\Z/N\Z). \]
Then $\Gamma_{H_1} = \Gamma_1(N)$ (following the conventions in Section \ref{ssec:group-theory}), and the points of $X_{H_1}$ are given by classes $[(E,P)]$ where $E$ is a generalized elliptic curve and $P \in E[N]$ is a point of order $N$. 

We compute that $H_1^{\varpi} = \{\smallmat{*}{*}{0}{1}\} \leq \GL_2(\Z/N\Z)$ has level $N$.  (Swapping basis vectors transforms $H_1^{\varpi}$ to the transpose $H_1^T = \{\smallmat{1}{0}{*}{*}\} \leq \GL_2(\Z/N\Z)$.)  The points of $X_{H_1^T}$ are classes $[(E,\upsilon)]$ where $E$ is a generalized elliptic curve and $\upsilon \colon \mu_N \hookrightarrow E[N]$ is an inclusion.  We have $\Gamma_{H_1^\varpi}=\Gamma^1(N)$ (conjugate to $\Gamma_1(N)$ in $\SL_2(\Z)$), but of course the moduli interpretation is different for the two cases.

Lemma \ref{lemma:conjugation} shows that in fact we have an isomorphism $X_{H_1} \xrightarrow{\sim} X_{H_1^{\varpi}}$ defined over $\Q$.  In terms of the moduli problems above, this isomorphism is defined as follows: the point $[(E,\upsilon)]$ over a base ring $R$ maps to the pair $[(E',Q')]$ over $R$, where $E' \colonequals E/\upsilon(\mu_N)$ and $Q' \in E'[N]$ is the image of any point $Q \in E[N]$ such that $e_N(\upsilon(\zeta_N), Q) = \zeta_N$ under the Weil pairing $e_N$.  This isomorphism was observed by Poonen--Schaefer--Stoll \cite{mr2309145}*{Remark 5.1} in the case $N=7$, where they write $X_\mu(N)$ for $X_{H_1^T}$. 
\end{example}

With this setup, we can prove our first main result.  

\begin{proposition} \label{prop:j1n2}
For all Dirichlet characters $\chi$, every simple isogeny factor of $\JNconn_\chi$ over $\Q$ is an isogeny factor of $J_1(N^2)$ and of $\GL_2$-type over $\Q$.  More precisely, each such factor is isogenous to a modular abelian variety of the form $A_f$, where $f$ is a Galois orbit of newforms of level dividing $N^2$ and character of conductor dividing $N$.  
\end{proposition}

\begin{proof}
%As a warmup, we begin with $\JNconn$ itself (i.e., the case where $\chi$ is trivial).  There is a non-constant map defined (on points of the associated moduli problems)
%\begin{align*}
%\pi \colon X_1(N^2) &\to X(N)' \\
%[(E,P)] &\mapsto [(E/\langle NP\rangle, \phi(P), \psi)],
%\end{align*}
%where $\phi\colon E \to E/\langle NP\rangle$ is the quotient map and $\psi\colon E/\langle NP\rangle \to E$ is the dual isogeny to $\psi$; this proves the first part of the statement.  For the second, Ribet \cite{ribet:Abelian-varieties-over-Q-and-modular-forms}*{\S 3} observed that every isogeny factor of $J_1(N^2)$ is of $\GL_2$-type over $\Q$: more precisely, we have a decomposition $J_1(N^2)$ into a product of simple abelian varieties $A_f$ (possibly with multiplicity) with $f$ a Galois orbit of newforms of level dividing $N^2$.  
Let
\[ H \colonequals \left\{ \begin{bmatrix} a & b \\ c & d \end{bmatrix} : a \equiv 1 \psmod{N},\ c \equiv 0 \psmod{N^2} \right\} \leq \GL_2(\Zhat). \]
Then $\Gamma_H = \Gamma_0(N^2) \cap \Gamma_1(N)$; in particular, there is a natural surjective map $X_1(N^2) \to X_H$.  

By Lemma \ref{lemma:conjugation}, we have an isomorphism of modular curves $X_H \xrightarrow{\sim} X_{H^{\varpi}}$.  We compute directly that $H^{\varpi}=H_0$ is the preimage of $\{\smallmat{*}{0}{0}{1}\}\le \GL_2(\Z/N\Z)$, so $X_{H^{\varpi}}=\XNconn$.

Ribet \cite{ribet:Abelian-varieties-over-Q-and-modular-forms}*{\S 3} observed that every isogeny factor of $J_1(N^2)$ is of $\GL_2$-type over $\Q$: more precisely, we have a decomposition $J_1(N^2)$ into a product of simple abelian varieties $A_f$ over~$\Q$, possibly with multiplicity, where $f$ is a Galois orbit of newforms of level dividing $N^2$.  In fact, the natural quotient map $X_1(N^2) \to X_H$ is defined by the subgroup of the diamond operators $\{\langle d\rangle :d \in (\Z/N^2\Z)^\times \text{ with } d\equiv 1\bmod N\}$, so the cusp forms on $X_H$ are precisely those with conductor dividing $N$ (so of modulus $N$).  Finally, the isomorphism $X_H \xrightarrow{\sim} X(N)'$ over~$\Q$ transports Ribet's observation to the Jacobian $J(N)'$, as claimed.  

To finish, we consider the twist: from the decomposition in the previous paragraph,
\[ \JNconn_\chi \sim \prod_f A_{f \otimes \chi}^{e_f} \]
where $f \otimes \chi$ is the twist of $f$ by the character $\chi$.  Let $\psi$ be the character of $f$, of modulus $N$ and conductor $\cond(\psi)$.  By Shimura \cite{shimura1996introarithmetic}*{Proposition 3.64}, the level of $f \otimes \chi$ divides
\[ \lcm(N^2, \cond(\chi)^2,\cond(\chi)\cond(\psi)) = N^2 \]
and moreover $f \otimes \chi$ has character $\psi\chi^2$ still of modulus $N$.  Thus $A_{f \otimes \chi}$ is an isogeny factor over $\Q$ of $J_1(N^2)$, and the result follows as in the previous case.  
\end{proof}

\begin{theorem}\label{thm:finalans}
Let $H \leq \GL_2(\Zhat)$ be an open subgroup of level $N$, let $X_H$ be the modular curve attached to~$H$, and let $J_H$ be its Jacobian.  Then each simple factor of $J_H$ is isogenous to a simple factor of $J_1(N^2)$.  In particular, every simple factor of $J_H$ is of $\GL_2$-type.  
\end{theorem}

\begin{proof}
Let $A$ be a simple factor of $J_H$.  We have a map $\XNgdis \to X_H$, realizing $J_H$ as an isogeny factor of $\JNgdis$, so $A$ is a simple isogeny factor of $\JNgdis$.  Corollary \ref{cor:val} implies that $A$ is a simple isogeny factor of $\JNconn_\chi$ for some Dirichlet character $\chi$.  By Proposition \ref{prop:j1n2}, every simple factor of $\JNconn_\chi$ is an isogeny factor of $J_1(N^2)$, hence of $\GL_2$-type.
\end{proof}

\begin{corollary}\label{coro:kolyvagin}
Let $A$ be an isogeny factor of $J_H$ of analytic rank $0$. Then $A(\Q)$ is finite.
\end{corollary}

\begin{proof}
Apply Kato \cite{KatoModForms}*{Cor.~14.3}.
\end{proof}

\begin{remark}
Assaf \cite{Assaf} has given a method for computing Hecke operators on arbitrary congruence subgroups.  According to Theorem \ref{thm:finalans}, the Galois orbits of cusp forms appearing can be matched with classical modular forms according to the corresponding system of Hecke eigenvalues (forms with level dividing $N^2$ and character of modulus $N$), thereby explicitly providing the isotypic decomposition of $J_H$ up to isogeny over $\Q$.  

This method is indeed sensitive to the group $H$!  For example, the groups \glink{16.24.1.6}, \glink{16.24.1.7}, \glink{16.24.1.8} (numbered $163$, $162$, $164$ by Rouse--Zureick-Brown \cite{RZB}) have the same intersection with $\SL_2(\Z/16\Z)$, with Cummins-Pauli label \href{https://mathstats.uncg.edu/sites/pauli/congruence/csg1M.html#level16}{$\mathrm{16D}^1$}, as can be seen in Table~\ref{table:groups2}, but the associated modular Jacobians are nonisogenous elliptic curves \href{https://www.lmfdb.org/EllipticCurve/Q/128/c/2}{\texttt{128.c2}}, \href{https://www.lmfdb.org/EllipticCurve/Q/128/b/2}{\texttt{128.b2}}, \href{https://www.lmfdb.org/EllipticCurve/Q/128/d/2}{\texttt{128.d2}} with $j$-invariant 128 (so isomorphic over $\C$ but not even isogenous over $\Q$).  
\end{remark}

\begin{remark}
The same proof works for an arbitrary (possibly
geometrically disconnected) Jacobian of a Shimura curve, attached to a quaternion algebra over a totally
real field with a unique real split place. The step which generalizes Ribet's observation (there are Hecke operators and degeneracy operators) is given by Hida \cite{hida:On-abelian-varieties-with-complex-multiplication-as-factors-of-the-Jacobians-of-Shimura-curves}*{Proposition 4.8}.
\end{remark}

%----------------------------------------------------------------------------
\subsection*{Acknowledgements} The authors would like to thank Asher Auel and Bjorn Poonen for helpful conversations.  Voight was supported by a Simons Collaboration Grant (550029).
%----------------------------------------------------------------------------

%----------------------------------------------------------------------------
%\vfill\newpage
\section{Tables}\label{S:tables}
%----------------------------------------------------------------------------

The tables that follow provide data for various subgroups $H$ of $\ell$-power level that include:
\begin{itemize}\setlength{\itemsep}{0pt}
\item the label $\texttt{N.i.g.n}$ of $H$ as defined in Section~\ref{ssec:subgroup-labels};
\item a list of generators for $H$ in $\GL_2(N)$ (which we believe is minimal);
\item for $g(H)\le 24$, the label of $H\cap \SL_2(\Z)$ assigned by Cummins and Pauli \cites{CumminsPauli:congruence-subgroups,CumminsPauli:database};
\item the number of cusps $\#X_H^\infty(\Qbar)$ and rational cusps $\#X_H^\infty(\Q)$, as described in Section~\ref{sec:modular-curves};
\item the analytic rank $r$ of $J_H:=\Jac(X_H)$, whose computation is described in Section \ref{sec:analytic-rank};
\item the genus $g$ of $X_H$;
\item the dimensions and LMFDB labels of the Galois orbits of newforms $[f]$ whose modular abelian varieties $A_f$ are isogeny factors of $J_H$ (exponents denote multiplicities).
\vspace{-10pt}
\end{itemize}

\begin{table}
\begin{center}
\small
\begin{tabular}{lllccccl}
label & generators & CP & \hspace{-12pt}$\#X_H^\infty(\Qbar)$ & \hspace{-4pt}$\#X_H^\infty(\Q)$ & $r$ & $g$ & dimensions\\\toprule
\gtarget{16.24.1.5} & $\smallmat{0}{5}{1}{10}, \smallmat{4}{5}{5}{2}, \smallmat{0}{3}{11}{10}, \smallmat{13}{10}{8}{7}$ & \href{https://mathstats.uncg.edu/sites/pauli/congruence/csg1M.html#level16}{16D\textsuperscript{1}} & 2 & 2 & 1 & 1 & $1$\\[2pt]
& \multicolumn{7}{l}{\parbox[l]{12.5cm}{\raggedright\mflabel{128.2.a.a}}}\\\hline\noalign{\vskip 3pt}
\gtarget{16.24.1.10} & $\smallmat{14}{9}{9}{2}, \smallmat{15}{6}{6}{15}, \smallmat{9}{2}{10}{7}, \smallmat{12}{3}{3}{12}$ & \href{https://mathstats.uncg.edu/sites/pauli/congruence/csg1M.html#level8}{8D\textsuperscript{1}} & 3 & 1 & 1 & 1 & $1$\\[2pt]
& \multicolumn{7}{l}{\parbox[l]{12.5cm}{\raggedright\mflabel{256.2.a.a}}}\\\hline\noalign{\vskip 3pt}
\gtarget{16.24.1.11} & $\smallmat{4}{7}{5}{12}, \smallmat{12}{3}{1}{0}, \smallmat{2}{1}{3}{10}, \smallmat{0}{9}{5}{0}$ & \href{https://mathstats.uncg.edu/sites/pauli/congruence/csg1M.html#level8}{8D\textsuperscript{1}} & 3 & 1 & 1 & 1 & $1$\\[2pt]
& \multicolumn{7}{l}{\parbox[l]{12.5cm}{\raggedright\mflabel{256.2.a.a}}}\\\hline\noalign{\vskip 3pt}
\gtarget{16.24.1.13} & $\smallmat{5}{4}{8}{3}, \smallmat{10}{15}{3}{6}, \smallmat{0}{15}{1}{4}, \smallmat{12}{9}{11}{0}$ & \href{https://mathstats.uncg.edu/sites/pauli/congruence/csg1M.html#level16}{16C\textsuperscript{1}} & 3 & 1 & 1 & 1 & $1$\\
\bottomrule
\end{tabular}
\bigskip

\caption{Subgroups $H$ of $\ell$-power level that have rational non-cuspidal non-CM points but do not arise as $\rho_{E,\ell^\infty}(G_\Q)$ for any non-CM elliptic curve $E/\Q$ (the non-cuspidal non-CM points in $X_H(\Q)$ all correspond to elliptic curves $E/\Q$ with $\rho_{E,\ell^\infty}(\Q)$ properly contained in $H$).}\label{table:missing}
\vspace{-10pt}
\end{center}
\end{table}

\begin{table}
\begin{center}
\small
\begin{tabular}{lllccccl}
label & generators & CP & \hspace{-12pt}$\#X_H^\infty(\Qbar)$ & \hspace{-4pt}$\#X_H^\infty(\Q)$ & $r$ & $g$ & dimensions\\\toprule
\gtarget{9.12.1.1} & $\smallmat{5}{5}{6}{8}, \smallmat{7}{0}{0}{8}$ & \href{https://mathstats.uncg.edu/sites/pauli/congruence/csg1M.html#level9}{9A\textsuperscript{1}} & 2 & 2 & 0 & 1 & $1$\\[1pt]
& \multicolumn{7}{l}{\parbox[l]{12.5cm}{\raggedright\mflabel{27.2.a.a}}}\\\hline\noalign{\vskip 3pt}
\gtarget{9.36.2.1} & $\smallmat{2}{3}{6}{4}, \smallmat{2}{6}{6}{8}$ & \href{https://mathstats.uncg.edu/sites/pauli/congruence/csg2M.html#level9}{9A\textsuperscript{2}} & 4 & 2 & 0 & 2 & $2$\\[1pt]
& \multicolumn{7}{l}{\parbox[l]{12.5cm}{\raggedright\mflabel{81.2.a.a}}}\\\hline\noalign{\vskip 3pt}
\gtarget{9.54.1.1} & $\smallmat{0}{1}{2}{0}, \smallmat{7}{0}{0}{5}$ & \href{https://mathstats.uncg.edu/sites/pauli/congruence/csg1M.html#level9}{9E\textsuperscript{1}} & 6 & 1 & 0 & 1 & $1$\\[1pt]
& \multicolumn{7}{l}{\parbox[l]{12.5cm}{\raggedright\mflabel{27.2.a.a}}}\\\hline\noalign{\vskip 3pt}
\gtarget{9.54.2.2} & $\smallmat{3}{2}{1}{3}, \smallmat{5}{3}{3}{4}$ & \href{https://mathstats.uncg.edu/sites/pauli/congruence/csg2M.html#level9}{9B\textsuperscript{2}} & 6 & 0 & 0 & 2 & $2$\\
\bottomrule
\end{tabular}
\bigskip

\caption{Arithmetically maximal subgroups of $3$-power level.}\label{table:groups3}
\vspace{-10pt}
\end{center}
\end{table}

\begin{table}
\begin{center}
\small
\begin{tabular}{lllccccl}
label & generators & CP & \hspace{-12pt}$\#X_H^\infty(\Qbar)$ & \hspace{-4pt}$\#X_H^\infty(\Q)$ & $r$ & $g$ & dimensions\\\toprule
\gtarget{25.50.2.1} & $\smallmat{3}{12}{11}{7}, \smallmat{11}{7}{14}{21}$ & \href{https://mathstats.uncg.edu/sites/pauli/congruence/csg2M.html#level25}{25E\textsuperscript{2}} & 2 & 0 & 2 & 2 & $2$\\[2pt]
& \multicolumn{7}{l}{\parbox[l]{12.5cm}{\raggedright\mflabel{625.2.a.a}}}\\\hline\noalign{\vskip 3pt}
\gtarget{25.75.2.1} & $\smallmat{10}{3}{23}{5}, \smallmat{12}{10}{10}{16}$ & \href{https://mathstats.uncg.edu/sites/pauli/congruence/csg2M.html#level25}{25F\textsuperscript{2}} & 7 & 0 & 2 & 2 & $2$\\[2pt]
& \multicolumn{7}{l}{\parbox[l]{12.5cm}{\raggedright\mflabel{625.2.a.b}}}\\\hline\noalign{\vskip 3pt}
\gtarget{25.150.4.1} & $\smallmat{18}{15}{15}{8}, \smallmat{7}{10}{5}{14}$ & \href{https://mathstats.uncg.edu/sites/pauli/congruence/csg4M.html#level25}{25F\textsuperscript{4}} & 14 & 5 & 0 & 4 & $4$\\[2pt]
& \multicolumn{7}{l}{\parbox[l]{12.5cm}{\raggedright\mflabel{25.2.d.a}}}\\\hline\noalign{\vskip 3pt}
\gtarget{25.150.4.2} & $\smallmat{7}{20}{20}{7}, \smallmat{22}{2}{13}{22}$ & \href{https://mathstats.uncg.edu/sites/pauli/congruence/csg4M.html#level25}{25F\textsuperscript{4}} & 14 & 0 & 0 & 4 & $4$\\
\bottomrule
\end{tabular}
\bigskip

\caption{Arithmetically maximal subgroups of $5$-power level.}\label{table:groups5}
\end{center}
\end{table}

\begin{table}
\begin{center}
\small
\begin{tabular}{lllccccl}
label & generators & CP & \hspace{-12pt}$\#X_H^\infty(\Qbar)$ & \hspace{-4pt}$\#X_H^\infty(\Q)$ & $r$ & $g$ & dimensions\\\toprule
\gtarget{7.42.1.1} & $\smallmat{3}{5}{0}{4}, \smallmat{5}{1}{5}{2}$ & \href{https://mathstats.uncg.edu/sites/pauli/congruence/csg1M.html#level7}{7A\textsuperscript{1}} & 6 & 0 & 0 & 1 & $1$\\[2pt]
& \multicolumn{7}{l}{\parbox[l]{12.5cm}{\raggedright\mflabel{49.2.a.a}}}\\\hline\noalign{\vskip 3pt}
\gtarget{7.56.1.1} & $\smallmat{6}{0}{0}{4}, \smallmat{4}{0}{0}{5}$ & \href{https://mathstats.uncg.edu/sites/pauli/congruence/csg1M.html#level7}{7B\textsuperscript{1}} & 8 & 2 & 0 & 1 & $1$\\[2pt]
& \multicolumn{7}{l}{\parbox[l]{12.5cm}{\raggedright\mflabel{49.2.a.a}}}\\\hline\noalign{\vskip 3pt}
\gtarget{7.56.1.2} & $\smallmat{0}{3}{6}{0}, \smallmat{0}{2}{2}{0}$ & \href{https://mathstats.uncg.edu/sites/pauli/congruence/csg1M.html#level7}{7B\textsuperscript{1}} & 8 & 0 & 0 & 1 & $1$\\[2pt]
& \multicolumn{7}{l}{\parbox[l]{12.5cm}{\raggedright\mflabel{49.2.a.a}}}\\\hline\noalign{\vskip 3pt}
\gtarget{49.56.1.1} & $\smallmat{39}{18}{0}{5}, \smallmat{47}{33}{0}{33}$ & \href{https://mathstats.uncg.edu/sites/pauli/congruence/csg1M.html#level49}{49A\textsuperscript{1}} & 8 & 2 & 0 & 1 & $1$\\
\bottomrule
\end{tabular}
\bigskip

\caption{Arithmetically maximal subgroups of $\ell$-power level for $\ell=7,11$.}\label{table:groups7and11}
\end{center}
\end{table}

\begin{table}
\begin{center}
\small
\begin{tabular}{lllccccl}
label & generators & CP & \hspace{-12pt}$\#X_H^\infty(\Qbar)$ & \hspace{-4pt}$\#X_H^\infty(\Q)$ & $r$ & $g$ & dimensions\\\toprule
\gtarget{13.78.3.1} & $\smallmat{0}{11}{8}{0}, \smallmat{11}{2}{10}{2}$ & \href{https://mathstats.uncg.edu/sites/pauli/congruence/csg3M.html#level13}{13A\textsuperscript{3}} & 6 & 0 & 3 & 3 & $3$\\[2pt]
& \multicolumn{7}{l}{\parbox[l]{12.5cm}{\raggedright\mflabel{169.2.a.b}}}\\\hline\noalign{\vskip 3pt}
\gtarget{13.84.2.1} & $\smallmat{1}{10}{0}{6}, \smallmat{12}{8}{0}{11}$ & \href{https://mathstats.uncg.edu/sites/pauli/congruence/csg2M.html#level13}{13A\textsuperscript{2}} & 12 & 6 & 0 & 2 & $2$\\[2pt]
& \multicolumn{7}{l}{\parbox[l]{12.5cm}{\raggedright\mflabel{13.2.e.a}}}\\\hline\noalign{\vskip 3pt}
\gtarget{13.84.2.2} & $\smallmat{3}{7}{0}{8}, \smallmat{12}{4}{0}{12}$ & \href{https://mathstats.uncg.edu/sites/pauli/congruence/csg2M.html#level13}{13A\textsuperscript{2}} & 12 & 0 & 0 & 2 & $2$\\[2pt]
& \multicolumn{7}{l}{\parbox[l]{12.5cm}{\raggedright\mflabel{169.2.e.a}}}\\\hline\noalign{\vskip 3pt}
\gtarget{13.84.2.3} & $\smallmat{9}{2}{0}{7}, \smallmat{4}{4}{0}{7}$ & \href{https://mathstats.uncg.edu/sites/pauli/congruence/csg2M.html#level13}{13A\textsuperscript{2}} & 12 & 0 & 0 & 2 & $2$\\
\bottomrule
\end{tabular}
\end{center}
\bigskip

\caption{Arithmetically maximal subgroups of $\ell$-power level for $\ell=13,17,19$.}\label{table:groups13and19}
\end{table}

\begin{table}
\begin{center}
\small
\begin{tabular}{lllccccl}
label & generators & CP & \hspace{-12pt}$\#X_H^\infty(\Qbar)$ & \hspace{-4pt}$\#X_H^\infty(\Q)$ & $r$ & $g$ & dimensions\\\toprule
\gtarget{23.24.2.1} & $\smallmat{13}{13}{0}{11}, \smallmat{20}{0}{0}{18}$ & \href{https://mathstats.uncg.edu/sites/pauli/congruence/csg2M.html#level23}{23A\textsuperscript{2}} & 2 & 2 & 0 & 2 & $2$\\[2pt]
& \multicolumn{7}{l}{\parbox[l]{12.5cm}{\raggedright\mflabel{23.2.a.a}}}\\\hline\noalign{\vskip 3pt}
\gtarget{23.253.13.1} & $\smallmat{1}{19}{1}{1}, \smallmat{16}{13}{9}{7}$ & \href{https://mathstats.uncg.edu/sites/pauli/congruence/csg13M.html#level23}{23A\textsuperscript{13}} & 11 & 0 & 13 & 13 & $4^{2}, 5$\\[2pt]
& \multicolumn{7}{l}{\parbox[l]{12.5cm}{\raggedright\mflabel{529.2.a.g}, \mflabel{529.2.a.h}, \mflabel{529.2.a.i}}}\\\hline\noalign{\vskip 3pt}
\gtarget{23.276.15.1} & $\smallmat{0}{12}{2}{0}, \smallmat{0}{21}{18}{0}$ & \href{https://mathstats.uncg.edu/sites/pauli/congruence/csg15M.html#level23}{23A\textsuperscript{15}} & 12 & 1 & 13 & 15 & $2, 4^{2}, 5$\\[2pt]
& \multicolumn{7}{l}{\parbox[l]{12.5cm}{\raggedright\mflabel{23.2.a.a}, \mflabel{529.2.a.g}, \mflabel{529.2.a.h}, \mflabel{529.2.a.i}}}\\[-2pt]
\midrule
\gtarget{29.30.2.1} & $\smallmat{23}{26}{0}{12}, \smallmat{27}{2}{0}{4}$ & \href{https://mathstats.uncg.edu/sites/pauli/congruence/csg2M.html#level29}{29A\textsuperscript{2}} & 2 & 2 & 0 & 2 & $2$\\
\bottomrule
\end{tabular}
\bigskip

\caption{Arithmetically maximal subgroups of $\ell$-power level for $\ell=23,29,31,37$.}\label{table:groups23to37}
\end{center}
\end{table}
\begin{table}
\begin{center}\small
\begin{tabular}{lllccccl}
label & generators & CP & \hspace{-12pt}$\#X_H^\infty(\Qbar)$ & \hspace{-4pt}$\#X_H^\infty(\Q)$ & $r$ & $g$ & dims\\\toprule
\gtarget{8.12.1.1} & $\smallmat{0}{3}{5}{4}, \smallmat{0}{3}{1}{0}, \smallmat{5}{2}{0}{3}, \smallmat{2}{5}{5}{0}$ & \href{https://mathstats.uncg.edu/sites/pauli/congruence/csg1M.html#level8}{8A\textsuperscript{1}} & 2 & 2 & 0 & 1 & $1$\\[0.5pt]
& \multicolumn{7}{l}{\parbox[l]{12.5cm}{\raggedright\mflabel{32.2.a.a}}}\\[-1pt]\hline\noalign{\vskip 2pt}
\gtarget{8.12.1.2} & $\smallmat{5}{2}{4}{3}, \smallmat{0}{3}{3}{6}, \smallmat{6}{1}{7}{6}$ & \href{https://mathstats.uncg.edu/sites/pauli/congruence/csg1M.html#level8}{8A\textsuperscript{1}} & 2 & 2 & 0 & 1 & $1$\\[0.5pt]
& \multicolumn{7}{l}{\parbox[l]{12.5cm}{\raggedright\mflabel{64.2.a.a}}}\\[-1pt]\hline\noalign{\vskip 2pt}
\gtarget{8.12.1.3} & $\smallmat{0}{3}{7}{4}, \smallmat{0}{3}{3}{0}, \smallmat{6}{1}{3}{6}, \smallmat{1}{4}{0}{7}$ & \href{https://mathstats.uncg.edu/sites/pauli/congruence/csg1M.html#level8}{8A\textsuperscript{1}} & 2 & 2 & 0 & 1 & $1$\\[0.5pt]
& \multicolumn{7}{l}{\parbox[l]{12.5cm}{\raggedright\mflabel{32.2.a.a}}}\\[-1pt]\hline\noalign{\vskip 2pt}
\gtarget{8.12.1.4} & $\smallmat{4}{5}{7}{0}, \smallmat{1}{2}{2}{1}, \smallmat{5}{0}{0}{7}, \smallmat{3}{6}{6}{7}$ & \href{https://mathstats.uncg.edu/sites/pauli/congruence/csg1M.html#level8}{8A\textsuperscript{1}} & 2 & 2 & 0 & 1 & $1$\\
\end{tabular}
\end{center}
\end{table}
\begin{table}
\begin{center}\small
\begin{tabular}{lllccccl}
label & generators & CP & \hspace{-12pt}$\#X_H^\infty(\Qbar)$ & \hspace{-4pt}$\#X_H^\infty(\Q)$ & $r$ & $g$ & dims\\\toprule
\gtarget{16.24.1.18} & $\smallmat{7}{10}{4}{15}, \smallmat{15}{14}{2}{13}, \smallmat{5}{9}{8}{7}, \smallmat{3}{12}{12}{7}$ & \href{https://mathstats.uncg.edu/sites/pauli/congruence/csg1M.html#level16}{16B\textsuperscript{1}} & 2 & 2 & 0 & 1 & $1$\\[0.5pt]
& \multicolumn{7}{l}{\parbox[l]{12.5cm}{\raggedright\mflabel{256.2.a.c}}}\\[-1pt]\hline\noalign{\vskip 2pt}
\gtarget{16.24.1.20} & $\smallmat{5}{0}{2}{11}, \smallmat{1}{10}{2}{7}, \smallmat{11}{13}{0}{5}, \smallmat{9}{12}{6}{11}$ & \href{https://mathstats.uncg.edu/sites/pauli/congruence/csg1M.html#level16}{16B\textsuperscript{1}} & 2 & 2 & 0 & 1 & $1$\\[0.5pt]
& \multicolumn{7}{l}{\parbox[l]{12.5cm}{\raggedright\mflabel{256.2.a.c}}}\\[-1pt]\hline\noalign{\vskip 2pt}
\gtarget{16.24.2.1} & $\smallmat{14}{11}{11}{8}, \smallmat{4}{1}{11}{0}, \smallmat{8}{5}{9}{2}, \smallmat{2}{11}{1}{6}$ & \href{https://mathstats.uncg.edu/sites/pauli/congruence/csg2M.html#level16}{16A\textsuperscript{2}} & 2 & 2 & 0 & 2 & $2$\\[0.5pt]
& \multicolumn{7}{l}{\parbox[l]{12.5cm}{\raggedright\mflabel{128.2.b.a}}}\\[-1pt]\hline\noalign{\vskip 2pt}
\gtarget{16.24.2.2} & $\smallmat{11}{6}{0}{5}, \smallmat{1}{3}{2}{5}, \smallmat{11}{7}{10}{11}, \smallmat{13}{2}{4}{13}$ & \href{https://mathstats.uncg.edu/sites/pauli/congruence/csg2M.html#level16}{16A\textsuperscript{2}} & 2 & 2 & 0 & 2 & $2$\\
\end{tabular}
\end{center}
\end{table}

\begin{table}
\begin{center}\small
\begin{tabular}{lllccccl}
label & generators & CP & \hspace{-12pt}$\#X_H^\infty(\Qbar)$ & \hspace{-4pt}$\#X_H^\infty(\Q)$ & $r$ & $g$ & dims\\\toprule
\gtarget{16.96.2.15} & $\smallmat{9}{8}{4}{1}, \smallmat{9}{4}{8}{9}, \smallmat{7}{8}{8}{1}, \smallmat{15}{0}{0}{7}, \smallmat{15}{12}{0}{5}$ & \href{https://mathstats.uncg.edu/sites/pauli/congruence/csg2M.html#level16}{16I\textsuperscript{2}} & 14 & 6 & 0 & 2 & $2$\\[0.5pt]
& \multicolumn{7}{l}{\parbox[l]{12.5cm}{\raggedright\mflabel{16.2.e.a}}}\\[-1pt]\hline\noalign{\vskip 2pt}
\gtarget{16.96.2.16} & $\smallmat{7}{0}{4}{11}, \smallmat{1}{0}{0}{15}, \smallmat{15}{4}{12}{7}, \smallmat{9}{0}{12}{9}, \smallmat{9}{12}{0}{1}$ & \href{https://mathstats.uncg.edu/sites/pauli/congruence/csg2M.html#level16}{16I\textsuperscript{2}} & 14 & 2 & 0 & 2 & $2$\\[0.5pt]
& \multicolumn{7}{l}{\parbox[l]{12.5cm}{\raggedright\mflabel{64.2.e.a}}}\\[-1pt]\hline\noalign{\vskip 2pt}
\gtarget{16.96.2.17} & $\smallmat{3}{4}{0}{1}, \smallmat{3}{12}{4}{7}, \smallmat{15}{8}{4}{15}, \smallmat{9}{8}{12}{1}, \smallmat{13}{8}{0}{15}$ & \href{https://mathstats.uncg.edu/sites/pauli/congruence/csg2M.html#level16}{16I\textsuperscript{2}} & 14 & 2 & 0 & 2 & $2$\\[0.5pt]
& \multicolumn{7}{l}{\parbox[l]{12.5cm}{\raggedright\mflabel{128.2.e.a}}}\\[-1pt]\hline\noalign{\vskip 2pt}
\gtarget{16.96.2.18} & $\smallmat{7}{8}{12}{7}, \smallmat{7}{0}{4}{13}, \smallmat{9}{4}{0}{3}, \smallmat{1}{0}{4}{13}, \smallmat{1}{4}{8}{7}$ & \href{https://mathstats.uncg.edu/sites/pauli/congruence/csg2M.html#level16}{16I\textsuperscript{2}} & 14 & 2 & 0 & 2 & $2$\\
\end{tabular}
\end{center}
\end{table}
\begin{table}
\begin{center}\small
\begin{tabular}{lllccccl}
label & generators & CP & \hspace{-12pt}$\#X_H^\infty(\Qbar)$ & \hspace{-4pt}$\#X_H^\infty(\Q)$ & $r$ & $g$ & dims\\\toprule
\gtarget{16.96.3.343} & $\smallmat{5}{12}{2}{3}, \smallmat{0}{15}{5}{12}, \smallmat{11}{8}{14}{9}$ & \href{https://mathstats.uncg.edu/sites/pauli/congruence/csg3M.html#level16}{16S\textsuperscript{3}} & 8 & 4 & 1 & 3 & $1, 2$\\[0.5pt]
& \multicolumn{7}{l}{\parbox[l]{12.5cm}{\raggedright\mflabel{16.2.e.a}, \mflabel{128.2.a.a}}}\\[-1pt]\hline\noalign{\vskip 2pt}
\gtarget{16.96.3.346} & $\smallmat{4}{7}{13}{0}, \smallmat{10}{13}{15}{14}, \smallmat{5}{8}{14}{15}$ & \href{https://mathstats.uncg.edu/sites/pauli/congruence/csg3M.html#level16}{16S\textsuperscript{3}} & 8 & 0 & 1 & 3 & $1, 2$\\[0.5pt]
& \multicolumn{7}{l}{\parbox[l]{12.5cm}{\raggedright\mflabel{128.2.a.a}, \mflabel{128.2.e.b}}}\\[-1pt]\hline\noalign{\vskip 2pt}
\gtarget{32.48.1.1} & $\smallmat{29}{4}{0}{3}, \smallmat{31}{4}{0}{1}, \smallmat{29}{27}{0}{15}, \smallmat{7}{23}{0}{25}, \smallmat{21}{5}{0}{1}$ & \href{https://mathstats.uncg.edu/sites/pauli/congruence/csg1M.html#level32}{32A\textsuperscript{1}} & 8 & 4 & 0 & 1 & $1$\\[0.5pt]
& \multicolumn{7}{l}{\parbox[l]{12.5cm}{\raggedright\mflabel{32.2.a.a}}}\\[-1pt]\hline\noalign{\vskip 2pt}
\gtarget{32.48.1.2} & $\smallmat{3}{0}{5}{3}, \smallmat{13}{0}{13}{15}, \smallmat{17}{0}{7}{25}, \smallmat{19}{16}{20}{13}, \smallmat{29}{0}{16}{5}$ & \href{https://mathstats.uncg.edu/sites/pauli/congruence/csg1M.html#level32}{32A\textsuperscript{1}} & 8 & 2 & 0 & 1 & $1$\\
\end{tabular}
\end{center}
\end{table}
\begin{table}
\begin{center}\small
\begin{tabular}{lllccccl}
label & generators & CP & \hspace{-12pt}$\#X_H^\infty(\Qbar)$ & \hspace{-4pt}$\#X_H^\infty(\Q)$ & $r$ & $g$ & dims\\\toprule
\gtarget{32.96.2.37} & $\smallmat{13}{6}{0}{1}, \smallmat{29}{0}{16}{25}, \smallmat{9}{1}{0}{7}, \smallmat{21}{12}{0}{31}, \smallmat{15}{18}{0}{31}$ & \href{https://mathstats.uncg.edu/sites/pauli/congruence/csg2M.html#level32}{32C\textsuperscript{2}} & 14 & 2 & 0 & 2 & $2$\\[0.5pt]
& \multicolumn{7}{l}{\parbox[l]{12.5cm}{\raggedright\mflabel{64.2.e.a}}}\\[-1pt]\hline\noalign{\vskip 2pt}
\gtarget{32.96.2.38} & $\smallmat{19}{27}{16}{7}, \smallmat{21}{9}{16}{25}, \smallmat{7}{13}{0}{17}, \smallmat{3}{19}{16}{9}, \smallmat{23}{17}{16}{23}$ & \href{https://mathstats.uncg.edu/sites/pauli/congruence/csg2M.html#level32}{32C\textsuperscript{2}} & 14 & 6 & 0 & 2 & $2$\\[0.5pt]
& \multicolumn{7}{l}{\parbox[l]{12.5cm}{\raggedright\mflabel{16.2.e.a}}}\\[-1pt]\hline\noalign{\vskip 2pt}
\gtarget{32.96.2.63} & $\smallmat{19}{18}{30}{23}, \smallmat{0}{5}{23}{8}, \smallmat{1}{8}{0}{11}$ & \href{https://mathstats.uncg.edu/sites/pauli/congruence/csg2M.html#level32}{32B\textsuperscript{2}} & 14 & 0 & 0 & 2 & $2$\\[0.5pt]
& \multicolumn{7}{l}{\parbox[l]{12.5cm}{\raggedright\mflabel{1024.2.a.e}}}\\[-1pt]\hline\noalign{\vskip 2pt}
\gtarget{32.96.2.64} & $\smallmat{20}{23}{7}{12}, \smallmat{31}{18}{10}{1}, \smallmat{12}{5}{7}{12}$ & \href{https://mathstats.uncg.edu/sites/pauli/congruence/csg2M.html#level32}{32B\textsuperscript{2}} & 14 & 0 & 2 & 2 & $2$\\
\end{tabular}
\end{center}
\end{table}
\begin{table}
\begin{center}\small
\begin{tabular}{lllccccl}
label & generators & CP & \hspace{-12pt}$\#X_H^\infty(\Qbar)$ & \hspace{-4pt}$\#X_H^\infty(\Q)$ & $r$ & $g$ & dimensions\\\toprule
\gtarget{32.96.5.19} & $\smallmat{21}{12}{4}{1}, \smallmat{24}{31}{21}{0}, \smallmat{24}{5}{25}{8}, \smallmat{0}{21}{27}{8}$ & \href{https://mathstats.uncg.edu/sites/pauli/congruence/csg5M.html#level32}{32H\textsuperscript{5}} & 6 & 2 & 5 & 5 & $1, 4$\\[0.5pt]
& \multicolumn{7}{l}{\parbox[l]{12.5cm}{\raggedright\mflabel{256.2.a.a}, \mflabel{1024.2.a.g}}}\\[-1pt]\hline\noalign{\vskip 2pt}
\gtarget{32.96.5.20} & $\smallmat{15}{28}{20}{15}, \smallmat{21}{20}{20}{27}, \smallmat{28}{17}{9}{4}, \smallmat{0}{17}{27}{0}$ & \href{https://mathstats.uncg.edu/sites/pauli/congruence/csg5M.html#level32}{32H\textsuperscript{5}} & 6 & 2 & 1 & 5 & $1, 4$\\[0.5pt]
& \multicolumn{7}{l}{\parbox[l]{12.5cm}{\raggedright\mflabel{256.2.a.a}, \mflabel{1024.2.a.j}}}\\[-1pt]\hline\noalign{\vskip 2pt}
\gtarget{32.96.5.21} & $\smallmat{4}{13}{7}{28}, \smallmat{31}{0}{24}{21}, \smallmat{0}{3}{25}{0}, \smallmat{1}{28}{12}{31}$ & \href{https://mathstats.uncg.edu/sites/pauli/congruence/csg5M.html#level32}{32F\textsuperscript{5}} & 4 & 2 & 5 & 5 & $1, 4$\\[0.5pt]
& \multicolumn{7}{l}{\parbox[l]{12.5cm}{\raggedright\mflabel{256.2.a.b}, \mflabel{1024.2.a.h}}}\\[-1pt]\hline\noalign{\vskip 2pt}
\gtarget{32.96.5.22} & $\smallmat{9}{20}{4}{9}, \smallmat{20}{9}{7}{12}, \smallmat{21}{0}{8}{19}, \smallmat{24}{23}{27}{0}$ & \href{https://mathstats.uncg.edu/sites/pauli/congruence/csg5M.html#level32}{32F\textsuperscript{5}} & 4 & 2 & 1 & 5 & $1, 4$\\
\bottomrule
\end{tabular}
\bigskip

\caption{Arithmetically maximal subgroups of $2$-power level.}\label{table:groups2}
\end{center}
\end{table}

\begin{table}[bth!]
\begin{tabular}{lccll}
subgroup & 4-orbits & 8-orbits & $D$ & $E/\Q$\\\toprule
\glink{16.64.2.1} & $12$ & $48$ & $\minus 3$  & \eclabel{27.a3}\\
\glabel{16.96.4.13} & $4,8$ & $16,32$ & $\minus 4$  & \eclabel{288.a1}\\
\glabel{16.192.9.83} & $4,8$ & $16,32$ & $\minus 3$  & \eclabel{36.a3}\\
\glabel{16.192.9.143} & $4^3$ & $16^3$ & $\minus 4$  & \eclabel{288.d3}\\
\glabel{16.192.9.145} & $4,8$ & $8^2,32$ & $\minus 4$  & \eclabel{2304.a2}\\
\glabel{16.192.9.178} & $2^2,8$ & $8^2,32$ & $\minus 4$  & \eclabel{288.d4}\\
\glabel{16.192.9.189} & $4^3$ & $16^3$ & $\minus 4$  & \eclabel{2304.a1} \\
\glabel{16.384.9.462} & $2^2,4^2$ & $16,8^4$ & $\minus 4$  & \eclabel{64.a3}\\
\glabel{16.384.9.568} & $2^4,4$ & $16,8^4$ & $\minus 4$  & \eclabel{32.a3}\\
\glabel{16.384.9.600} & $2^2,4^2$ & $4^2,16^2,8$ & $\minus 4$  & \eclabel{64.a4} \\
\glabel{16.384.9.633} & $1^2,2,4^2$ & $4^2,16^2,8$ & $\minus 4$  & \eclabel{32.a}\\
\glabel{16.384.9.817} & $4,8$ & $16^2,8^2$ & $\minus 4$  & \eclabel{256.c2}\\
\glabel{16.384.9.819} & $4,8$ & $16^2,8^2$ & $\minus 4$  & \eclabel{256.b2}\\
\glabel{16.384.9.829} & $4^3$ & $16^2,8^2$ & $\minus 4$  & \eclabel{256.c1}\\
\glabel{16.384.9.831} & $4^3$ & $16^2,8^2$ & $\minus 4$  & \eclabel{256.b1}\\
\bottomrule
\end{tabular}
\bigskip

\caption{Twists $H$ of $\calN_\calO(16)$ for $\calO=\Z[i],\Z[\zeta_3]$ and $E/\Q$ with potential CM by $\calO$ and $\rho_{E,2^\infty}(G_\Q)=H$.}\label{table:stwists16}
\end{table}

\begin{table}
\vspace{60pt}
\begin{tabular}{lccll}
subgroup & 3-orbits & 9-orbits & $D$ & $E/\Q$\\\toprule
\glabel{27.324.18.1} & $2,6$ & $18,54$ & $\minus 3$  & \eclabel{144.a3}\\
\glabel{27.648.18.1} & $1^2,6$ & $9^2,54$ & $\minus 3$  & \eclabel{36.a4}\\
\glabel{27.648.18.4} & $2,3^2$ & $18,27^2$ & $\minus 3$  & \eclabel{36.a3}\\
\glabel{27.972.55.16} & $2^2,4$ & $6,12,18,36$ & $\minus 3$  & \eclabel{432.e3}\\
\glabel{27.972.55.22} & $2,6$ & $18^2,36$ & $\minus 3$  & \eclabel{3888.h1}\\
\glabel{27.972.55.23} & $2,6$ & $18^2,36$ & $\minus 3$  & \eclabel{3888.n1}\\
\glabel{27.1944.55.31} &$1^2,2^3$ &  $3^2,6^2,18^3$ & $\minus 3$  & \eclabel{27.a4}\\
\glabel{27.1944.55.37} & $1^2,3^2$ & $6^3,9^2,18^2$ & $\minus 3$  & \eclabel{27.a3}\\
\glabel{27.1944.55.43} & $1^2,6$ & $9^2,18^3$ & $\minus 3$  & \eclabel{243.b2}\\
\glabel{27.1944.55.44} & $1^2,6$ & $9^2,18^3$ & $\minus 3$  & \eclabel{243.a2}\\
\glabel{27.1944.55.49} & $2,3^2$ & $9^2,18^3$ & $\minus 3$  & \eclabel{243.b1}\\
\glabel{27.1944.55.50} & $2,3^2$ & $9^2,18^3$ & $\minus 3$  & \eclabel{243.a1}\\
\bottomrule
\end{tabular}
\bigskip

\caption{Twists $H$ of $\calN_\calO(27)$ for $\calO=\Z[\zeta_3]$ and $E/\Q$ with potential CM by $\calO$ and $\rho_{E,3^\infty}(G_\Q)=H$.}\label{table:stwists27}
\end{table}

\begin{table}
\begin{tabular}{lcclcrlc}
&&&&&\multicolumn{3}{c}{example}\\
\cmidrule{6-8}
subgroup & 4-orbits & 8-orbits & $D$-constraint && $D$ & $E/\Q(j(E))$ & $\Q(j(E))$\\\toprule
\glink{16.64.2.1} & $12$ & $48$ & $5 \bmod 8$ && $\minus 11$ & \eclabel{121.b2} & $\Q$\\
\glabel{16.96.3.323} & $2^2,8$ & $8^2,32$ & $16 \bmod 64$ && $\minus 48$ & \ecnflabel{2.2.12.1}{36.1-a3} & $\Q(\sqrt{3})$\\
$^*$\glabel{16.192.3.539} & $1^2,2,8$ & $4^4,32$ & && & $^*$\ecnflabel{2.2.12.1}{16.1-a4} &\\
$^*$\glabel{16.192.3.546} & $1^2,2,8$ & $4^2,8,32$ & && & $^*$\ecnflabel{2.2.12.1}{256.1-c4} &\\
$^*$\glabel{16.192.3.553} & $2^2,4^2$ & $4^2,16^2,8$ & && & $^*$\ecnflabel{2.2.12.1}{16.1-a3} &\\
$^*$\glabel{16.192.3.564} & $2^2,4^2$ & $8^2,16^2$ & && & $^*$\ecnflabel{2.2.12.1}{256.1-c3} &\\
\bottomrule
\end{tabular}
\bigskip

\caption{Subgroups $\calN_\calO(16)$ of determinant index 1 that arise for imaginary quadratic orders $\calO$ of discriminant $D$ and their quadratic twists (noted by asterisks), along with elliptic curves that realize them, identified by LMFDB label or an element $a\in\Q(j(E))$ to twist by, where $a^4-2a^3-13a^2+14a+19=0$.}\label{table:CM16a}
\end{table}

\begin{table}
\captionsetup{width=0.91\linewidth}
\begin{tabular}{lcclcrlc}
&&&&& \multicolumn{3}{c}{example}\\
\cmidrule{6-8}
subgroup & 4-orbits & 8-orbits & constraint & & $D$ & $E/\Q(j(E))$ & $\Q(j(E))$\\\toprule
\glabel{8.2.2-16.48.0.297} & $2^2,8$ & $4^2,8,32$ & $0 \bmod 64$ && $\minus  64$ & \ecnflabel{2.2.8.1}{1568.2-i1} & $\Q(\sqrt{2})$\\
${}^*$\glabel{8.2.2-16.96.0.280} & $1^2,2,8$ & $2^4,8,32$ & && & $^*$\ecnflabel{2.2.8.1}{32.1-a2} &\\
${}^*$\glabel{8.2.2-16.96.0.292} & $1^2,2,8$ & $4^4,32$ & && & $^*$\ecnflabel{2.2.8.1}{1024.1-k2} &\\
${}^*$\glabel{8.2.2-16.96.0.304} & $2^2,4^2$ & $4^4,16^2$ & && & $^*$\ecnflabel{2.2.8.1}{32.1-a1}&\\
${}^*$\glabel{8.2.2-16.96.0.316} & $2^2,4^2$ & $4^2,8,16^2$ & && & $^*$\ecnflabel{2.2.8.1}{1024.1-k1}&\\
${}^*$\glabel{8.2.2-16.96.0.372} & $2^2,8$ & $2^2,4,8,32$ & && & $^*$\ecnflabel{2.2.8.1}{1024.1-f1}&\\
\bottomrule
\end{tabular}
\bigskip

\caption{Subgroups $\calN_\calO(16)$ of determinant index $i>1$ for imaginary quadratic orders $\calO$ of discriminant $D$ and their quadratic twists (noted by asterisks), along with elliptic curves that realize them, identified by LMFDB label except for the curve $E_1:y^2 + axy + y = x^3 + \frac{a^2}{2}x^2 + \frac{48a^3 + 45a^2 - 250a - 246}{2}x + (144a^3 + 121a^2 - 752a - 638),$
where $a^4-6a^2+4=0$.
}\label{table:CM16b}
\end{table}

\begin{table}
\begin{tabular}{lcclcrlc}
&&&&& \multicolumn{3}{c}{example}\\
\cmidrule{6-8}
subgroup & 3-orbits & 9-orbits & $D$-constraint && $D$ & $E/\Q(j(E))$ & $\Q(j(E))$\\\toprule
\glink{27.243.12.1} & $8$ & $72$ & $2\bmod 3$ && $\minus 7$ & \eclabel{49.a4} & $\Q$\\
\glabel{27.324.13.16} & $2,6$ & $6,12,54$ & $0\bmod 27$ && $\minus 27$ & \eclabel{432.e2} & $\Q$\\
$^*$\glabel{27.648.13.25} & $1^2,6$ & $3^2,6^2,54$ & && & $^*$\eclabel{27.a2} & \\
$^*$\glabel{27.648.13.34} & $2,3^2$ & $6^3,27^2$ & && & $^*$\eclabel{27.a1} & \\
\glabel{27.324.16.26} & $2,6$ & $6,12,54$ & $18\bmod 27$ && $\minus 36$ & \ecnflabel{2.2.12.1}{144.1-b1} & $\Q(\sqrt{3})$\\
$^*$\glabel{27.648.16.35} & $1^2,6$ & $3^2,6^2,54$ & && & $^*$\ecnflabel{2.2.12.1}{9.1-a2}\\
$^*$\glabel{27.648.16.53} & $2,3^2$ & $6^3,27^2$ & && & $^*$\ecnflabel{2.2.12.1}{9.1-a1}\\
\glabel{27.324.16.29} & $2,6$ & $6,12,54$ & $9\bmod 27$ && $\minus 72$ & \ecnflabel{2.2.24.1}{16.1-a3} & $\Q(\sqrt{6})$\\
$^*$\glabel{27.648.16.38} & $1^2,6$ & $3^2,6^2,54$ & && & $^*$\ecnflabel{2.2.24.1}{1.1-a4}\\
$^*$\glabel{27.648.16.56} & $2,3^2$ & $6^3,27^2$ & && & $^*$\ecnflabel{2.2.24.1}{1.1-a3}\\
\glabel{27.324.18.1} & $2,6$ & $18,54$ & $6\bmod 9$ && $\minus 12$ & \eclabel{144.a1} & $\Q$\\
$^*$\glabel{27.648.18.1} & $1^2,6$ & $9^2,54$ & && & $^*$\eclabel{36.a2}\\
$^*$\glabel{27.648.18.4} & $2,3^2$ & $18,27^2$ & && & $^*$\eclabel{36.a1}\\
\glabel{27.324.19.10} & $2,6$ & $18,54$ & $3\bmod 9$ && $\minus 15$ & \ecnflabel{2.2.5.1}{2025.1-e1} & $\Q(\sqrt{5})$\\
$^*$\glabel{27.648.19.19} & $1^2,6$ & $9^2,54$ & && & $^*$\ecnflabel{2.2.5.1}{81.1-a1}\\
$^*$\glabel{27.648.19.22} & $2,3^2$ & $18,27^2$ & && & $^*$\ecnflabel{2.2.5.1}{81.1-a2}\\
\glabel{27.486.28.7} & $4^2$ & $12,24,36$ & $1\bmod 3$ && $\minus 11$ & \eclabel{121.b2} & $\Q$ \\\bottomrule
\end{tabular}
\bigskip

\caption{Subgroups $\calN_\calO(27)$ that arise for imaginary quadratic orders $\calO$ of discriminant $D$ and their quadratic twists (noted by asterisks), along with elliptic curves that realize them, identified by LMFDB label.}\label{table:CM27}
\end{table}

\FloatBarrier

\renewcommand{\eprint}[1]{\href{https://arxiv.org/abs/#1}{arXiv:#1}}

\def\polhk#1{\setbox0=\hbox{#1}{\ooalign{\hidewidth
  \lower1.5ex\hbox{`}\hidewidth\crcr\unhbox0}}} \def\cprime{$'$}
% \bib, bibdiv, biblist are defined by the amsrefs package.

\end{document}